\documentclass[preprint]{imsart}
\usepackage[utf8]{inputenc}
\usepackage{amsmath}
\usepackage{amssymb}
\usepackage{graphicx}
\RequirePackage{natbib} %
\usepackage{bbm}
\usepackage{amsthm}
\usepackage{xcolor}
\usepackage{tabularx,multicol,multirow,booktabs,csquotes,comment,mathtools}
\usepackage[margin=1.25in]{geometry}
\usepackage{verbatim}
\usepackage{hyperref}
\hypersetup{
	colorlinks,
	linkcolor={red!50!black},
	citecolor={blue},
	urlcolor={blue!80!black}
}

\RequirePackage{hypernat}
\usepackage{bm}
\usepackage[shortlabels]{enumitem}
\usepackage{subfigure}
\usepackage[noabbrev,capitalize]{cleveref}
\usepackage{appendix}

\allowdisplaybreaks

\newtheorem{thm}{Theorem}[section]
\newtheorem{cor}[thm]{Corollary}
\newtheorem{defi}{Definition}[section]
\newtheorem{assume}{Assumption}[section]
\newtheorem{prop}[thm]{Proposition}

\newtheorem{lemma}[thm]{Lemma}
\newtheorem{ex}{Example}[section]
\newtheorem{remark}{Remark}[section]

\numberwithin{equation}{section}

\DeclareMathOperator{\ct}{\tilde{\mathnormal c}}

\DeclareMathOperator{\Cov}{Cov}
\DeclareMathOperator{\Var}{Var}

\newcommand{\diag}{\text{\normalfont{Diag}}}

\newcommand{\pQ}{Q^{\mathrm{prod}}}
\newcommand{\R}{\mathbb{R}}
\newcommand{\I}{\mathbf{I}}

\DeclareMathOperator{\pt}{\mathnormal{p}/2}
\DeclareMathOperator{\an}{\alpha_{\mathnormal p}}

\DeclareMathOperator{\otb}{B_{slab}}

\DeclareMathOperator{\EE}{\mathbb{E}}
\DeclareMathOperator{\EEX}{\mathbb{E}_{\bx}}
\DeclareMathOperator{\PP}{\mathbb P}

\DeclareMathOperator*{\xp}{\xrightarrow{\mathnormal P}}
\DeclareMathOperator*{\xd}{\xrightarrow{\mathnormal d}}

\DeclareMathOperator*{\argmax}{arg\,max}
\DeclareMathOperator*{\argmin}{arg\,min}
\newcommand{\bgamma}{\bm{\gamma}}

\DeclareMathOperator{\Off}{\mathrm{Off}}

\DeclareMathOperator{\mcn}{N}

\newcommand{\A}{\mathbf A}

\newcommand{\bc}{\mathbf c}
\newcommand{\bq}{\mathbf q}
\newcommand{\bu}{\mathbf u}

\newcommand{\bv}{\mathbf v}
\newcommand{\bw}{\mathbf w}
\newcommand{\bx}{\mathbf X}
\newcommand{\by}{\mathbf y}

\newcommand{\C}{\mathbf C}
\newcommand{\M}{\mathbf M}
\newcommand{\bbst}{\bbeta^{\star}}

\newcommand{\Sigman}{{\boldsymbol{\Sigma}}_{\mathnormal p}}
\newcommand{\bbeta}{\mathnormal{\boldsymbol \beta}}

\newcommand{\boldc}{\mathnormal{\mathbf{c}}}
\newcommand{\sumin}{\sum_{\mathnormal i=1}^{\mathnormal p}}
\newcommand{\maxin}{\max_{\mathnormal i=1}^{\mathnormal p}}
\newcommand{\maxjn}{\max_{\mathnormal j=1}^{\mathnormal p}}
\newcommand{\sumjn}{\sum_{\mathnormal j=1}^{\mathnormal p}}
\newcommand{\sumij}{\sum_{\mathnormal{ i,j}=1}^{\mathnormal p}}
\newcommand{\sumik}{\sum_{\mathnormal{ i,k}=1}^{\mathnormal p}}
\newcommand{\sumkn}{\sum_{\mathnormal k=1}^{\mathnormal p}}
\newcommand{\vep}{\varepsilon}
\newcommand{\EEB}{\mathbb{E}_{\bbeta^\star}}
\newcommand{\OPX}{o_{P,\bx}}
\newcommand{\BOPX}{O_{P,\bx}}
\newcommand{\OPB}{o_{P,\bbeta^\star}}
\newcommand{\BOPB}{O_{P,\bbeta^\star}}

\usepackage{tikz}
\usetikzlibrary{shapes, calc, shapes, arrows, positioning, bayesnet}
\usetikzlibrary{decorations.pathreplacing}

\tikzstyle{qedge}=[->,thick,black]
\tikzstyle{pre}=[->,thick,dotted]
\tikzstyle{pres}=[-,dotted]

\definecolor{myblue}{rgb}{0.0265,    0.6137,    0.8135}
\definecolor{myyellow}{rgb}{0.9290,    0.6940,    0.1250}

\usepackage{bbm}
\tikzstyle{neuron}=[draw, circle,minimum size=25pt,inner sep=0pt, fill=black!20]

\tikzstyle{hidden}=[draw, circle,minimum size=25pt,inner sep=0pt, fill=white]

\tikzstyle{hiddens}=[draw,circle,minimum size=17pt,inner sep=0pt, fill=white]

\usetikzlibrary{arrows.meta}
\tikzset{>={Latex[width=2mm,length=2mm]}}

\tikzstyle{arr}=[->, thick, black]

\usetikzlibrary{decorations.text}

\begin{document}
	\begin{frontmatter}
		\title{CLT in high-dimensional Bayesian linear regression with low SNR}
		\runtitle{CLT in high-dimensional Bayesian linear regression}
        \runauthor{Lee et al.}

		 \begin{aug}
			\author{\fnms{Seunghyun} \snm{Lee}\ead[label=e1]{ sl4963@columbia.edu}},
			\author{\fnms{Nabarun} \snm{Deb}\ead[label=e2]{nabarun.deb@chicagobooth.edu}},
		 	\and
		 	\author{\fnms{Sumit} \snm{Mukherjee}\ead[label=e3]{sm3949@columbia.edu}}
		 \end{aug}
        \address{Department of Statistics, Columbia University\printead[presep={,\ }]{e1,e3}}
        \address{Chicago Booth School of Business\printead[presep={,\ }]{e2}}
  
		\begin{abstract}
  We study central limit theorems for linear statistics in high-dimensional Bayesian linear regression with product priors. Unlike the existing literature where the focus is on posterior contraction, we work under a non-contracting regime where neither the likelihood nor the prior dominates the other. This is motivated by modern high-dimensional datasets characterized by a bounded signal-to-noise ratio. This work takes a first step towards understanding limit distributions for  one-dimensional projections of the posterior, as well as the posterior mean, in such regimes. Analogous to 
  contractive settings, the resulting limiting distributions are Gaussian, but they heavily depend on the chosen prior and center around the Mean-Field approximation of the posterior. 
  We study two concrete models of interest to illustrate this phenomenon --- the white noise design, and the (misspecified) Bayesian  model. As an application, we construct credible intervals and compute their coverage probability under any misspecified prior. Our proofs rely on a combination of recent developments in Berry-Esseen type bounds for Random Field Ising models and both first and second order Poincar\'{e} inequalities. Notably, our results do not require any sparsity assumptions on the prior.
  \end{abstract}


		
	\end{frontmatter}
	
	\maketitle
	
	\section{Introduction}
We consider the linear regression model:
\begin{equation}\label{eq:model}  
    \mathbf{y} = \mathbf{X}\boldsymbol{\beta} + \boldsymbol{\vep},
\end{equation}
where $\mathbf{y} \in \mathbb{R}^n$ is the response vector, $\mathbf{X}$ is an $n \times p$ design matrix, $\boldsymbol{\beta} \in \mathbb{R}^p$ is the vector of regression coefficients, and $\boldsymbol{\vep} \sim \mcn(\mathbf{0}_n, \sigma^2 \mathbf{I}_n)$ represents unobserved Gaussian noise with known variance $\sigma^2 > 0$. Within a Bayesian framework, we further assume a prior distribution on the coefficient vector $\boldsymbol{\beta}$. A natural starting point is to assume that
\begin{equation}\label{eq:Bayeslin}
    \beta_1, \ldots, \beta_p \overset{\text{i.i.d}}{\sim} \mu,
\end{equation}
where $\mu$ is the prior distribution for each component of $\boldsymbol{\beta}$. A popular  application of the above model is for estimating effect size distributions or genetic architectures of complex physiological characteristics; see e.g.,~\cite{zhou2021fast,zhang2018estimation,O'Hagan2013,morgante2023flexible}.  %
Under this Bayesian model specification, the posterior distribution $\nu_{\mathbf{y},\mathbf{X}}(\bbeta) = \nu_{\mathbf{y},\mathbf{X}}(\bbeta \mid \by, \bx)$ is given by
    \begin{align}\label{eq:posterior_intro}
        \frac{d \nu_{\mathbf{y},\mathbf{X}}}{d \mu^{\otimes p}}(\boldsymbol{\beta}) \propto &\exp\left(-\frac{\|\by - \bx \bbeta\|^2}{2 \sigma^2} \right)  \propto \exp\left( - \frac{1}{2} \bbeta^\top \Sigman \bbeta + \Big(\frac{\bx^\top \by}{\sigma^2}\Big)^\top \bbeta \right).
    \end{align}
    Here, the $p \times p$ matrix $\Sigman$ is defined as $\Sigman := \sigma^{-2} (\bx^\top \bx)$ and it encodes the dependence structure of the posterior.
Without strong assumptions, such as assuming $\mu$ is Gaussian or the design is exactly orthogonal, the above posterior involves an intractable normalizing constant. Consequently, performing statistical inference under $\nu_{\by, \bx}$ is a challenging problem, especially for high-dimensional posteriors with $p \to \infty$. We note that $\nu_{\by,\bx}\equiv \nu_{\by,\bx,\mu}$ also depends on the prior $\mu$, but we suppress this dependence for notational convenience.

Bayesian inference methods for such high-dimensional linear models have been applied across various scientific disciplines, including genomics \citep{Guan2011}, marketing \citep{lenk1990new}, and finance \citep{rachev2008bayesian}, among others \citep{O'Hagan2013, george1997approaches}. In a broad sense, existing methods can be divided into two directions: sampling and variational inference.
The first relies on Markov Chain Monte Carlo (MCMC) methods \citep{metropolis1953equation, Hastings1970, george1993variable, Ishwaran2005} to sample from the posterior. While supporting theory guarantees asymptotically exact samples from the posterior, this becomes computationally intensive as $p$ increases. In contrast, variational inference avoids sampling from the posterior by approximating it with a more tractable family of distributions \citep{wainwright2008graphical}. Despite being computationally attractive, the theoretical understanding of variational inference in high-dimensions has been relatively limited. Furthermore, empirical studies often report problems such as variance underestimation when the variational posterior is used for uncertainty quantification \citep{minka2005divergence, zhang2018advances,blei2017variational,margossian2023shrinkage}.

In terms of theoretical studies for high-dimensional Bayesian linear models, most existing studies address the classical high signal-to-noise ratio (SNR) setting, where the operator norm of the quadratic interaction $\Sigman$ in \eqref{eq:posterior_intro} diverges, and consequently the data likelihood dominates the prior asymptotically. In such scenarios, the Bernstein–von Mises (BvM) theorem typically applies \citep{van2000asymptotic,ghosal1999asymptotic,Castillo2015, Katsevich2023}, and the posterior distribution contracts around the ordinary least squares estimator. In contrast, modern high-dimensional datasets often exhibit bounded signal strength, characterized by a bounded operator norm of $\Sigman$. This \emph{bounded SNR regime} implies that the prior maintains a non-vanishing influence on the posterior, rendering the BvM contraction around the least squares estimator inapplicable and introducing novel technical challenges. The bounded SNR regime has garnered significant attention in recent literature for Bayesian statistics \citep{mukherjee2022variational, barbier2020mutual, qiu2024sub, mukherjee2024naive, celentano2023mean} as well as high-dimensional sampling \citep{montanari2024provably,chewi2025book}. It also arises naturally in economics when dealing with partially identified models where point estimation is no longer possible \citep[see][]{tamer2010partial,moon2012bayesian}. 

To motivate our scaling choice for the bounded SNR regime (see \cref{assn:ht} and \cref{rem:bdsnr} below), let us focus on the ultra simplistic case $p=1$ and compare two simple models:
$$ \textrm{Model 1 : }Y_1,\ldots ,Y_n \overset{i.i.d.}{\sim} N(\theta,1) \quad \mbox{vs} \quad \textrm{Model 2 : }Y_1,\ldots ,Y_n\overset{i.i.d.}{\sim} N(h/\sqrt{n},1),$$ for parameters $\theta,h\in\R$. Here both models have the same noise/variance level. In Model 1, the frequentist MLE
$\bar{Y}_n=n^{-1}\sum_{i=1}^n Y_i$ is consistent for $\theta$, whereas in Model 2 the MLE $\sqrt{n}\bar{Y}$ is not consistent for $h$. In fact, there are no consistent estimators for $h$ in Model 2. From a Bayesian perspective, suppose we put priors on $\theta\sim\mu_1$ and $h\sim\mu_2$. In Model 1, the posterior is close in total variation to $N(\bar{Y}_n,1/n)$ with high probability, by the classical Bernstein-von Mises theorem. Thus the posterior contracts around the MLE $\bar{Y}$ at rate $n^{-1/2}$. Also, the approximate posterior distribution is free of the prior $\mu_1$. On the other hand,  the posterior in Model 2 is proportional to $$ \mu_2(h) \exp\left(-\frac{1}{2}\sum_{i=1}^n (Y_i-n^{-1/2}h)^2\right)\propto\mu_2(h)\exp(h\sqrt{n}\bar{Y}_n-h^2/2).$$
Here the approximate posterior is a linear and quadratic tilt of the prior, which does not contract to the MLE, or any other point. In fact, the posterior is heavily influenced by the prior $\mu_2$, even asymptotically. Overall, estimation and inference is harder in the context of Model 2, which makes it a more reasonable model for weak signals. Theoretically, the crucial difference between the above models is that, in view of \eqref{eq:model}, Model 1 has a design matrix ($n\times 1$) $\bx=\mathbf{1}_n$ whereas Model 2 has $\bx=n^{-1/2}\mathbf{1}_n$. So $\lVert \bx\rVert=\sqrt{n}\to\infty$ for Model 1 (hence high SNR) whereas $\lVert \bx\rVert=O(1)$ for Model 2 (hence low/bounded SNR). This motivates us to study the posterior \eqref{eq:posterior_intro} under the bounded SNR condition $\lVert \bx\rVert=O(1)$ in the high-dimensional setting. 

The existing line of work on statistical guarantees for bounded SNR Bayesian linear regression \citep[see][]{fan2023gradient,mukherjee2022variational,mukherjee2023mean,fan2025dynamical} focuses on characterizing Bayes optimal risk or average coverage guarantees under the high-dimensional posterior in \eqref{eq:posterior_intro}. By taking a one-dimensional projection of the posterior, our paper takes a step beyond such consistency-type results by understanding \emph{exact limiting distributions} with an explicit characterization of the mean and variance. To the best of our knowledge, ours is the first paper to derive limit distribution theory in such non-contracting regimes. Our results allow precise posterior uncertainty quantification such as constructing credible intervals with explicit expressions (as opposed to that based on quantiles of posterior samples), and evaluating its coverage probability. Additionally, by comparing the true posterior variance with the variance under the mean-field variational approximation, we precisely characterize when the variance is underestimated. Surprisingly, in contrast to the folklore, the mean-field posterior can overestimate the variance in  certain scenarios (see \cref{rem:varoverunder} for details).

\subsection{Main Contributions}\label{sec:contrib}

The primary objective of this paper is to derive limiting distributions of the following:
\begin{enumerate}
    \item[(i)] the parameter $\bq^{\top}\bbeta$ under the posterior distribution \eqref{eq:posterior} for some $\bq\in\R^p$, given $\mathbf{y}$,
    \item[(ii)] the posterior mean $\EE_{\nu_{\mathbf{y},\mathbf{X}}}[\bq^{\top}\bbeta]$ under additional data generating assumptions for $\by$.
\end{enumerate}
Here, $\bq$ is a deterministic vector normalized to have length one, i.e. $\|\bq\|=1$. The linear statistic $\bq^\top \bbeta$ is the mean response of a new observation with feature vector $\bq$. Assuming that $\bbeta \sim \nu_{\by,\bx}$ is a posterior sample, $\bq^\top \bbeta$ can also be understood as the one-dimensional projection of the posterior. The posterior mean $\EE_{\nu_{\by,\bx}}[\mathbf{q}^{\top}\bbeta]$ is the Bayes optimal point estimator for $\bq^\top \bbeta$ under a squared error loss (see \eqref{eq:Bayesopt} below). Therefore, these quantities have been a central focus in linear regression studies 
\citep{zhu2018linear, azriel2020non-sparse, zhao2023estimation, Bellec2021, chang2023inference}. More generally, understanding limiting distributions of the posterior as well as Bayes estimators is a classical topic in Bayesian asymptotics \citep{bickel1969some,le1953some, lehmann2006theory}. 

Our main results in \cref{sec:bayeslinhigh} address the above goals (i) and (ii), under popular assumptions on the design matrix and the regression coefficients. To overcome the absence of posterior contraction, we assume that the projection direction $\bq$ is ``delocalized'' in the sense that $\|\bq\|_\infty \to 0$, and prove a \emph{CLT} for $\bq^\top \bbeta = \sumin q_i \beta_i$. Note that in our asymptotic regime, CLT does not hold without this delocalization assumption, similar to CLTs for linear statistics of independent random variables, where this condition is referred to as the uniform asymptotic negligibility (UAN) condition.
    To the best of our knowledge, we provide the first Bernstein-von Mises type limiting distributions for \emph{non-contracting posteriors}. Unlike classical Bernstein-von Mises results, our limiting distributions are non-universal in that they heavily depend on the chosen prior $\mu$ (see \eqref{eq:Bayeslin}). In \cref{sec:Bvmanalog}, we use these results to build asymptotic credible posterior intervals as well as quantify their uncertainty, under potential model misspecification. As a by-product of our results, we establish conditions under which the asymptotic behavior of $\bq^{\top}\bbeta$ under the Mean-Field variational approximation matches that under the posterior $\nu_{\by,\bx}$. We provide a more detailed description of our contributions below.
    \begin{itemize} 
 \item In \cref{prop:regret}, we show that the natural point estimator of $\bq^{\top}\bbeta$ under the Mean-Field approximation of the posterior is in fact, asymptotically Bayes optimal. We also provide a finite sample regret bound comparing the risk of the Mean-Field estimator to the Bayes optimal risk. 
 \item In \cref{cor:regression i.i.d design}, we prove CLTs for $\bq^\top \bbeta$ and $\EE_{\nu_{\mathbf{y},\mathbf{X}}}[\bq^{\top}\bbeta]$ under white noise design when $p\ll n^{2/3}$. Our result accommodates both cases where the true regression coefficients are deterministic or random. In this case, the limit laws under the Mean-Field approximation and the posterior match exactly. As an application, when the coefficients are sampled from the popular spike-and-slab prior, we characterize the precise threshold on sparsity up to which the limit law matches that under a global null (see \cref{cor:sparse}).
 \item In \cref{cor:regression Bayesian truth}, we prove CLTs with a general design when the regression coefficient vector ${\bm \beta}$ are sampled i.i.d from some distribution, say $\mu^{\star}$ (which may not be equal to $\mu$). In this case, our results apply to both fixed and random design matrices, including the white noise (random) design, and the analysis of variance (ANOVA) design (see \cref{ex:anova design}). In the ANOVA design, our result highlights that depending on $\bq$ and the design $\bx$, the variance under the Mean-Field approximation can be larger, equal, or smaller than the posterior variance (see \cref{rem:varoverunder} for details). 
 \item In \cref{sec:Bvmanalog}, we propose a $100(1-\alpha)\%$-credible interval for $\bq^{\top}\bbeta$ under the posterior, based on \cref{cor:regression Bayesian truth}. In \cref{cor:coverage probability Bayes}, we establish a limiting formula for the coverage probability of the proposed credible interval  under any  \emph{misspecified prior}. The coverage probability matches the target level $1-\alpha$ when the prior is correctly specified. This provides a Bernstein-von Mises type coverage guarantee, under potential model misspecification, for non-contracting posteriors.
 \end{itemize}
 
Our results apply to any compactly supported prior $\mu$, and complement the limitations of the Laplace approximation \citep[see][]{Spokoiny2023,katsevich2024laplace,giordano2025good} as a technique for approximating the posterior. This approximation is typically useful when the SNR is high and the effect of the prior $\mu$ washes away in the asymptotic limit. Moreover it requires the posterior to have density with respect to the Lebesgue measure, and does not allow priors that incorporate exact sparsity (i.e. positive mass at zero). Our results do not impose such restrictions, and are readily able to generalize to fractional posteriors (see \cref{sec:fractional posterior}).

\subsection{Notation}
	For two measures $Q_1,Q_2$ on the same probability space, define the Kullback-Leibler divergence between $Q_1$ and $Q_2$ as
 $$ \mathrm{KL}(Q_1|Q_2) := \begin{cases} \EE_{Q_1}\log{\frac{dQ_1}{dQ_2}} & \mbox{if } Q_1 \ll Q_2, \\  \infty & \mbox{otherwise}.\end{cases} $$
 \noindent For any $p\times p$ real symmetric matrix $\mathbf{M}_p$, let $\lVert \mathbf{M}_p\rVert$, $\lVert \mathbf{M}_p\rVert_{r}$ denote the $(2,2)$-operator norm, the $(r,r)$-operator norm for $r\in (2,\infty]$, respectively. Let $\mathrm{Off}(\mathbf{M}_p)$ denote the off-diagonal operator that sets the diagonal entries in $\mathbf{M}_p$ to 0.
    For a vector $\mathbf{a}=(a_1,\ldots ,a_p)\in \R^p$, let $\lVert \mathbf{a}\rVert$ and $\lVert \mathbf{a}\rVert_{r}$ denote the  $\ell_2$ and  $\ell_r$-vector norms, for $r\in (2,\infty]$. Let $\mbox{Diag}(\mathbf{a})$ denote the $p\times p$ diagonal matrix with diagonal entries $a_1,\ldots ,a_p$. 
    For nonnegative sequences $\{a_p\}_{p\ge 1}$ and $\{b_p\}_{p\ge 1}$, we write $a_p \lesssim b_p$, if there exists a constant $C>0$ free of $p$ (and $n$) such that $a_p\le C b_p$. We use $a_p\ll b_p$ if $a_p/b_p\to 0$ as $p\to\infty$. The notation $\overset{w}{\longrightarrow}, \overset{d}{\longrightarrow}, \overset{P}{\longrightarrow}$ will denote weak convergence, convergence in distribution, and convergence in probability, respectively. $\delta_{x}$ denotes the point mass at some $x\in\R$.
   We denote the $p$ dimensional all-zero vector, all-one vector, and $p\times p$ identity matrix by $\mathbf{0}_p$, $\mathbf{1}_p$, and $\I_p$  respectively. Finally $\mathcal{C}^2_b$ denotes the space of functions from $\R\to [-1,1]$ with uniformly bounded first and second derivatives.

   Throughout the paper, we work on an asymptotic setting where both $p$ and $n$ diverge to infinity. We will always view $n \equiv n(p)$ as a function of $p$ which satisfies $n(p)\to \infty$ as $p \to \infty$. All limits will therefore be written in terms of $p$.

\section{Setup and preliminary results}\label{sec:mainres}

This section formulates the problem setting as well as necessary concepts and notations. Throughout, we will assume that the prior $\mu$ in the Bayesian model formulation \eqref{eq:Bayeslin} is compactly supported and nondegenerate. This is a common assumption in the literature \citep{mukherjee2022variational, celentano2023mean, mukherjee2023mean, fan2023gradient}, imposed for various technical reasons such as ensuring that the normalizing constant in the posterior \eqref{eq:posterior_intro} is well-defined. For the sake of definiteness, in this work we assume that the compact support of $\mu$ is the interval $[-1, 1]$.

Next, we further simplify the posterior $\nu_{\mathbf{y},\mathbf{X}}(\bbeta)$ (see \eqref{eq:posterior_intro}) by separating out the diagonal elements of the $p \times p$ quadratic interaction matrix $\Sigman$:
\begin{align}\label{eq:anbayes}
    \Sigman = \text{Diag}(\mathbf{d}^{(p)}) - \A_p\quad \text{ where }\quad \A_p := -\Off(\Sigman), \quad \mathbf{d}^{(p)} = (d_1, d_2, \ldots , d_p).
\end{align}
Absorbing the diagonal elements $d_i:= \sigma^{-2} (\mathbf{X}^{\top}\mathbf{X})(i,i)$ into the base measure $\mu$ has been useful for understanding similar quadratic interaction models \citep{mukherjee2022variational, lee2025rfim}. Under these notations, we can re-write the posterior $\nu_{\mathbf{y},\mathbf{X}}(\bbeta)$ as follows:
\begin{align}
        \nu_{\mathbf{y},\mathbf{X}}(d\boldsymbol{\beta}) \propto & \exp\left( \frac{1}{2} \boldsymbol{\beta}^{\top} \A_p \boldsymbol{\beta}+\mathbf{c}^{\top}\boldsymbol{\beta}\right)\prod_{i=1}^p \mu_i(d\beta_i), \quad \mbox{where} \label{eq:posterior} \\
        \mathbf{c} &:=\frac{\mathbf{X}^{\top}\mathbf{y}}{\sigma^2}, \;\; \;\;\frac{d\mu_i}{d\mu}(\beta_i) \propto \exp\left(-\frac{d_i\beta_i^2}{2}\right), \label{eq:posterior_2}
    \end{align}
for $i = 1, 2, \ldots, p$. Here, we caution that the sign of the coupling matrix $\A_p$ has the opposite sign compared to the related works \citep{mukherjee2022variational, mukherjee2023mean}. %
Throughout this section (except for \cref{subsec:dgp}), $(\by, \bx)$ is considered deterministic, and the randomness only arises from the posterior. %

\subsection{Design matrix assumptions}\label{sec:assumptions}

We make some assumptions on the $n \times p$ design matrix $\bx$. Using the decomposition \eqref{eq:anbayes}, we separate the assumptions for the off-diagonal and diagonal parts. These are satisfied for a various examples of interest, covering both deterministic and random designs as illustrated at the end of the section.

Our first assumption is the so-called ``high-temperature'' assumption in the literature \citep{Reza2018, maillard2019high, lee2025rfim}. This condition has been often imposed to guarantee uniqueness of the Mean-Field approximation of the posterior $\nu_{\by,\bx}$ (see \cref{lem:uniqueness of optimizers} for a precise statement).
	
	\begin{assume}[high-temperature]\label{assn:ht}
        For some constant $\rho\in (0,1)$, assume that              \begin{align}\label{eq:mht}
				\lVert \A_p\rVert_{4}\le \rho.
			\end{align}
	\end{assume}
\noindent  Noting that $\|\A_p\|_2 \le \|\A_p\|_4 \le \|\A_p\|_\infty$ (e.g. see Lemma 2.1 in \cite{lee2025rfim}), this condition can sometimes be verified by computing the more tractable $(\infty,\infty)$ norm, given by $\|\A_p\|_\infty=\max_{1\le i\le p}\sum_{j=1}^p|\A_p(i,j)|$. %
However, it is not always advisable to use the $(\infty,\infty)$ norm as a bound, as it can give a sub-optimal result. As an illustration, for the special case when each entry of $\bx$ is i.i.d. sub-Gaussians (see \cref{ex:iid gaussian design}), we can explicitly compute the different norms (see \cref{lem:L_p norm}) as: 
$$\|\A_p\|_2 = O_{P,\bx}(\sqrt{p}/\sqrt{n}), \quad \|\A_p\|_4 = O_{P,\bx}(p^{3/4}/\sqrt{n}), \quad \|\A_p\|_\infty = O_{P,\bx}(p/\sqrt{n}).$$
Thus, using the $(4,4)$ norm the high-temperature condition holds as soon as $p\ll n^{2/3}$, whereas using the $(\infty,\infty)$ norm we require $p\ll \sqrt{n}$. This illustrates the benefits of the $(4,4)$ norm as opposed to the more commonly used $(\infty, \infty)$ norm in this example, since the resulting high-temperature condition allows a wider range of $p$.

Our second assumption is a ``Strong Mean-Field'' assumption, which guarantees the approximation accuracy of the Mean-Field approximation (see \cref{lem:mf accuracy} for a precise statement). 

	\begin{assume}[Strong Mean-Field regime (SMF)]\label{assn:mf}
     With $\A_p$ as in \eqref{eq:model}, define 
    \begin{align}\label{eq:rowcontrol}
        \alpha_p := \max_{1\leq i\leq p} \sum_{j=1}^p \A_p(i,j)^2,
    \end{align}
    and assume that $\alpha_p = o(p^{-1/2}).$
	\end{assume}
    We refer to the requirement $\alpha_p=o(p^{-1/2})$ as a \eqref{eq:rowcontrol} as a Strong Mean-Field assumption, as it implies the condition
    ${\rm Tr}(\A_p^2)=o(p)$, which has been used in the literature as a Mean-Field assumption for deriving Law of Large numbers (see Def. 1.3 in \cite{basak2017universality}, eq. (1.5) in \cite{lacker2024mean}, or Thm. 3 in \cite{mukherjee2023mean}).
    \\

Our final assumption involves $\mathbf{d}^{(p)}$, the diagonal elements of $\Sigman$. This assumption implies that all column vectors have similar length close to $d_0$, and can be understood as a homogeneity assumption on $\bx$. This assumption will not be imposed on most results in this section, but will be required for the CLTs in \cref{sec:bayeslinhigh}.

\begin{assume}[homogeneous design]\label{assn:homogeneous diagonals}
  Suppose there exists fixed constants $d_{\text{min}}, d_0 \in (0,\infty)$ such that $d_i \ge d_{\text{min}}$ for all $i$ and
  $$\sumin (d_i-d_0)^2=O(1).$$
\end{assume}

\cref{assn:homogeneous diagonals} essentially assumes that $\{\mu_i\}_{1\le i\le p}$ (see \eqref{eq:posterior}) are all close to some common measure $\mu_0$. This will be the case, for example, when the columns of $\bx$ are all centered to have mean $0$ and scaled to have variance $1$. In that case, one can choose $d_0=\sigma^{-2}$. 

\vspace{2mm}
\noindent \emph{Example design matrices.}\quad
Our assumptions can accommodate both deterministic and random design matrices $\bx$ that commonly appear in statistics.

\begin{ex}[Gaussian sequence model]\label{ex:Gausseq}
    As a toy example, suppose a statistician fits the Gaussian sequence model $y_i=\beta_i+\epsilon_i$, which corresponds to taking $n=p$ and $\bx=\mathbf{I}_p$ in \eqref{eq:Bayeslin}. In this case the posterior \eqref{eq:posterior} simplifies to 
    $$\nu_{\by,\bx}(d\bbeta)\propto \exp\left(-\frac{1}{\sigma^2}\left(\langle \by,\bbeta\rangle + \frac{1}{2}\sumin \beta_i^2\right)\right)\prod_{i=1}^p \mu(d\beta_i).$$
    Therefore, conditioned on the data $\by$, the posterior is a product measure which is trivially a product distribution, and also non-contracting for fixed (non-sparse) $\mu$. In this setting, $\A_p = 0$ and Assumptions \ref{assn:ht} and \ref{assn:mf} hold. \cref{assn:homogeneous diagonals} holds for any $\sigma^2>0$ for the choice $d_0=\sigma^{-2}$.  
\end{ex}   
  
\begin{ex}[Deterministic design]\label{ex:anova design}
       Next, we consider the popular two-way ANOVA model with $p/2$ levels (assuming $p$ is even) and no repeated measurements:
        \begin{equation*}
            y_{i,j} = \frac{1}{\sqrt{p}} (\delta_i + \gamma_j) + \epsilon_{i,j}, \quad 1 \le i,j \le p/2
        \end{equation*}
      Here, $(\delta_1, \ldots, \delta_{p/2}) \in [-1, 1]^{p/2}$ and $(\gamma_1, \ldots, \gamma_{p/2}) \in [-1, 1]^{p/2}$ denote the levels of the two factors, respectively. The errors $\epsilon_{i,j}$ are i.i.d $\mcn(0, \sigma^2)$. We define $n := (p/2)^2$, and let $\by \in \mathbb{R}^n$ denote the vectorized expression of the matrix $(y_{i,j})$. Similarly, let $\bbeta = (\delta_1, \ldots, \delta_{p/2}, \gamma_1, \ldots, \gamma_{p/2})^\top \in [-1,1]^p$ denote the coefficient vector, and let $\bx$ denote the corresponding $n \times p$ design matrix. 
      Then, it follows that $$\A_p = - \frac{1}{p \sigma^2} \begin{pmatrix}
          \mathbf{0}_{\pt \times \pt} & \mathbf{1}_{\pt \times \pt} \\
          \mathbf{1}_{\pt \times \pt} & \mathbf{0}_{\pt \times \pt}
      \end{pmatrix}$$ and $d_i = 1/(2\sigma^2)$ for all $i$. Consequently, \cref{assn:ht} holds when $\sigma^2 > \frac{1}{2}$, \cref{assn:mf} always holds, and \cref{assn:homogeneous diagonals} holds with $d_0 = 1/(2\sigma^2)$. 
\end{ex}

\begin{ex}[White noise design]\label{ex:iid gaussian design}

      We consider a random design with iid sub-Gaussian rows, i.e. suppose that $X_{k,i} \overset{i.i.d}{\sim} \frac{1}{\sqrt{n}} F$, $1\le k\le n$, $1\le i\le p$, where $F$ is a sub-Gaussian distribution (see \cite[Section 2.5]{vershynin2018high}) with mean zero and variance one. 
      Then in the asymptotic regime $p \ll n^{2/3}$, all three assumptions hold by taking $d_0 = \sigma^{-2}$. Checking the first two assumptions are somewhat non-trivial, but follows from Lemmas 4.1 and 4.2 in \cite{lee2025rfim}.
  \end{ex}

\begin{remark}[Bounded signal-to-noise ratio]\label{rem:bdsnr}
Under Assumptions \ref{assn:ht} and \ref{assn:homogeneous diagonals}, $\|\Sigman\|$ is bounded by a constant. This puts us in a \emph{bounded SNR} regime where the posterior distribution $\nu_{\by,\bx}$ \emph{does not} contract around the least squares estimator $\hat{\bbeta}^{\text{LSE}}:=(\mathbf{X}^{\top}\mathbf{X})^{-1}\mathbf{X}^{\top}\mathbf{y}$ or any ``true parameter'' (see e.g., \cite[Corollary 4]{mukherjee2022variational}). This is in sharp contrast to the more classical setting considered in \cite{ghosal2000asymptotic,Katsevich2023}, where the eigenvalues of $\Sigman$ are assumed to be large,  the posterior contracts around $\hat{\bbeta}^{\text{LSE}}$, and the limiting distribution does not depend on the prior $\mu$ (i.e.~Bernstein-von Mises type theorems hold).
\end{remark}

In the absence of posterior contraction, it is unclear how to provide a tractable centering (as opposed to the posterior mean, which is hard to compute) for the statistic of interest $\bq^{\top}\bbeta$, to get analogous results. %
Below we show that a natural centering based on the \emph{Mean-Field} approximation to the posterior, serves the purpose.

\subsection{Mean-Field variational approximation}
The Mean-Field approximation of the posterior $\nu_{\by,\bx}$ arises by approximating $\nu_{\by,\bx}$ with a product distribution $Q = \prod_{i=1}^p Q_i$ on $[-1,1]^p$ \citep[see][]{bishop2006pattern,blei2017variational}. In the existing literature, point estimation under the Mean-Field approximation is known to exhibit desirable properties; see  \cite{Yang2020,wang2018frequentist,zhang2020convergence,bhattacharya2023gibbs,bhattacharya2024ldp}. Commonly, the ``best'' approximation $ \pQ$ is defined as the product distribution that minimizes the KL divergence between $Q$ and $\nu_{\by,\bx}$, obtained by solving the optimization:
\begin{align}\label{eq:mean-field}
    \argmin_{Q=\prod_{i=1}^p Q_i} \mathrm{KL}\big(Q| \nu_{\by,\bx} \big) = \argmax_{Q=\prod_{i=1}^p Q_i}\left(\EE_Q\left[\frac{1}{2}\bbeta^{\top}\A_p\bbeta+\boldc^{\top}\bbeta\right]-\mathrm{KL}\big(Q|\prod_{i=1}^p \mu_i\big)\right).
\end{align}
Here, the expression being maximized on the last term, up to additive constants, is often denoted as the Evidence Lower Bound (ELBO) (see eq. (13) in \cite{blei2017variational} or eq. (7) in \cite{mukherjee2023mean}). The last equality follows by expanding the KL divergence while noting that the posterior normalizing constant (the hidden denominator in \eqref{eq:posterior})
does not depend on $Q$. 
Interestingly, %
the above maximization over measures reduces to an optimization over real numbers. Under our high-temperature \cref{assn:ht}, this maximization has a unique maximizer, and can be written as the unique root of a fixed point equation. We summarize these facts in the following Lemma, which requires notion of \emph{exponential tilting}.
    
\begin{defi}[Exponential tilting]\label{def:expfam}
        Recalling the definition of measures $\mu_i$ from \eqref{eq:posterior_2}, let $\psi_i(\cdot)$ denote the log moment generating function of $\mu_i$, defined by
		$\psi_i(\theta):=\log \int e^{\theta z}d\mu_i(z)$
		for $\theta \in \R$. Define the \emph{exponential tilts} of $\mu_i$ as $$\frac{d\mu_{i,\theta}}{d \mu_i}(z) := \exp(\theta z - \psi_i(\theta)).$$
		By standard exponential family calculations, we have
		$$\psi_i'(\theta)=\EE_{\mu_{i,\theta}}[Z],\quad \psi_i''(\theta)={\rm Var}_{\mu_{i,\theta}}[Z].$$
		In particular, note that both $\psi_i', \psi_i''$ are uniformly bounded by $1$.
        As $\psi_i''(\cdot)$ is strictly positive, $\psi_i':\mathbb{R} \mapsto (-1,1)$ is strictly increasing, and is invertible.
	\end{defi}

\begin{lemma}[Lemma 2.2 in \cite{lee2025rfim}]\label{lem:uniqueness of optimizers}
    Under \cref{assn:ht}, the optimization in \eqref{eq:mean-field} has a unique maximizer $\pQ.$ In particular, by defining $u_i \in (-1, 1)$ as the mean of the $i$th marginal of $\pQ$ as
    $u_i := \EE_{\beta_i \sim Q_i^{\textnormal{prod}}} [\beta_i]$, the vector
    $\bu = (u_1, \ldots, u_p) \in (-1,1)^p$ satisfies the following fixed point equation:
		\begin{align}\label{eq:fpeq}
			u_i &= \psi_i'(s_i + c_i), \quad s_i := \sumjn \A_p(i,j)u_j, \,\,\,\,\forall i = 1, \ldots, p.
		\end{align}
   Also, for each $i$,  the marginal $Q_i^{\textnormal{prod}}$ is uniquely characterized by $u_i$  via the relationship $Q_i^{\textnormal{prod}} = \mu_{i,(\psi_i')^{-1}(u_i)}$, where  $\mu_{i,{\mathbf \cdot}}$ is the exponential tilt of $\mu_i$ defined as in \cref{def:expfam}.
\end{lemma}
\begin{remark}\label{rem:depxy}
We note in passing that both the mean vector $\bu$ and the product measure $\pQ$ depend on $(\bx, \by, \mu)$, but we suppress this dependence for the sake of simplicity. In the remainder of the paper, we will refer the product measure $\pQ$ as the \emph{Mean-Field approximation} of $\nu_{\by,\bx}$ or the \emph{Mean-Field posterior} for short. 
\end{remark}

Now, given that $\pQ$ is well-defined, a fundamental question is whether it is a good approximation of the posterior. One way to address this is to see whether the KL divergence $\rm{KL}(\pQ \mid \nu_{\by, \bx})$ is small. The following Lemma illustrates that under our assumptions on the coupling matrix $\A_p$, this is indeed true.

\begin{lemma}[Corollary 9 in \cite{mukherjee2022variational}]\label{lem:mf accuracy}
    Under Assumptions \ref{assn:ht}, \ref{assn:mf}, \ref{assn:homogeneous diagonals}, we have ${\rm KL}(\pQ \mid \nu_{\by, \bx}) = o(p).$
\end{lemma}
\noindent Note that the typical size of the KL divergence between two $p$-dimensional measures is $O(p)$, so \cref{lem:mf accuracy} indicates that the true posterior $\nu_{\by, \bx}$ and its Mean-Field approximation $\pQ$ is smaller than that. Also, the original claim in \cite{mukherjee2022variational} is stated under slightly weaker assumptions.

We end this short section on Mean-Field variational approximation by highlighting its computational benefits as opposed to sampling from the posterior.

 \begin{remark}[Computing the Mean-Field optimizer]\label{rem:compMField}
 Note that $\bu$ (and hence $\pQ$) can be computed from the data $(\by,\bx)$ and the prior $\mu$, without the need for any posterior sampling. It can be computed efficiently using message passing algorithms (see \cite[Theorem 1.10]{jain2018mean}). In fact, under the high-temperature condition, the Mean-Field solution can be computed exponentially fast (in the number of iterations) by recursively using the fixed point equation \eqref{eq:fpeq} or using gradient based optimization (see \cite[Chapter 2]{Nesterov2018} and  \cite[Chapter 5]{Nocedal1999}).
 \end{remark}

\subsection{Moment bounds with Mean-Field centering}\label{sec:moment bounds}
 As indicated above, the main goal of this paper is to understand the limiting distribution of $\bq^\top \bbeta$ under the posterior, and its posterior mean, under the normalization $\|\bq\|=1$. This section provides some preliminary moment bounds for $\bq^\top \bbeta$ centered using the Mean-Field approximation. In all results in this subsection, the implied constants in $\lesssim$ depends only on $\rho$ (see \eqref{eq:mht}).

 First, consider the problem of understanding the limiting behavior of $$T(\bbeta) := \bq^\top \bbeta\,,$$ where $\bbeta \sim \nu_{\by, \bx}$ is a posterior sample. The following result shows that $T(\bbeta)$ concentrates around the Mean-Field posterior mean $\EE_{\pQ}[\bq^\top \bbeta] = \bq^\top \bu$. 

 \begin{prop}[Theorem 2.3 in \cite{lee2025rfim}]\label{prop:lln}
    Under the high-temperature \cref{assn:ht}, we have
        \begin{align}\label{eq:firstbd}		
          \EE_{\nu_{\by,\bx}}\left[\bq^\top \bbeta - \EE_{\nu_{\by,\bx}}[\bq^\top \bbeta]  \right]^2\le \EE_{\nu_{\by,\bx}}\left[\bq^\top (\bbeta-\bu) \right]^2 \lesssim \max\big(1, p\alpha_p^2\big).
        \end{align}
 \end{prop}

 To better understand the bound, consider a toy case with $\bq = \frac{1}{\sqrt{p}} \mathbf{1}_p$. Then, the problem boils down to understanding the behavior of the coordinate-wise average $\bar{\bbeta} := \frac{1}{p} \sumin \beta_i$ under the posterior. \cref{prop:lln} gives the bound 
 $$\EE_{\nu_{\by,\bx}}[\bar{\bbeta} -\bar{\bu}]^2 \lesssim \frac{1}{p} + \alpha_p^2.$$
This implies the law of large numbers $\bar{\bbeta} = \bar{\bu} + o_P(1)$, as long as $\alpha_p = o(1)$. We stress that the requirement $\alpha_p=o(1)$ is actually needed here, as there are counterexamples to the above law of large numbers when $\alpha_p$ is bounded away from $0$ \citep{qiu2024sub,celentano2023mean}. We draw the reader's attention to another conclusion from \eqref{eq:firstbd}. If the SMF \cref{assn:mf} holds, the RHS above simplifies to $1/p$. Note that this matches the bound one gets under the product approximation $\pQ$:
 $$\EE_{\pQ} [\bar{\bbeta} -\bar{\bu}]^2 \lesssim \frac{1}{p}.$$
Thus under the SMF assumption, the behavior of $\bar{\bbeta}$ is \enquote{very much like} the average of $p$ independent random variables. We will later build on this intuition to derive CLTs under the SMF assumption.

Next, we compare the two centerings in \eqref{eq:firstbd}: $\EE_{\nu_{\by,\bx}}[\bq^\top \bbeta]$ and $\bq^\top \bu$.
Here, the first centering is the Bayes estimator for $\bq^{\top}\bbeta$ %
with  a squared error loss function:
\begin{align}\label{eq:Bayesopt}
    \delta(\by,\bx):=\argmin_{a \in \mathbb{R}} \int \left(a-\bq^\top \bbeta \right)^2\nu_{\by,\bx}(d\bbeta) = \EE_{\nu_{\by,\bx}}[\bq^\top \bbeta].
\end{align}
The second, $\bq^\top \bu$ is a more tractable alternative to $\delta(\by, \bx)$, as illustrated in \cref{rem:compMField}. Given this, a natural question to ask would be how tight is the first inequality in \eqref{eq:firstbd}. 
The following Proposition provides a quantification of this bound.
    \begin{prop}[Regret bound]\label{prop:regret}
        Under the high-temperature \cref{assn:ht}, we have:
        \begin{align}\label{eq:optmf}
        |\delta(\by,\bx)- \bq^\top \bu |\lesssim \sqrt{p}\an,
        \end{align}
        which gives the regret bound
        $$\EE_{\nu_{\by,\bx}}\big(\bq^\top \bbeta - \bq^\top \bu\big)^2 - \EE_{\nu_{\by,\bx}}\big(\bq^\top \bbeta - \delta(\by,\bx) \big)^2 \lesssim p \alpha_p^2.$$
    \end{prop}

\begin{remark}[Bayes optimality of Mean-Field]\label{rem:optmf}
    From \cref{prop:regret}, it follows that the Mean-Field centering $\bq^\top \bu$ is asymptotically Bayes optimal under the SMF \cref{assn:mf}, in the sense that the RHS becomes $o(1)$.
    In comparison to parallel results on the optimality of the Mean-Field optimizer which are asymptotic in nature (see e.g.~\cite[Corollary 1]{mukherjee2023mean} and \cite[Theorem 1]{mukherjee2022variational}), our result provides explicit non-asymptotic rates of convergence.
\end{remark}

\subsection{Data generating assumption}\label{subsec:dgp}
This section introduces additional assumptions on how the true data $\by$ is generated. Note that the posterior $\nu_{\by,\bx}$ is constructed given \emph{any} observation $\by$, and that we have imposed no assumption on $\by$ so far. Consequently, all previous results hold for any deterministic vector $\by$. However, this is not enough to derive the precise limiting distributions that will follow, and we impose the following assumption on the data generating process (DGP) for $\by$. Mainly, we assume that $\by$ is generated from a regression model with covariates $\bx$ and ``true'' parameter $\bbst$. Here, $\bbst$ is not required to be supported in $[-1, 1]$ and may be unbounded. Given $(\mathbf{y},\mathbf{X},{\bm \beta}^\star)$, the statistician generates a sample ${\bm \beta}$ from the statistician's posterior $\nu_{{\bf y},{\bf X}}\equiv \nu_{{\bf y},{\bf X},\mu}$ (see \eqref{eq:posterior_intro}), where $\mu$ is a prior chosen by the statistician for the Bayesian model. This allows the 
``true/data generating'' model to be different from the ``Bayesian'' model \eqref{eq:model}, \eqref{eq:Bayeslin} that was used to build the posterior. In other words, the Bayesian model is allowed to be misspecified. We refer to \cref{fig:dgp} for a pictorial representation of the combined model. While our subsequent results can be extended to allow a misspecified regression variance $\sigma^2 > 0$, for the sake of simplicity of notation we will assume that $\sigma^2$ is correctly specified.

\begin{assume}[Data generating process]\label{assn:dgp}
  Assume that the response vector $\by$ is generated from the ``true/data generating'' model:
  \begin{equation}\label{eq:truemodel}
      \by = \bx\bbst + \boldsymbol{\vep}^\star.
  \end{equation}
  Here, $\mathbf{X}$ is the same $n \times p$ design matrix that was used to construct the posterior in \eqref{eq:model}. The vector $\bbst \in \mathbb{R}^p$ collects the true regression coefficients (which may be deterministic or random), and $\boldsymbol{\vep}^\star \sim \mcn(\mathbf{0}_n, \sigma^2 \mathbf{I}_n)$ represents unobserved Gaussian noise. %
\end{assume}

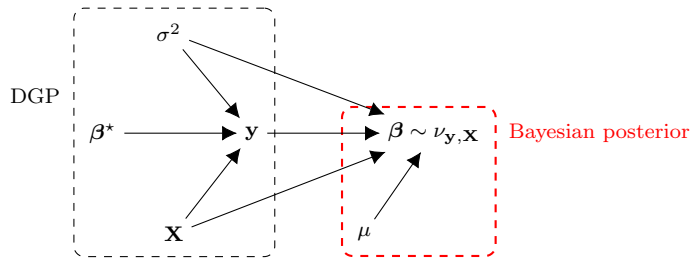
\begin{figure}
    \centering
\begin{tikzpicture}[node distance=1.5cm, auto]
    \centering
    \node[] (betast) {$\bbeta^\star$};
    \node[above=1cm of betast] (sigma_tmp) {};
    \node[right=0.5cm of sigma_tmp] (sigmast) {$\sigma^2$};
    \node[below=1cm of betast] (beta_tmp) {};
    \node[right=0.6cm of beta_tmp] (X) {$\bx$};
    \node[right=of betast] (y) {$\by$};
    \node[right=of y] (beta) {$\bbeta \sim \nu_{\by,\bx}$};
    \node[right=2.1cm of X] (mu) {$\mu$};

    \node[] (label) at (-0.9,0.5) {DGP}; 
    \node[] (label) at (6.6,0) {\color{red} {Bayesian posterior}}; 

    \node [draw=black,dashed,fit=(betast)(sigmast)(X)(y), inner sep=0.1cm, rounded corners] (GBbox) {};
    \node [draw=red,thick,dashed,fit=(mu)(beta), inner sep=0.1cm, rounded corners] (GBbox) {};

    \draw[->] (betast) -- (y);
    \draw[->] (sigmast) -- (y); \draw[->] (X) -- (y);
    \draw[->] (y) -- (beta); \draw[->] (X) -- (beta);
    \draw[->] (mu) -- (beta); \draw[->] (sigmast) -- (beta); 
\end{tikzpicture}
    \caption{
    Graphical illustration of the data generating process (DGP) and the Bayesian posterior. %
    }
    \label{fig:dgp}
\end{figure}

\section{Main result: limiting distributions for the linear statistic and posterior mean}\label{sec:bayeslinhigh}
  This section works under two specific data generating processes that result in asymptotic Gaussian fluctuations for $T(\bbeta) = \bq^{\top}\bbeta$ (when $\bbeta\sim \nu_{\by,\bx}$) and $\delta(\by, \bx) = \EE_{\nu_{\by,\bx}}[\bq^{\top}\bbeta]$. We separate our results in this Section into two parts where we consider (i) a white noise (random) design with general coefficients $\bbst$ (see \cref{sec:wnd}), (ii) a general design matrix with random potentially misspecified coefficients $\bbst$ (see \cref{sec:btm}). 
  Throughout the section, we will reserve the notation $\bbeta$ for the posterior samples.

 To present our main results, we require two additional notations. These notations are related to the limiting variance as well as the centering. First, we define $\psi_0$ as the limiting version of the log mgf $\psi_i$s. Here, recall that $d_0 = d_0(\bx)$ is the constant from \cref{assn:homogeneous diagonals} that approximates the diagonal elements of $\Sigman$.
 
\begin{defi}\label{def:psi_0}
    Define $\mu_0$ by setting $\frac{d \mu_0}{d \mu}(z) \propto \exp \left(- \frac{d_0 z^2}{2}\right)$,
    and define $\mu_{0,\theta}$ to be the exponential tilt of $\mu_0$ (similar to \cref{def:expfam}), given by
    $$\frac{d\mu_{0,\theta}}{d\mu_0}(z)=\exp(
    \theta z-\psi_0(\theta)), \text{ where } \psi_0(\theta):=\log \int e^{\theta z}d\mu_0(z)$$
    for $\theta\in\R$.
 \end{defi}

Next, define $\upsilon_p$ as the following data-dependent quantity. In both theorems that will follow, $\upsilon_p$ can be used to consistently estimate the limiting variance.
 \begin{defi}\label{def:upsilon_p}
     Define $\upsilon_p = \upsilon_p(\by,\bx) := \sumin q_i^2 \psi_i''(c_i) \in (0, 1)$ where $\psi_i(\cdot)$s are defined in \cref{def:expfam} and $c_i$s are defined in \eqref{eq:posterior_2}.
 \end{defi}

\subsection{White noise design}\label{sec:wnd}

In this Section, we assume that the design matrix $\bx$ is the white noise design from \cref{ex:iid gaussian design}. This is a popular choice in high-dimensional Bayesian linear regression \citep{qiu2024sub,barbier2020mutual,mukherjee2023mean}. Recall that each entry of $\bx$ is assumed to be i.i.d sub-Gaussian. In this case we will set $d_0 := \sigma^{-2}$.

\begin{thm}\label{cor:regression i.i.d design}
   Assume that the design matrix $\bx$ is the white noise design in \cref{ex:iid gaussian design}, and suppose $\by$ is generated according to \cref{assn:dgp} for some true $\bbst$ which satisfies
    $\sumin (\beta_i^\star)^4 = O(p)$. %
    Assume further that $\lVert \bq \rVert_\infty \to 0$, and %
    the following limit is well-defined for some $\upsilon > 0$:
  \begin{align}
      \upsilon &:= \lim_{p \to \infty} \sumin q_i^2 \EE \psi_0''(\beta_i^\star/\sigma^2+W), \label{eq:regression i.i.d convergence assumption upsilon} 
  \end{align}
  where $W\sim \mcn(0,\sigma^{-2})$. 
  Then, in the asymptotic regime $p\ll n^{2/3}$, the following statements hold.

  \begin{enumerate}[(a)]
        \item %
        We have $\upsilon_p \mid \bx, \bbeta^\star \xp \upsilon$.
        Further, with $\bbeta \sim \nu_{\by, \bx}$, we have
    \begin{align}\label{eq:iid clt 1}
        \left(T(\bbeta) - \bq^\top \bu\right) \mid \by,\bx \overset{d}{\longrightarrow} \mcn(0,\upsilon),   \quad \Big(T(\bbeta) - \sumin q_i \psi_i'(c_i)\Big) \mid \by,\bx\overset{d}{\longrightarrow} \mcn(0,\upsilon).
    \end{align}

    \item Additionally, suppose that the following limit is well-defined for some $\varsigma^2 > 0$
            \begin{align}
                \varsigma^2 := \lim_{p \to \infty} \sumin q_i^2 \Var \psi_0'(\beta_i^\star/\sigma^2+W).\label{eq:regression i.i.d convergence assumption sigma}
            \end{align}Then 
             \begin{align}\label{eq:iid clt 2}
                 \left(\delta(\by,\bx) - \sumin q_i\EE \psi_0'(\beta_i^\star/\sigma^2+W)\right)\mid\bx,\bbst \overset{d}{\longrightarrow} \mcn(0,\varsigma^2).
             \end{align}
    \end{enumerate}
\end{thm}

We elaborate more on the conclusions of \cref{cor:regression i.i.d design}. Part (a) shows that the posterior $\nu_{\by, \bx}$ projected on $\bq$ is approximately Gaussian. The first limit in \eqref{eq:iid clt 1} considers $T(\bbeta)$ centered by the Mean-Field centering, as in the moment bound in \cref{prop:lln}. Recall that computing the Mean-Field centering requires solving the fixed point equation \eqref{eq:fpeq}. The second limit instead centers by an explicit quantity of $\sumin q_i \psi_i'(c_i)$, which avoids computing the Mean-Field approximation. Note that by using either limits in \eqref{eq:iid clt 1}, we can build credible intervals for $T(\bbeta)$ by approximating the limiting variance $\upsilon$ with the sample based quantity $\upsilon_p(\by, \bx)$.

Next, part (b) addresses the fluctuations of the posterior mean $\delta(\by,\bx)$. Here, the randomness arises from the DGP for $\by$, conditioned on the covariates $\bx$ and true parameter $\bbst$. Even for this result, the centering is a (somewhat) explicit function of $\bbst$, in that one doesn't have to solve a fixed point equation to compute the centering. Interestingly, the centering in part (b) is not $\bq^\top \bbst$, the centering for the Bayes estimator under classical asymptotics (e.g. see Theorem 8.3 in \cite{lehmann2006theory}). This is another consequence of the bounded SNR setting, where the non-vanishing effect of the prior makes the posterior mean asymptotically biased.

In addition to assuming the white noise design $\bx$, \cref{cor:regression i.i.d design} requires some extra assumptions, and we discuss them below. First, we assumed $\|\bq\|_\infty \to 0$, which is analogous to the uniform asymptotitc negibility (UAN) condition for the classical Lidbeberg-Feller CLT, and rules out localized vectors $\bq$. In other words, we assume that the summand $\sumin q_i \beta_i$ is not dominated by finitely many terms. This is a necessary condition to get a Normal limit, as each coordinate $\beta_i$ is generally non-Normal and does not concentrate around $u_i$ due to our bounded SNR scaling. Thus, we average out the coordinates so that the Central Limit Theorem can kick in to guarantee a Normal limit.

    Second, we require the existence of the two limits in \eqref{eq:regression i.i.d convergence assumption upsilon} and \eqref{eq:regression i.i.d convergence assumption sigma}, which define the limiting variances. This can be reformulated in terms of weighted empirical measures. To wit, define 
    $$\mathcal{L}_p^{\bq}:=\sumin q_i^2 \delta_{\beta^*_i}.$$
    If we assume that $\mathcal{L}_p^{\bq}\overset{w}{\longrightarrow} \mathnormal L$, then the limits can be written as $\upsilon=\EE\psi_0''(B/\sigma^2+W)$ and $\varsigma^2=\EE[\Var(\psi_0'(B/\sigma^2+W)|B)]$, where $B\sim \mathnormal L $ and $W\sim \mcn(0,\sigma^{-2})$ are independent. 

    Here, the value of the limiting variances $\upsilon$ and $\varsigma^2$ may appear abstract at first glance. Interestingly, these values can be understood by considering the limiting distributions \emph{under the Mean-Field approximation} $\pQ$. Recall that $\pQ$ is a product measure where the $i$th marginal $Q_i^{\text{prod}}$ has mean $u_i$. Proving a CLT for independent $\bbeta_{\mathrm{ind}} \sim \pQ$ is significantly simpler, and it can be checked that $\bq^\top \bbeta_{\mathrm{ind}}$ has mean $\bq^\top \bu$ and approximate variance $\upsilon_p$. The following proposition illustrates that the conclusions in \cref{cor:regression i.i.d design} still hold under the Mean-Field posterior. Thus, under the white noise design in \cref{ex:iid gaussian design}, the Mean-Field approximation produces valid limiting distributions. In particular, it does not underestimate %
    the variance, as opposed to what is empirically observed in the literature \citep{blei2017variational,zhang2018advances}. Consequently, credible intervals constructed under $\pQ$ is still valid under the true posterior $\nu_{\by,\bx}$.

\begin{prop}[Analog of \cref{cor:regression i.i.d design} under $\pQ$]\label{prop:iid mean-field clt}
Suppose $(\by,\bx)$ satisfy the conditions of \cref{cor:regression i.i.d design}. Let $\bbeta_{\rm ind}\sim \pQ\equiv \pQ(\by,\bx)$ be a random sample from the Mean-Field posterior (where $\pQ$ is as in \cref{lem:uniqueness of optimizers}). 
Then the following conclusions hold: 
\begin{enumerate}[(a)]
    \item The weak limits in \eqref{eq:iid clt 1} hold with $\bbeta$ replaced by $\bbeta_{\rm ind}$.
    \item Assuming the existence of the limit \eqref{eq:regression i.i.d convergence assumption sigma}, \eqref{eq:iid clt 2} holds when $\delta(\by, \bx)$ is replaced by $\bq^\top \bu = \EE_{\pQ} \left[\bq^\top \bbeta_{\mathrm{ind}}\right].$
\end{enumerate}
\end{prop}

\begin{remark}[Tightness of the asymptotic scaling]\label{rmk:berry esseen}
    While we have omitted displaying the finite-sample Berry-Esseen bounds in the CLTs in \cref{cor:regression i.i.d design} for brevity, our proof additionally reveals explicit rates of convergence. For example, in part (a), the main error terms in the Berry-Esseen upper bound for the Kolmogorov-Smirnov distance (excluding the error resulting from the empirical limit in \eqref{eq:regression i.i.d convergence assumption upsilon}) can be written as
    $$\sqrt{\frac{p}{n}} + \frac{p^{3/2}}{n} + \|\bq\|_\infty,$$
    and we believe that our asymptotic regime $p\ll n^{2/3}$ is a hard threshold. Indeed, this is a bottleneck in terms of verifying both our design matrix assumptions \ref{assn:ht} and \ref{assn:mf}.
    Compared to existing literature on high-dimensional posterior limiting distributions %
    that consider $p\ll n^{1/3}$ \citep{spokoiny2013bernstein,ghosal1999asymptotic,ghosal2000asymptotic,lu2017bernstein} or $p\ll \sqrt{n}$ \citep{Katsevich2023,giordano2025good}, our result allows much higher dimensional covariates. 
\end{remark}

     Below, we give applications of~\cref{cor:regression i.i.d design}, by illustrating that the limits in \eqref{eq:regression i.i.d convergence assumption upsilon} and \eqref{eq:regression i.i.d convergence assumption sigma} exist under natural assumptions on $\bbeta^\star$.

    \begin{ex}[Misspecified Bayesian model for $\bbeta^\star$]\label{ex:misspecific}
            Suppose the true coefficient vector $\bbeta^\star$ is generated i.i.d from a distribution $\mu^\star$ on $[-1,1]$. In particular, when $\mu^\star$ is equal to the prior $\mu$, the model \eqref{eq:Bayeslin} is correctly specified.
            Note that for any $\bq$ with $\lVert \bq \rVert_\infty \to 0$, and writing $W\sim \mcn(0,\sigma^{-2})$ (independent of $\bbeta^\star$), the limits
            \begin{align*}
                \upsilon=&\lim_{p \to \infty} \sumin q_i^2 \EE \big[\psi_0''(\beta_i^\star/\sigma^2 + W) \mid \beta_i^\star \big] = \EE \psi_0''(\beta_1^\star/\sigma^2+W),   \\
               \varsigma^2=& \lim_{p \to \infty} \sumin q_i^2 \Var \big[\psi_0'(\beta_i^\star/\sigma^2 + W) \mid \beta_i^\star\big] = \EE \big[\Var(\psi_0'(\beta_1^\star/\sigma^2+W)|\beta_1^\star) \big] 
            \end{align*}
            exist by a standard second momemt argument.
    \end{ex}

\begin{ex}[Spike-and-slab \citep{Ishwaran2005,Rockova2015,Castillo2012}]\label{ex:sslab} 
Sample $r_p$ from a $\mbox{Beta}(1,p^u)$ distribution for $u>0$ and then generate $(\beta_1^\star,\ldots ,\beta_p^\star)$ as follows:
$$\beta_i^\star|r_p \overset{i.i.d}{\sim} \, (1-r_p)\delta_0 \, + \, r_p \otb,$$
for some fixed distribution $\otb$ supported on $[-1,1]$. Note that $(\beta_1^\star,\ldots ,\beta_p^{\star})$ are only conditionally independent given $r_p$, and not marginally independent as in the previous example.
\end{ex}

\noindent In the above example, it is natural to expect that if the \emph{sparsity} is high enough, then the limiting behavior of $\sumin q_i\EE_{\nu_{\by,\bx}}[\beta_i]$ would be the same as in the null case $\bbeta^\star = \mathbf{0}_p$. The following Corollary characterizes this threshold of sparsity in terms of $u$ and $\bq$. %

  \begin{cor}\label{cor:sparse}
    Consider the assumptions of \cref{cor:regression i.i.d design}. Define $q^{(p)}_{tot}:=\sumin q_i$ and sample $W\sim \mcn(0,\sigma^{-2})$ independent of $\otb$.  
   In the setting of \cref{ex:sslab}, the following conclusions hold: 
   \begin{enumerate}[(a)]
   \item Suppose $\EE\psi_0'(\otb/\sigma^2+W)=0$. Then 
        $\delta(\by,\bx) \mid \bx, \bbst \overset{d}{\longrightarrow} \mcn\big(0,  \Var \psi_0'(W)\big)$.
    \item Suppose $\EE\psi_0'(\otb/\sigma^2+W)\neq 0$ and $p^{-u}q^{(p)}_{tot}\to \zeta$ for some $\zeta\in\R$, we then have
    $\delta(\by,\bx) \mid \bx,\bbst \overset{d}{\longrightarrow} \mcn\big(\zeta \EE\psi_0'(\otb/\sigma^2+W), \Var \psi_0'(W)\big)\,.$
    \item Finally, if $\EE\psi_0'(\otb/\sigma^2+W)\neq 0$ and $p^{-u}q^{(p)}_{tot}\to \pm \infty$, then $\delta(\by,\bx) \, \mid \bx,\bbst \stackrel{P}{\to}\pm \infty.$
   \end{enumerate}
\end{cor}

Here, \cref{cor:sparse} shows that the sparsity threshold depends on the symmetry of $\otb$, and the size of $q_{tot}^{(p)}=\sumin q_i$. If $\otb$ is a symmetric distribution or $\bq$ is an approximate contrast satisfying $q_{tot}^{(p)}=o(p^u)$, then $\delta(\by,\bx)$ always exhibits null behavior. The latter is always satisfied when $u>1/2$, as the Cauchy-Schwartz inequality gives $|q_{tot}^{(p)}|\le \sqrt{p}$. Otherwise, the limiting distribution may exhibit a nonzero centering, or even diverge in the limit.

\subsection{Misspecified Bayesian model}\label{sec:btm}
In this Section, we weaken the random design assumption from \cref{sec:wnd}. We will stick to the data generating \cref{assn:dgp}, but additionally assume that (i) the vector $\bq$ is an approximate eigenvector of $\Sigman$, (ii) $\beta^{\star}_1,\ldots ,\beta^{\star}_p\overset{i.i.d}{\sim}\mu^{\star}$ for some probability distribution $\mu^{\star}$ on $\R$ which is symmetric around $0$, and (iii) the Bayesian prior $\mu$ used to construct the posterior \eqref{eq:posterior} is also symmetric. Note that $\mu^{\star}$ need not be equal to $\mu$.
While the prior $\mu$ used in \eqref{eq:posterior} is supported on $[-1, 1]$, $\mu^\star$ does not have to satisfy this. Consequently, we allow a wide range of $\mu^\star$ that ranges from degenerate distributions to measures with unbounded support. 

Below, we present the CLTs for $T(\bbeta)$ and $\delta(\by, \bx)$ under this setup.
\begin{thm}\label{cor:regression Bayesian truth}
    Suppose that the design matrix $\mathbf{X}$ satisfies Assumptions \ref{assn:ht}-\ref{assn:homogeneous diagonals}, and that $\by$ is generated according to \cref{assn:dgp}.
    Also, suppose that $\A_p = -\frac{\mathrm{Off}(\bx^\top \bx)}{\sigma^2}$ has an approximate delocalized eigenpair $(\lambda_p, \bq)$ in the following sense: 
    \begin{align}\label{eq:approximate eigenpair}
        \lVert \A_p \bq - \lambda_p \bq\rVert \to 0, \quad \lambda_p \to \lambda \in (-1, 1), \quad \|\bq\|_\infty \to 0.
    \end{align} %
    Suppose that the Bayesian prior $\mu$ in \eqref{eq:Bayeslin} is symmetric. 
    For a symmetric measure $\mu^\star$ on $\mathbb{R}$ with finite fourth moments, suppose that $\beta_i^\star \overset{i.i.d.}{\sim} \mu^\star$. 
    Let $W_0\sim \mcn(0,d_0)$ be a random variable that is independent of $\bbst$. 
          
    \begin{enumerate}[(a)]
        \item 
        \begin{enumerate}
            \item[(i)]
       Under the above conditions, we have $\upsilon:=\EE \psi_0''(d_0\beta_1^\star+W_0)>0$ (this depends on $\mu,\mu^\star,d_0)$, and $\upsilon_p \mid \bx, \bbeta^\star \xp \upsilon$. Further we have
    \begin{equation*}
        T(\bbeta) - \bq^\top \bu \mid \by,\bx \overset{d}{\longrightarrow} \mcn\left(0,\frac{\upsilon}{1-\lambda\upsilon}\right).
    \end{equation*}

   \item[(ii)]
    Additionally, assume that the eigenpair $(\lambda_p, \bq)$ satisfies 
    \begin{align}\label{eq:stronger eigenpair}
        (1+\sqrt{p\alpha_p})\lVert \A_p \bq - \lambda_p \bq \rVert = o(1).
    \end{align}
    Then we can modify the centering as follows:
    \begin{equation}\label{eq:regression CLT clean}
        T(\bbeta) - \frac{\sumin q_i \psi_i'(c_i)}{1-  \lambda_p \upsilon_p}  \mid \by,\bx \overset{d}{\longrightarrow} \mcn\left(0,\frac{\upsilon}{1-\lambda\upsilon}\right).
    \end{equation}
    \end{enumerate}

    \item %
    Assume \eqref{eq:stronger eigenpair}, and slightly strengthen \cref{assn:homogeneous diagonals} by assuming $\sumin (d_i-d_0)^2=o(1)$. Then, setting
    $$\varsigma^2 := \frac{\EE \left[\Var \Big( \psi_0'(d_0 \beta_1^\star + W_0) \mid \beta_1^\star\Big)\right]}{(1-\lambda \upsilon)^2} - \frac{\lambda \upsilon^2}{(1-\lambda \upsilon)^2}$$
    (which depends on $\mu, \mu^\star, d_0, \lambda$), we have $\varsigma^2 > 0$. Further, we have
  \begin{small}  \begin{align}\label{eq:regression CLT annealed}
        \left(\delta(\by,\bx) - \frac{1}{1-\lambda \upsilon} \sumin q_i \Big(\EE \big[\psi_0'(d_0 \beta_i^\star +W_0) \mid \beta_i^\star] -\lambda \upsilon \beta_i^\star \Big) \right)\mid\bx,\bbst \overset{d}{\longrightarrow} \mcn(0,\varsigma^2).
    \end{align}
    \end{small}
    \end{enumerate}
\end{thm}

We first understand the conclusion of the Theorem. While \cref{cor:regression Bayesian truth} presents CLTs analogous to that as in \cref{cor:regression i.i.d design}, the centerings and limiting distributions are more complicated. This is because we are allowing general design matrices beyond the white noise design, which modify the limiting variances. To compare the results, note that \eqref{eq:approximate eigenpair} holds under the white noise design by taking $\lambda = 0$ and an arbitrary unit vector ${\bf q}$, as long as $\|\bq\|_\infty \to 0$. In other words, an arbitrary unit vector ${\bf q}$ satisfying the delocalization condition $\|{\bf q}\|_\infty\to 0$ is an approximate eigenvector of $\A_p$ with eigenvalue $0$. Thus, plugging in $\lambda_p, \lambda = 0$ in \eqref{eq:regression CLT clean} and \eqref{eq:regression CLT annealed} recovers the claims in \cref{cor:regression i.i.d design}, as the values $\upsilon, \varsigma^2$ in both the results agree.

We also discuss the centering in \eqref{eq:regression CLT annealed} for $\delta({\bf y},{\bf X})$. By simple algebra, we can re-write the centering as 

\begin{align}\label{eq:bias}
    \bq^\top \bbst + \text{bias}_p(\bbst), \quad   \text{bias}_p(\bbst) := \frac{1}{1-\lambda \upsilon} \sumin q_i \Big( \EE \big[\psi_0'(d_0 \beta_i^\star +W_0) \mid \beta_i^\star] - \beta_i^\star \Big).
\end{align}
This can be viewed as the true parameter of interest plus some bias due to the non-vanishing effect of the prior. Under the true model assumption $\beta_i^\star \overset{i.i.d.}{\sim} \mu^\star$ for some symmetric $\mu^\star$, $\text{bias}_p(\bbst)$ is a random variable with mean $0$ and non-vanishing variance. This implies that the Bayes estimator will be asymptotically biased (with high probability) for each realization of $\bbeta^\star$, but the bias will have expectation $0$ when $\bbeta^\star$ is marginalized out.

Next, we discuss the approximate eigenpair assumptions  on the off-diagonal matrix $\A_p = -\Off(\Sigman)$ (see \eqref{eq:approximate eigenpair} and \eqref{eq:stronger eigenpair}). This is a somewhat technical requirement that results from \citep{lee2025rfim}, which is a large building block for our proofs. Under \cref{assn:homogeneous diagonals}, the diagonals of the gram matrix $\Sigman = \sigma^{-2} (\bx^\top \bx)$ are essentially constant, and so we can equivalently understand the vector $\bq$ as an approximate eigenvector of $\Sigman = \sigma^{-2} (\bx^\top \bx)$, due to the decomposition in \eqref{eq:anbayes}. Also, under the high-temperature assumption \cref{assn:ht}, the eigenvalues $\lambda_p$ must satisfy $$|\lambda_p| \le \|\A_p\| \le \|\A_p\|_4 \le \rho < 1, \quad \forall p.$$

The following Proposition re-addresses comparing the posterior limiting distribution to that under the Mean-Field approximation, but for general design matrices $\bx$ in the setting of Theorem \ref{cor:regression Bayesian truth}. 

\begin{prop}[Analog of \cref{cor:regression Bayesian truth} under $\pQ$]\label{prop:Bayes true model mean-field clt}
    Consider the setting of \cref{cor:regression Bayesian truth}, except for assuming that $\bbeta$ is a random sample from the Mean-Field approximation $\pQ$.
\begin{enumerate}[(a)]
    \item Recall the definition of $\upsilon$ from part (a) of \cref{cor:regression Bayesian truth}. Then we have 
    $$T(\bbeta)-\bq^{\top}\bu \mid \by, \bx \overset{d}{\longrightarrow} N(0,\upsilon).$$
    \item Under the setting of part (b) in \cref{cor:regression Bayesian truth}, \eqref{eq:regression CLT annealed} holds when $\delta(\by, \bx)$ is replaced by $\bq^\top \bu = \EE_{\pQ} \left[\bq^\top \bbeta_{\mathrm{ind}}\right].$
\end{enumerate}
\end{prop}

\begin{remark}[Variance under/overestimation]\label{rem:varoverunder}
First, note that the limiting distributions of the posterior mean $\delta(\by, \bx)$ in \cref{cor:regression Bayesian truth}, part (b),  and \cref{prop:Bayes true model mean-field clt}, part (b), are the same. This is not surprising, since we have already established that they are close under milder assumptions in \eqref{eq:optmf}. However, the limiting distributions of $T(\bbeta) - \bq^\top \bu$ in \cref{cor:regression Bayesian truth}, part (a),  and \cref{prop:Bayes true model mean-field clt}, part (a), are different. The limiting variance under the true posterior $\nu_{\by,\bx}$ is $(1-\upsilon\lambda)^{-1}\upsilon$ whereas the variance under $\pQ$ is $\upsilon$. This is not equal unless $\lambda = 0$, which was the case under the white noise design in the previous section. In general, depending on whether $\lambda<0$, $=0$, or $>0$, the limiting variance under $\nu_{\by,\bx}$ is smaller, equal, or larger, than that of the corresponding Mean-Field approximation. 

For a concrete example of this phenomenon, recall the high-dimensional ANOVA design in \cref{ex:anova design}. There, we have two exact eigenpairs with nonzero eigenvalues: (i) $\lambda_p^{(1)} = - \frac{1}{2\sigma^2} < 0$, $\bq^{(1)} = \frac{1}{\sqrt{p}} \mathbf{1}_p$ and (ii) $\lambda_p^{(2)} = \frac{1}{2\sigma^2}$, $\bq^{(2)} = \frac{1}{\sqrt{p}}(\mathbf{1}_{p/2}^\top, -\mathbf{1}_{p/2}^\top)^\top$. Therefore, in case (i), the Mean-Field variance is higher than the posterior variance, whereas it is lower in case (ii). All other eigenvectors (any delocalized vector $\bq$ orthogonal to $\bq^{(1)}, \bq^{(2)}$) correspond to $\lambda = 0$, so the Mean-Field approximation captures the correct posterior variance. We refer the reader to \cref{fig:true coverage probability} for a visual illustration. %
\end{remark}

\subsection{Detour: Extension to fractional posteriors}\label{sec:fractional posterior}
An alternative to the posterior in \eqref{eq:posterior} is the popular notion of $\gamma$-fractional posteriors (see \cite{Grunwald2012,Yang2019,Miller2019,Avella2022,ray2023asymptotics}; also known as tempered, power, or $\alpha$-posteriors in the literature). Here, $\gamma \in (0, \infty)$ is the exponent for the likelihood when defining the posterior.
In the context of the linear model \eqref{eq:model}, the $\gamma$-fractional posterior is given by 
\begin{small}
\begin{equation}\label{eq:alph-posterior}
		\begin{aligned}	&\;\;\;\;\nu^{(\gamma)}_{\mathbf{y},\mathbf{X}}(d\boldsymbol{\beta})\propto \exp\left( \frac{1}{2} \boldsymbol{\beta}^{\top} \A_p^{(\gamma)} \boldsymbol{\beta}+(\mathbf{c}^{(\gamma)})^{\top}\boldsymbol{\beta}\right)\prod_{i=1}^p \mu^{(\gamma)}_i(d\beta_i), \qquad \mbox{where}\\
        \mathbf{c}^{(\gamma)}:=\frac{\mathbf{X}^{\top}\mathbf{y}}{\sigma^2/\gamma}, \;\; &\A_p^{(\gamma)}(i,j)= - \frac{(\mathbf{X}^{\top}\mathbf{X})(i,j)\mathbf{1}_{i\neq j}}{\sigma^2/\gamma}, \;\;\frac{d\mu_i^{(\gamma)}}{d\mu}(\beta_i)\propto \exp\left(-\frac{d^{(\gamma)}_i\beta_i^2}{2}\right),  
        \end{aligned}
\end{equation}
\end{small}

\noindent and $ d^{(\gamma)}_i:=\frac{(\mathbf{X}^{\top}\mathbf{X})(i,i)}{\sigma^2/\gamma}$, for $i,j = 1, 2, \ldots, p$. Note that the above display can be viewed as the usual posterior \eqref{eq:posterior}, but with a re-scaled standard error $\sigma^{(\gamma)} := \sigma/\sqrt{\gamma}$.
Consequently, both Theorems \ref{cor:regression i.i.d design} and \ref{cor:regression Bayesian truth} directly apply to $\gamma$-fractional posteriors, after this simple change of variance. 

In particular, the $\gamma <1$ case has received a lot of attention, as the resulting posterior exhibits robustness against model misspecification \citep{Avella2022} and contracts with respect to the associated R\'{e}nyi divergence under mild requirements \citep[Theorem 3.2]{Yang2019}. Our results do not require any hard constraints on $\gamma$, and the dependence of $\gamma$ is more delicately reflected on the high-temperature \cref{assn:ht}. Note that \cref{assn:ht} becomes weaker as $\gamma$ decreases, allowing a wider range of design matrices.

\section{Application: Bernstein-von Mises analogs}\label{sec:Bvmanalog} 

An important result in the literature is the Bernstein-von Mises theorem \citep{van2000asymptotic, Kleijn2012}. This theorem implies that the Bayesian level $100(1-\alpha)\%$ credible intervals are also asymptotically $100(1-\alpha)\%$ confidence intervals, under a contracting posterior that satisfies some regularity conditions. An immediate question arises: do similar conclusions apply to bounded SNR settings with a non-contracting posterior? We answer this question in an \emph{average sense}, under the setting of \cref{sec:btm}.
\\

\noindent Given $\alpha\in (0,1)$, let $c_{\alpha/2}$ be a constant such that $\PP(|\mcn(0,1)|\ge c_{\alpha/2})=\alpha$. Note that  \cref{cor:regression Bayesian truth}, part (a), yields the following level $(1-\alpha)$ posterior credible set for $T(\bbeta) = \bq^\top \bbeta$: 
$$\mathcal{I} \equiv \mathcal{I}(\by, \bx, \mu) :=  \Big[\bq^\top \bu\pm c_{\alpha/2} \sqrt{\frac{\upsilon_p}{1-\lambda_p \upsilon_p}} \Big].$$
Here, we have replaced the limiting variance $\upsilon/(1-\lambda \upsilon)$ by its sample-based approximation $\upsilon_p/(1-\lambda_p \upsilon_p)$, so all terms in $\mathcal{I}$ depend only on the data $\by, \bx$ and the prior $\mu$. As in \cref{sec:btm}, we assume $\beta_1^\star,\ldots ,\beta_p^\star\overset{i.i.d.}{\sim}\mu^\star$, which may be different from $\mu$. Given a credible interval $\mathcal{I}$ and distribution $\mu^\star$, define its average coverage probability as:
$$\text{AC}(\mathcal{I}; \mu^\star) := \PP \Big(\bq^\top \bbst \in \mathcal{I} \mid \bx\Big).$$

\begin{thm}\label{cor:coverage probability Bayes}
    Consider the setting and notation from part (b) of \cref{cor:regression Bayesian truth}. %
    Recall that we denote the Bayesian prior as $\mu$ and data generating prior as $\mu^\star$. For random variables $\beta_1^\star \sim \mu^\star \perp W_0 \sim \mcn(0,d_0)$, set 
    $$\tau^2 := \frac{\Var\big(\beta_1^{\star}-\psi_{0}'(d_0\beta_1^{\star} + W_0)\big) - \lambda \upsilon^2 }{(1-\lambda \upsilon)^2},$$
    where $\upsilon = \EE \psi_0''(d_0 \beta_1^\star + W_0)$ is the same constant from \cref{cor:regression Bayesian truth}, and $\psi_0$ is defined in \cref{def:psi_0}. Note that the constants $\upsilon, \tau^2$ depend both on $\mu$ (implicitly via the function $\psi_0$) and $\mu^\star$.

\begin{enumerate}[(a)]
    \item Let $u_{i}$ be the Mean-Field optimizer in \eqref{eq:fpeq}. %
    Then, the following unconditional limit (on $\bbst$) holds:
    \begin{align}\label{eq:coverage clt modified}
        \sumin q_i(u_{i} - \beta_i^\star) \mid \bx \xd \mcn\big(0, \tau^2\big).
    \end{align}

    \item The average $100(1-\alpha)\%$ coverage probability of the credible interval $\mathcal{I}(\by,\bx,\mu)$ satisfies
    \begin{align}\label{eq:coverage}
        \textnormal{AC}(\mathcal{I}(\by, \bx, \mu); \mu^\star) \to 1 - 2 \Phi \left(c_{\alpha/2} \sqrt{\frac{\upsilon}{(1-\lambda \upsilon)\tau^2}}\right).
    \end{align}
    In particular, when the nondegenerate prior is well-specified in the sense that $\mu = \mu^\star$, the RHS of \eqref{eq:coverage} is equal to $1 - \alpha$.

\end{enumerate}
\end{thm}

We elaborate on the results. 
Part (a) provides an unconditional limit (i.e., averaged over the distribution of $\beta_1^\star,\ldots ,\beta_p^\star$) for $\sumin q_i u_{i}$. Recalling from \cref{prop:regret} that $\delta(\by,\bx) = \bq^\top \bu + o(1),$ part (a) can be viewed as the unconditional analog of the limiting distribution of the posterior mean $\delta(\by,\bx)$ (see part (b) in \cref{cor:regression Bayesian truth}) around $\bq^\top \bbeta^\star$, the ``correct'' centering. The proof additionally reveals that the unconditional variance $\tau^2$ is strictly larger than the conditional variance $\varsigma^2$ in \cref{cor:regression Bayesian truth}, which is a natural consequence of marginalizing out the randomness of the centering bias (recall the decomposition \eqref{eq:bias}).

Part (b) computes the asymptotic coverage probability of the Bayesian credible interval $\mathcal{I}$ under potentially misspecified priors. In particular, when the prior is well-specified with $\mu = \mu^\star$, the theorem shows that the credible interval $\mathcal{I}$ also functions as an asymptotic confidence set with an \emph{average true coverage} of level $(1-\alpha)$. This is analogous to the usual consequence of the Bernstein-von Mises theorem in the high SNR regime \citep{van2000asymptotic}. To elaborate, when $\mu = \mu^\star$, the proof of part (b) shows that $\tau^2 = (1-\upsilon\lambda)^{-1}\upsilon$. This value is identical to the limiting \emph{posterior variance} of $\bq^\top \bbeta$ in part (a) of \cref{cor:regression Bayesian truth}. The fact that the two variances (under randomness of the true data and posterior, respectively) are the same implies the same level of uncertainty, which is again analogous to standard results for contracting posteriors with high SNR. Additionally, the averaging out the randomness of $\bbst$ in \eqref{eq:coverage} is necessary, as the conditional coverage probability would have $O(1)$ fluctuations otherwise (due to the random bias in \eqref{eq:bias}).

In Figure \ref{fig:true coverage probability} below, we take $\alpha = 0.05$ and visualize the limiting true coverage probability (in the RHS of \eqref{eq:coverage}) with respect to $\lambda \in (-1, 1)$ under varying assumptions on the true prior $\mu^\star$. We assume that the Bayesian specifies the prior $\mu = \text{Unif}[-1,1]$ to build the posterior, whereas true prior is (left) $\mu^\star = \mu$, (middle) $\mu^\star = \frac{1}{2}\delta_0 + \frac{1}{2} \mcn(0,1)$, (right) $\mu^\star = \frac{1}{3}\delta_{-1} + \frac{1}{3} \delta_0 + \frac{1}{3} \delta_1$. Here we set the diagonal elements of $\Sigman$ as $d_0 = 1$.
In addition to the posterior credible intervals considered in \cref{cor:coverage probability Bayes} (solid blue lines), we also consider the credible intervals based on the Naive Mean-Field (NMF) approximation of the posterior (dashed red lines):
$$\mathcal{I}^{\text{NMF}}(\by, \bx, \mu) := \Big[\sumin q_i u_i \pm c_{\alpha/2} \sqrt{\upsilon_p} \Big]$$
which is a natural consequence of the CLT under $\pQ$ in \cref{prop:Bayes true model mean-field clt}.
The solid line in the left panel is constant at $0.95$, which implies that the Bayesian credible intervals have coverage guarantees under a correct model specification. Also, by comparing the solid and dashed lines, the credible intervals based on the Mean-Field approximation have higher coverage if and only if $\lambda < 0$. This is a consequence of the inconsistency of the variance under the Mean-Field approximation, as discussed in \cref{rem:varoverunder}.

\begin{figure}[h!]
    \centering
    \includegraphics[width=\linewidth]{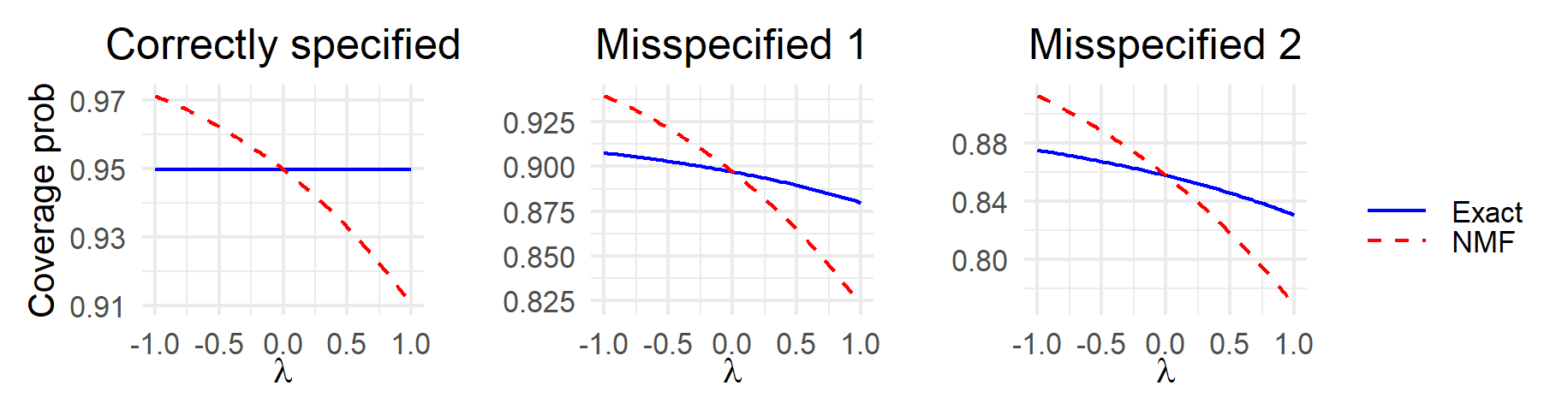}
    \caption{True coverage under a (left) correctly specified prior $\mu^\star = \mu = \text{Unif}[-1,1]$, and misspecified priors: (middle) $\mu^\star = \frac{1}{2}\delta_0 + \frac{1}{2} \mcn(0,1)$, (right) $\mu^\star = \frac{1}{3}\delta_{-1} + \frac{1}{3} \delta_0 + \frac{1}{3} \delta_1$. ``Exact'' denotes the coverage of the $95\%$ credible intervals in \cref{cor:coverage probability Bayes}, and ``NMF'' denotes the coverage of the $95\%$ credible intervals based on the mean-field posterior approximation $\pQ$.}
    \label{fig:true coverage probability}
\end{figure}

\begin{comment}
In \cref{fig:change sigma square}, the average coverage probability of credible intervals under correctly specified $\gamma$-fractional posteriors $\nu^{(\gamma)}_{\mathbf{y},\mathbf{X}}$ with $\gamma = 0.5, 1, 2$ are displayed. Here, we assume a correctly specified prior $\mu^\star = \mu = \text{Unif}[-1, 1]$. %
Note that the second panel of \cref{fig:change sigma square} is identical to the first panel of \cref{fig:true coverage probability}, since the fractional posterior with $\gamma = 1$ reduces to the usual posterior $\nu_{\by,\bx}$.
When $\gamma < 1$, both the exact and NMF credible intervals enjoy a coverage probability greater than $95\%$. On the other hand, when $\gamma > 1$, both the true fractional posterior and its Mean-Field approximation leads to under-coverage. 
 
\begin{figure}[h!]
    \centering
    \includegraphics[width=\linewidth]{figs/misspecified_new.png}
    \caption{True coverage probability under potentially misspecified regression variance. The prior $\mu^\star = \mu = \text{Unif}[-1, 1]$ is correctly specified. ``Exact'' denotes the coverage of the $95\%$ credible intervals in \cref{cor:coverage probability Bayes}, and ``NMF'' denotes the coverage of the $95\%$ credible intervals based on the mean-field posterior approximation $\pQ$.}
    \label{fig:change sigma square}
\end{figure}
\end{comment}

  %
  %
  %
  %
  %
  %
  %
  %
  %
  %
  %
  
  %
  %
  %
  %
  %
  %
  %
  %
  %
  %
  %
  %
  %

  %

\section{Proof overview and technical challenges}
The main results of this paper are Theorems \ref{cor:regression i.i.d design}, \ref{cor:regression Bayesian truth}, and  \ref{cor:coverage probability Bayes}. Theorems  \ref{cor:regression i.i.d design} and \ref{cor:regression Bayesian truth} both derive a CLT for linear statistics of the form $\mathbf{q}^\top {\bm \beta}$, and of the Bayes estimator $\delta(\mathbf{y},\mathbf{X})=\EE[{\bf q}^\top {\beta}]$,  under two different sets of assumptions. The main idea behind these results is that the log-posterior of ${\bm \beta}$ in a linear regression model has a quadratic form, which essentially reduces it to a \enquote{Random Field Ising model} \citep[see][]{chatterjee2023features,de1991fluctuations,lowe2023propagation}. In a recent paper \cite{lee2025rfim}, we studied Berry-Esseen type CLTs in the setting of Random Field Ising models, where the \enquote{external magnetic field} consists of i.i.d random variables. 
The main subtlety arises from the entrywise dependence of the \enquote{external magnetic field} vector $\bc = \sigma^{-2} \bx^\top \by$ as  
    $$\bc \mid \bx, \bbst \sim \mcn (\Sigman \bbst, \Sigman)$$
    under the data generating assumption (see \cref{assn:dgp}), where $\Sigman$ is defined after \eqref{eq:posterior_intro}. Moreover, the high-dimensional nature of the problem where both $p,n\to\infty$ also makes it challenging to deal with the entrywise dependence.   Consequently, bounding the individual error terms in the CLT from \cite{lee2025rfim}  requires a much more delicate analysis.

    To further highlight the technical challenges, let us consider the random design setting of \cref{cor:regression i.i.d design}, and also assume $\sigma^2=1$. Here it might be tempting to use the approximation $\Sigman\approx\mathbf{I}_p$ to directly approximate the interaction term in the posterior (see \eqref{eq:posterior_intro}) with a coordinatewise separable term. To wit, note that 
    $$\sup_{\bbeta\in [-1,1]^p}\big|\bbeta^{\top}\Sigman\bbeta-\lVert \bbeta\rVert^2\big|= \sup_{\bbeta\in [-1,1]^p} \big|\bbeta^{\top}(\Sigman-\mathbf{I}_p)\bbeta \big| =  O_P\left(p\sqrt{\frac{p}{n}}\right)=O_P\left(\sqrt{\frac{p^3}{n}}\right).$$
    Here we have used the approximation $\lVert \Sigman-\mathbf{I}_p\rVert= O_p(\sqrt{p/n})$ (see \cite[Theorem 5.6.1]{vershynin2018high}). The above display shows that a direct product measure approximation of the entire posterior in total variation distance would require the stringent condition $p=o(n^{1/3})$. However, leveraging significant cancellations in the analysis of $\bq^{\top}\bbeta$ (see Lemmas \ref{lem:random design properties} and \ref{lem:lincomb expectation random design}), we show that for linear projections our CLT holds under the much weaker condition $p=o(n^{2/3})$. We note that even in the high SNR setting in the random design case, to the best of our knowledge, Bernstein-von Mises type results are known only up to $p=o(\sqrt{n})$ \citep[see][]{spokoiny2013bernstein,Katsevich2023}.
Possibly more challenging is the proof of \cref{cor:regression Bayesian truth}, where the design matrix $\mathbf{X}$ is assumed to be non random. In this case we assume that in the data generating model $\mathbf{y}=\bx \bbst +\boldsymbol{\vep}^\star$ the true coefficient vector ${\bm \beta}^\star$ is generated as an i.i.d vector from some prior $\mu^\star$ (potentially different from $\mu$). In this case, Lemmas \ref{lem:beta concentration} and \ref{lem:lincomb expectation Bayes} play the analogues of Lemmas \ref{lem:random design properties} and \ref{lem:lincomb expectation random design} respectively.

The above discussion applies to the CLT of ${\bf q}^\top \bbeta$ under the posterior. To analyze the Bayes estimator $\delta({\bf y},{\bf X})=\EE_{\nu_{{\by},{\bx}}}{\bf q}^\top {\bm \beta}$, we use Proposition 2.4 to first get the approximation $\delta({\bf y},{\bf X})\approx {\bf q}^\top {\bf u}$ in terms of ${\bf u}$, the solution to the mean field variational problem (see \eqref{eq:fpeq}). Noting that ${\bf u}$ is a function of ${\bf y}$ which is multivariate Gaussian (given $\mathbf{X},{\bm \beta}^\star)$, we then derive a CLT for ${\bf q}^\top {\bf u}$. For this goal, we first use the fixed point equation \eqref{eq:fpeq} to approximate ${\bf q}^\top {\bf u}$ by a more tractable object, and then use a CLT for non-linear functions of dependent Gaussians (the so-called ``second-order'' Poincar\'{e} inequality, see \cite[Theorem 2.2]{chatterjee2009fluctuations}). Along the way, we also use the usual Gaussian Poincar\'{e} inequality (see \cite[Corollary 2.27]{van2014probability}) to control various error terms.

    We move on to \cref{cor:coverage probability Bayes}, for which we need to study the limit distribution of ${\bf q}^\top ({\bf u}-{\bm \beta}^\star)$, under a potentially misspecifed prior. By again leveraging the approximation $\delta(\by,\bx)\approx \bq^{\top}\bu$, we express 
    ${\bf q}^\top ({\bf u}-{\bm \beta}^\star)$ as a sum of two terms, a correctly centralized version of the Bayes estimator $\delta({\bf y},{\bf X})$, and a bias term which is (approximately) a separable function of ${\bm \beta}^\star$. The first term has a Gaussian limit, using the CLT for Bayes estimators in \cref{cor:regression Bayesian truth}, part (b), and the second term has a CLT as ${\bm \beta}^\star$  has marginally i.i.d coordinates (coupled with the Lindeberg-Feller CLT). Combining the asymptotics for the two terms, the CLT for ${\bf q}^\top ({\bf u}-{\bm \beta}^\star)$ follows, from which we can readily compute the limiting coverage of the proposed credible interval explicitly.

\section{Discussion}\label{sec:conclusion}
This work provides limiting distributions of (i) one-dimensional projections of the posterior $T(\bbeta)=\bq^{\top}\bbeta$ for some $\bq\in\R^p$, $\lVert \bq\rVert=1$, and (ii) the posterior mean $\delta(\by, \bx)$. Our results are novel as we establish limiting distributions under a challenging high-dimensional asymptotic regime without posterior contraction. We also de-mystify uncertainty quantification of Mean-Field variational inference, by providing explicit conditions for its validity as well as illustrating the possibility of both over/under estimating the variance. Finally, we construct credible intervals for linear projections and characterize their asymptotic coverage under model misspecifications. Our result demonstrates that Bayesian procedures can enjoy true coverage under the low SNR scaling in an average sense, under the correct prior specification. 

We believe this work creates many interesting questions for future research. First, it would be important to improve on our technical conditions --- (a) extending beyond the high-temperature regime (see \cref{assn:ht}), (b) accommodating non-compactly supported priors $\mu$, (c) studying CLTs for linear projections in the white noise design setting in the $n^{2/3}\ll p\ll n$ regime, etc. We believe that it would be meaningful to go beyond these assumptions and derive limiting distributions under a more flexible setup.

Another intriguing direction is to derive limiting distributions and credible intervals in generalized linear models in a low SNR regime. Currently, our proof technique relies on the quadratic nature of the posterior $\nu_{\by,\bx}$. While there have been some recent developments in establishing consistency-type results in such settings \citep[see][]{mukherjee2024naive}, studying asymptotic distributions requires considerably finer analysis. 

Finally, it would be interesting to think beyond the subjective Bayesian framework considered in this paper (see \eqref{eq:Bayeslin}), and consider an empirical Bayes setting where the prior $\mu$ is estimated from the data $({\bf y},{\bf X})$. In particular, it would be interesting to extend the coverage guarantees in \cref{cor:coverage probability Bayes} to analogous statements for credible intervals built using an estimated prior $\hat{\mu}(\by,\bx)$. While consistent estimators for the data generating prior $\mu$ in regression problems has been recently established \citep{mukherjee2023mean,fan2023gradient}, our current proof does not immediately generalize due to double-dipping of $\by$ (once for sampling from the posterior and again for constructing $\hat{\mu}(\by,\bx)$), and we leave this question for future research.

\begin{acks}[Acknowledgments]
    The authors would like to thank Jaeyong Lee and Debdeep Pati for helpful discussions. 
\end{acks}

\begin{funding}
S.M.~was partially supported by NSF Grant DMS-2113414.
\end{funding}

\bibliographystyle{imsart-nameyear.bst} %
\bibliography{reference} 

\newpage
\section*{Appendix}

\begin{appendices}

We prove our main results from Sections \ref{sec:bayeslinhigh} and \ref{sec:Bvmanalog} in \cref{sec:pfbayeslinhigh}. The proofs in \cref{sec:pfbayeslinhigh} require some lemmas regarding concentration and limiting distributions for dependent random variables, and we prove them in \cref{sec:pfpfblh}. 
\cref{prop:regret} and some general weak convergence and calculus related lemmas are proved in \cref{sec:pfauxlem}.

\section{Proof of results in Sections \ref{sec:bayeslinhigh} and \ref{sec:Bvmanalog}}\label{sec:pfbayeslinhigh}
We separate this appendix into five parts. \cref{sec:clt} presents three CLTs that we will use to for our main results. \cref{sec:notation} introduces notations and common lemmas. Now, \cref{sec:pfwnd} proves all theoretical results from \cref{sec:wnd}, and \cref{sec:pftrueBayes} proves all claims from \cref{sec:btm}. Finally, \cref{cor:coverage probability Bayes} is proved in \cref{sec:pfBvmanalog}.

\subsection{Preliminaries: Central Limit Theorems}\label{sec:clt}
Asymptotic distribution theory for linear projections under quadratic interaction models with bounded SNR is relatively unexplored in the literature; with a few exceptions such as \cite{deb2020fluctuations,chatterjee2011nonnormal,Deb2024detecting,Kabluchko2022}. However these results neither allow a general compactly supported $\mu$, nor a non-constant field vector $\bc$ (see \eqref{eq:posterior_2}). 

Our proof arguments build upon recently developed finite-sample CLTs for (i) Gibbs measures with quadratic interactions \citep{lee2025rfim}, (ii) nonlinear transforms of Gaussian random variables \citep{chatterjee2009fluctuations}. Mainly, we use two results from \cite{lee2025rfim} to understand the fluctuations of $T(\bbeta)$, and another result from \cite{chatterjee2009fluctuations} to understand the fluctuations of $\delta(\by, \bx)$. These bounds are stated in terms of the following distances for random variables.

\begin{defi}[KS and BL distances]\label{def:KS distance}
    For random variables $X$ and $Y$ supported on $\R$, the Kolmogorov-Smirnov (KS) distance between their distributions is defined as:
    \begin{align}\label{eq:ksdist}
    d_{KS}(X,Y)=\sup_{t\in\R} \big| P(X\le t) - P(Y\le t)\big|.
    \end{align}
    On the other hand, the bounded-Lipschitz (BL) distance (see \cite[Section 8.3]{Bogachev2007}) between them is defined as: 
    \begin{align}\label{eq:bldist}
    d_{BL}(X,Y)=\sup\left\{\EE f(X)-\EE f(Y):\, \lVert f\rVert_{\infty}\le 1, \sup_{z_1\neq z_2}\frac{|f(z_1)-f(z_2)|}{|z_1-z_2|}\le 1\right\}.
    \end{align}
\end{defi}

We first state a finite-sample Berry-Esseen bound for the Gaussian approximation of $T(\bbeta)$. Recall that $\bu$ denotes the Mean-Field optimizers, and the notation $\upsilon_p$ from \cref{def:upsilon_p}.

\begin{lemma}[Theorem 2.4 in \cite{lee2025rfim}]\label{thm:clt}
    Suppose that $\bbeta\sim \nu_{\by,\bx}$ and assume the high-temperature \cref{assn:ht}. 
    Define $t_i := \sumjn \A_p(i,j) \psi_j'(c_j)$ for all $i= 1,\ldots,p$,
    and define the following:
    \begin{equation}\label{eq:random field}
       R_{1p} := \sumin \left(\sumjn \A_p(i,j) q_j (\psi_j''(c_j) - \upsilon_p)\right)^2,\qquad  %
        R_{2p} := \sumin t_i^2,\qquad   R_{3p}:=\sumin t_i^4.
    \end{equation}
    Fix $\lambda_p\in \R$, and set $\boldsymbol{\epsilon} := \A_p \bq - \lambda_p \bq$. Suppose $W_p \sim \mcn \left(0, \frac{\upsilon_p}{1 - \upsilon_p  \lambda_p } \right)$. We then have
    \begin{align}\label{eq:ksbd}
    &\;\;\;d_{KS}\left(T(\bbeta) -\bu^\top \bbeta \, ,\, W_p \mid \by, \bx \right)\lesssim \sqrt{R_{1p}} + \sqrt{\an R_{2p}} +  \frac{\sqrt{R_{3p}} + \sqrt{p} \an + \lVert \bq\rVert_{\infty}}{\upsilon_p} + \lVert \boldsymbol{\epsilon} \rVert_2 .
    \end{align}
    The implicit constant only depends on $\rho$ from \eqref{eq:mht}.
\end{lemma}
Note that the error terms in the above Berry-Esseen type bound are entirely \emph{data-driven}, in the sense that any choice of $\by, \bx$ are allowed. As stated in the Remarks 2.1-2.3 of \cite{lee2025rfim}, the following conditions are required for the RHS to be $o(1)$, and the bound to imply convergence in distribution:
\begin{itemize}
    \item (Delocalized eigenvector) $(\lambda_p, \bq)$ must satisfy $\|\bq\|\to 0$ and $\|\boldsymbol{\epsilon}\| = \|\A_p \bq - \lambda_p \bq\| \to 0$.
    \item (Strong Mean-Field) \cref{assn:mf} must hold.
    \item (Additional error terms) $R_{1p}, \an R_{2p}, R_{3p} = o(1),$ and $\upsilon_p$ is bounded away from $0$.
\end{itemize}

Under mildly stronger assumptions, the following result further simplifies the Mean-Field centering $\bq^\top \bu$ to a more explicit function.
   
\begin{lemma}[Theorem 2.5 in \cite{lee2025rfim}]\label{thm:CLT for u}
    Suppose $\bbeta\sim \nu_{\by,\bx}$ and consider the same notation as in \cref{thm:clt}. 
    Suppose that the high-temperature \cref{assn:ht} holds, and recall that $\boldsymbol{\epsilon} = \A_p \bq - \lambda_p \bq$. Set $\Psi'(\bc) := \left(\psi_1'(c_1), \ldots, \psi_p'(c_p) \right)$ and 
    $$R_{4p} := \left|\sum_{1 \le i,j \le p} \A_p(i,j) q_i(\psi_i''(c_i) - \upsilon_p) \psi_j'(c_j) \right|.$$
    \begin{enumerate}
        \item[(a)]
    Under the above notations, We have
    \begin{align}\label{eq:CLT for u}
         \bigg|\bq^\top \left(\bu - \frac{\Psi'(\bc)}{1 - \lambda_p \upsilon_p} \right)\bigg| \lesssim \sqrt{R_{1p} R_{2p}} + \sqrt{R_{3p}} + R_{4p} + \sqrt{R_{2p}}\lVert \boldsymbol{\epsilon} \rVert + \big|\boldsymbol{\epsilon}^{\top}\Psi'(\bc)\big|.
    \end{align}

 \item[(b)] Additionally, we have
    \begin{align}
    &\;\;\;\;d_{KS}\left(\sumin q_i\left(\beta_i-\frac{\psi_i'(c_i)}{1-\lambda_p\upsilon_p} \right)\, , \, W_p \mid \by, \bx  \right)\nonumber\\ &\lesssim \sqrt{R_{2p}}\left(\sqrt{R_{1p}} + \sqrt{\alpha_p} + \lVert \boldsymbol{\epsilon}\rVert\right) + \sqrt{R_{1p}} +  \frac{\sqrt{R_{3p}}  + \sqrt{p} \an + \lVert \bq\rVert_{\infty}}{\upsilon_p}+ R_{4p} +\big|\boldsymbol{\epsilon}^{\top}\Psi'(\bc)\big| .\label{eq:erbdcltforu}
    \end{align}
    In both cases, the implicit constant only depends on $\rho$.
    \end{enumerate}
\end{lemma}

Part (a) of the above theorem allows us to approximate $\bq^\top \bu$ by a more explicit linear term involving $\bq^\top \Psi'(\bc)$. Our final general Lemma provides finite sample error bounds for the Normal approximation of $\bq^\top \Psi'(\bc)$.

\begin{lemma}\label{lem:remove centering}
    Suppose that the data generating \cref{assn:dgp} holds for $\by$. Recall the definitions of $R_{1p},\ldots ,R_{4p}, \upsilon_p$ in lemmas \ref{thm:clt} and \ref{thm:CLT for u} (note that these terms are all functions of $\by,\bx$). Assume that  
    \begin{align}\label{eq:varlbd}
    \Var\left(\sumin q_i \psi_i'(c_i)|\bx,\bbst\right)\ge \kappa, 
    \end{align}
    for some $\kappa>0$. Then, under the SMF \cref{assn:mf}, we have 
    \begin{align*}
    &\;\;\;\;d_{BL}\left((1-\lambda_p\upsilon_p)\delta(\by,\bx)-\sumin q_i\EE [\psi_i'(c_i)|\bx,\bbst],\widetilde{W}_p \mid \bx, \bbst \right)\nonumber \\ &\lesssim \sqrt{p}\an+\lVert \bq\rVert_{\infty} + \EE\left[\left(\sqrt{R_{1p} R_{2p}} + \sqrt{R_{3p}} + R_{4p} + \sqrt{R_{2p}}\lVert \boldsymbol{\epsilon} \rVert + \big|\boldsymbol{\epsilon}^{\top}\Psi'(\bc)\big|\right) \mid \bx,\bbst\right],
    \end{align*}
    where $\widetilde{W}_p\sim \mcn(0,\Var(\sumin q_i\psi_i'(c_i)|\bx,\bbst))$. In the above display, the law of $\delta(\by,\bx)$ is conditional on $\bx,\bbst$ and the implied constant depends on $\rho$ from \eqref{eq:mht} and $\kappa$ from \eqref{eq:varlbd}. 
\end{lemma}

\subsection{Notations and common lemmas}\label{sec:notation}
 For the rest of this paper, $\EE$ will denote expectation conditioned on $\bx$, $\bbst$. We will work under \cref{assn:homogeneous diagonals} throughout this Section. We will prove our main results from \cref{sec:bayeslinhigh} by applying Lemmas \ref{thm:clt}, \ref{thm:CLT for u}, and \ref{lem:remove centering}. 
 To achieve this, we need to show that the corresponding error bounds converge to $0$ in probability. First, we provide some general bounds that are valid under \eqref{eq:truemodel} and will be used to cover both of our main examples from \cref{sec:bayeslinhigh}. The main technical challenge arises from the dependent nature of the random fields $\bc = \frac{\bx^\top \by}{\sigma^2}$ in \eqref{eq:posterior}, as 
\begin{align}\label{eq:truec}
    \bc \mid \bx, \bbeta^\star \sim \mcn \Big(\Sigman \bbeta^\star, \Sigman\Big),\quad \text{ where }\quad \Sigman=\frac{\mathbf{X}^\top\mathbf{X}}{\sigma^2},
\end{align}
 under model \eqref{eq:truemodel}. Our strategy to deal with this dependence is to define a surrogate sequence of independent random variables $\bc^\star$ as follows: 
\begin{defi}\label{def:c_approx}
Recall the definition of $d_i$ from \eqref{eq:posterior} and the definition of $d_0$ from \cref{assn:homogeneous diagonals}.
Set $$\mathfrak{E}_1 := \sumin (d_i - d_0)^2, \quad \mathfrak{E}_2 := d_{\text{min}}^{-1}.$$
For $1\le j \le p$, let $\bc^{\star}:=(c^{\star}_1,\ldots ,c^{\star}_p)$ be independent random variables, each with distribution \begin{align}\label{eq:c_approx}
c^{\star}_j \mid \bx, \bbeta^\star \sim \mcn\Big(d_0 \beta_j^\star - \mathfrak{F}_j, d_0 \Big), \quad \text{where} \quad \mathfrak{F}_j:=\sumin \A_p(j,i) \beta_i^\star.
\end{align}
\end{defi}

We will provide bounds for moments of appropriate functions under $\bc$ in terms of moments for the same functions under $\bc^{\star}$. The following lemma controls the second moment of linear combinations of $\{\psi_j'(c_j)\}_{1\le j\le p}$ and $\{\psi_j''(c_j)\}_{1\le j\le p}$ (recall $\psi_j$ from \cref{def:expfam}). For all lemmas that will follow in this subsection, the implicit constants that arise in the $\lesssim$ symbol will depend only on $d_0$ and upper bound on $\mathfrak{E}_2$. 

\begin{lemma}\label{lem:lincombound}
    Consider the regression setting from \cref{sec:bayeslinhigh}, and let $\boldsymbol{\gamma} := (\gamma_1, \ldots, \gamma_p)$ be a deterministic vector (potentially depending on $\bx$ and $\bbeta^\star$). Recall the definition of $\psi_0(\cdot)$ from \cref{assn:homogeneous diagonals}. For $a\in \{1,2\}$ we have
    \begin{align}\label{eq:lincombound}
    \Big| \EE\Big[\sumjn \gamma_j \psi_j^{(a)}(c_j)\Big] - \EE\Big[\sumjn \gamma_j\psi_0^{(a)}(c^{\star}_j)\Big]\Big| &\lesssim \|\boldsymbol{\gamma}\| \sqrt{\mathfrak{E}_1}, \quad
    \Var\Big[\sumjn \gamma_j \psi_j^{(a)}(c_j)\Big] \lesssim \lVert \bgamma\rVert^2 \|\Sigman\|,
    \end{align}
    where the superscript $(a)$ denotes the $a$-th derivative.
    Consequently, the following holds:
    \begin{enumerate}[(a)]
        \item $\bigg|\EE\Big[\sumjn \gamma_j \psi_j'(c_j) \Big]^2 - \Big[ \sumjn \gamma_j \EE\psi_0'(c^{\star}_j)\Big]^2\bigg| \lesssim \|\boldsymbol{\gamma}\|^2 (\mathfrak{E}_1 + \|\Sigman\|).$
        \item $\bigg|\EE\Big[\sumjn \gamma_j \big(\psi_j''(c_j) - U\big) \Big]^2 - \Big[\sumjn \gamma_j \big( \EE\psi_0''(c^{\star}_j) - U\big) \Big]^2\bigg| \lesssim  \|\boldsymbol{\gamma}\|^2 (\mathfrak{E}_1 + \|\Sigman\|).$
    \end{enumerate}
    Here, $U$ is an arbitrary constant, and all expectations are conditional on $(\mathbf{X},{\bm \beta}^*)$. 
\end{lemma}
\noindent The choice of $U$ in part (b) of \cref{lem:lincombound} is flexible and will be typically chosen as $U = \upsilon$. The above lemma is useful because our main error terms $R_{1p},\ldots ,R_{4p}$ are all functions of $\psi_j'(c_j)$s. Note that $R_{1p}$ involves the second moment of $\psi_j''(c_j)$s and $R_{2p}$ involves the second moment of $\psi_j'(c_j)$s. So \cref{lem:lincombound} is directly applicable to bound them. On the other hand $R_{3p}$ and $R_{4p}$ are more complicated. General upper bounds for them are provided in the following lemma.

\begin{lemma}\label{lem:R_{1p} to R_{4p} bound}
     Consider the linear regression model from \cref{sec:bayeslinhigh} and suppose Assumptions \ref{assn:ht} and \ref{assn:homogeneous diagonals} hold. Then we have
    \begin{enumerate}[(a)]
        \item $\EE R_{3p} \lesssim p \alpha_p^2 (\mathfrak{E}_1^2+\|\Sigman\|^2) + \sumin \Big[\sumjn \A_p(i,j) \EE \psi_0'(c^{\star}_j)\Big]^4$,
        \item For any constant $U$, we then have
        \begin{align*}
            \big[ \EE &  R_{4p}\big]^2 \lesssim \Big[\sumij \A_p(i,j) q_i (\EE \psi_0''(c_i^\star) - U) \EE \psi_0'(c_j^\star)\Big]^2 \\
            &+ \alpha_p^2 \Big((1+\mathfrak{E}_1) \|\bbst\|^2 + p(1+\mathfrak{E}_1) \Big) + \|\Sigman\| \EE \Big[{R}_{1p} + \sqrt{R_{3p}} \Big] \\
            &+ \mathfrak{E}_1\Bigg(\sumin q_i^2 \Big[\sumjn \A_p(i,j) \EE \psi_0'(c^{\star}_j) \Big]^2 + \sumjn \Big[\sumin \A_p(i,j)q_i(\EE \psi_0''(c^{\star}_i) -U)\Big]^2 + \an \Bigg)\\
            &+ \EE(\upsilon_p - U)^2 \Bigg(\Big[\sumij \A_p(i,j) q_i \EE \psi_0'(c^{\star}_j) \Big]^2 + \mathfrak{E}_1 + \|\Sigman\|\Bigg).
        \end{align*}
    \end{enumerate}
\end{lemma} 
Recall that $\EE$ denotes expectation condition on $\bx,\bbeta^\star$. Our final preparatory lemma provides a simplified expression for $\Var(\sumin q_i\psi_i'(c_i))$ which in turn helps obtain more tractable asymptotic distributions from \cref{lem:remove centering}.

\begin{lemma}\label{lem:remove centering two}
    Consider the linear regression model from \cref{sec:bayeslinhigh}, and suppose that \cref{assn:homogeneous diagonals} holds. 
    Define 
    \begin{align}\label{eq:vthetan}
    \varsigma^2_p := \sumin q_i^2 \Var(\psi_0'(c^{\star}_i)) - \sum_{i \neq j} \A_p(i,j) q_i q_j \EE \psi_0''(c^{\star}_i) \EE \psi_0''(c^{\star}_j),
    \end{align}
    Then, we have 
 
    \begin{align*}
        &\;\;\;\;\bigg|\Var\left(\sumin q_i\psi_i'(c_i)|\bx,\bbst\right)-\varsigma^2_p\bigg| \\ & \lesssim \lVert \bq\rVert_{\infty}\bigg(\sqrt{\mathfrak{E}_1}+\an(1+\sqrt{\mathfrak{E}_1}) \lVert \bbst\rVert + \sqrt{\sumin (d_i-d_0)^2(\beta_i^\star)^2}+\an(1+\mathfrak{E}_1)\sqrt{\sumin (\beta_i^\star)^4}\bigg) \\ & \qquad +\an\big(1+\sqrt{\mathfrak{E}_1}+\mathfrak{E}_1+\an\lVert\bbst\rVert^2\big).
    \end{align*}

\end{lemma}

\noindent We also introduce some additional notation that will help simplify proofs.
\begin{defi}\label{def:allnot}
For $m\in\R$, define functions
$$\phi_1(m) := \EE \psi_0'(\mcn(m,d_0)), \qquad \phi_2(m) := \EE \psi_0''(\mcn(m,d_0)),$$ 
where $d_0$ is defined in \cref{def:c_approx}. Recalling that $c^{\star}_j \mid \bx, \bbeta^\star \sim \mcn\Big(d_0 \beta_j^\star - \mathfrak{F}_j, d_0 \Big), \text{ where }\mathfrak{F}_j=\sumin \A_p(j,i) \beta_i^\star$ (see \eqref{eq:c_approx}), we note the following equalities: 
\begin{equation}\label{eq:tocite}
\EE \psi_0'(c_j^\star)=\phi_1(d_0\beta_j^\star-\mathfrak{F}_j), \quad \EE \psi_0''(c_j^\star)=\phi_2(d_0\beta_j^\star-\mathfrak{F}_j).
\end{equation}
\end{defi}

\subsection{CLTs under random design}\label{sec:pfwnd}
In this Section, we prove the CLTs under random design. We also set $d_0:=\sigma^{-2}$ and recall the definition of $\mathfrak{E}_1$ from \cref{def:c_approx}. In addition to the aforementioned general results, we will need some properties of random matrices as stated below. In this Section, all hidden constants under the $\lesssim$ symbol depend only on the sub-Gaussian norm of the design matrix entries, and $\sigma$. To distinguish from the expectations over the random fields $\bc \mid \bx, \bbst$, we write $\EE_{\bx}$ as the expectation over $\bx$. We also write $\BOPX(1)$ and $\OPX(1)$ to capture the randomness of $\bx$.

\begin{lemma}\label{lem:L_p norm}
    Let $\mathbf{Z}$ be a $n \times p$ random matrix whose entries are iid sub-Gaussians with mean zero and variance $1$, with $p\le n$. Then, the following bound holds for $\M_p := \frac{1}{n} \Off(\mathbf{Z}^\top \mathbf{Z})$:
    \begin{equation}\label{eq:l_p norm bound}
        \|\M_p\| = O_{P,\bx}\Big(\sqrt{\frac{p}{n}}\Big), \quad \|\M_p\|_{4} = O_{P,\bx}\Big(\frac{p^{3/4}}{\sqrt{n}}\Big).
    \end{equation}
\end{lemma}

\begin{lemma}\label{lem:alpha_n}
    Consider the same setting as in \cref{lem:L_p norm}, still with $p \le n$. %
    Then, we have $$\max_{i=1}^p \sumjn \M_p(i,j)^2 = O_{P,\bx}(p/n).$$
\end{lemma}

\noindent Lemmas \ref{lem:L_p norm} and \ref{lem:alpha_n} directly follow from the more general claims in [Lemmas 4.1 and 4.2, \cite{lee2025rfim}], by setting $\mathbf{G}_n(i,j) = I(i \neq j)$ there (i.e. considering no dilution in the coupling matrix). %

\begin{lemma}\label{lem:random design properties}
Let $\bx$ be the random design defined in \cref{cor:regression i.i.d design}. Recall the definition of $\mathfrak{F}_i$s from \eqref{eq:c_approx}. Then, for any deterministic vector $\bv, \bw \in \mathbb{R}^p$ (potentially depending on $\bbeta^\star$), the following holds:
\begin{enumerate}[(a)]
    \item $\EEX [\bv^\top \A_p \bw]^2 \lesssim \frac{\|\bv\|^2 \|\bw\|^2}{n}$
    \item For any $1\le i\le p$, $\EEX \Big[\sumjn \A_p(i,j) v_j \Big]^2 \lesssim \frac{\|\bv\|^2}{n}$ and in particular, $\EEX \mathfrak{F}_i^2 \lesssim \frac{p}{n}$
    \item For any $1\le i\le p$, $\EEX \Big[\sumjn \A_p(i,j) v_j \Big]^4 \lesssim \frac{\|\bv\|^4}{n^2}$ and in particular, $\EEX \mathfrak{F}_i^4 \lesssim \frac{p^2}{n^2}$
    \item For any $1\le i \le p$ and $\|\bv\|_{\infty} \le 1$, $\sumin \Big[\sumjn \A_p(i,j) \mathfrak{F}_j v_j \Big]^4 = O_{P, \bx} \Big(\frac{p^6}{n^4}\Big)$.\footnote{Here, we conjecture that the tight upper bound is actually $\frac{p^5}{n^4}$, but our weaker result is much simpler to prove and suffices for our purposes.}
\end{enumerate}
\end{lemma}

The following lemma simplifies the expectations in \eqref{eq:lincombound}, under the random design.
\begin{lemma}\label{lem:lincomb expectation random design}
Let $\bx$ be the random design defined in \cref{cor:regression i.i.d design}. Suppose Assumptions \ref{assn:ht} and \ref{assn:mf} hold.
Then, for any deterministic vector $\boldsymbol{\gamma} = (\gamma_1,\ldots ,\gamma_p) \in \mathbb{R}^p$ (potentially depending on $\bbeta^\star$) and $\mathcal{C}^2_b$ function $\Phi:\R \to [-1,1]$, the following holds:
    \begin{align}\label{eq:lincomb moment bound}
        \EE_{\bx} \Big[\Big[\sumin \gamma_j \big(\Phi(d_0\beta_j^\star - \mathfrak{F}_j) - \Phi(d_0\beta_j^\star) \big) \Big]^2 \Big] \lesssim \|\bgamma\|^2 \Big(\frac{p}{n} + \frac{p^3}{n^2}\Big).
    \end{align}
    In particular, under the asymptotic regime $p \ll n^{2/3}$ and $a = 1,2$, we have
    \begin{align}\label{eq:lincomb expectation}
        \EE_{\bx} \Big[\sumjn \gamma_j \Big(\EE [\psi_0^{(a)}(c^{\star}_j)|\bx,\bbst] - \phi_a (d_0 \beta_j^\star) \Big) \Big]^2 = o(\|\bgamma\|^2).
    \end{align}
\end{lemma}

\begin{proof}[Proof of \cref{cor:regression i.i.d design}]
  Recall the definition of $\upsilon$ from \eqref{eq:regression i.i.d convergence assumption upsilon}, and $\A_p$ from \eqref{eq:posterior}. 

\noindent (a) 
We prove the two limiting distributions in \eqref{eq:iid clt 1} by applying Lemmas \ref{thm:clt} and \ref{thm:CLT for u} with $\lambda_p = 0$. The limit of $\upsilon_p$ will be argued along the way (see the last term in \eqref{eq:suffice1}).
To ensure that both RHS in the Lemmas are $o_P(1)$, it suffices to show the following with high probability (with respect to $\bx$): 
\begin{equation}\label{eq:basicbounds1}
\lVert \A_p\rVert_4 =\OPX(1), \quad \alpha_p = \max_{1\le j\le p} \sumin \A_p(i,j)^2 = \OPX\Big(\frac{1}{\sqrt{p}}\Big), 
\end{equation}
and 
\begin{equation}
\begin{aligned}\label{eq:suffice1}
 \upsilon_p - \upsilon = \OPX(1), \,\, &R_{1p}=\OPX(1), \,\, \an R_{2p} =\OPX(1), \,\, R_{3p}=\OPX(1), \\ R_{2p} \|\boldsymbol{\epsilon}\|^2 = \OPX(1), \,\, &R_{1p}R_{2p}=\OPX(1), \,\, R_{4p}=\OPX(1), \,\, |{\bm \epsilon}^\top \Psi'({\bf c})|=\OPX(1).
\end{aligned}
\end{equation}
Here, the first and second line of \eqref{eq:suffice1} come from Lemmas \ref{thm:clt} and \ref{thm:CLT for u}, respectively.

 The first condition (high-temperature) in \eqref{eq:basicbounds1} follows from \cref{lem:L_p norm}, and the second (Strong Mean-Field) from \cref{lem:alpha_n}. 
 
Now, we prove each bound in \eqref{eq:suffice1}. For this purpose, we use the bounds 
\begin{align}\label{eq:newcall}
\mathfrak{E}_1 = \sumin (d_i-d_0)^2 = O_{P,\bx}\Big(\frac{p}{n}\Big), \quad \mbox{and} \quad \|\Sigman\| = O_{P,\bx}(1), 
\end{align}
multiple times. The above claim follows by noting that $d_i = \frac{1}{\sigma^2} \sum_{k=1}^n X_{k,i}^2$, $d_0 = \frac{1}{\sigma^2}$ and computing the expectation of $\mathfrak{E}_1$:
\begin{equation}\label{eq:newcall1}
\EE_{\bx} \mathfrak{E}_1 = \EE_{\bx} \Big[\sumin (d_i - d_0)^2\Big] = \frac{p}{\sigma^4} \Var_{\bx} \left(\sum_{k=1}^n X_{k,1}^2 \right) \lesssim \frac{p}{n}.
\end{equation}
Also, $\lVert\Sigman\rVert\le \lVert\A_p\rVert+\max_{i=1}^p d_i=\BOPX(1)$ from the above display and the fact that $\|\A_p\| = O_{P,\bx}\Big(\sqrt{\frac{p}{n}}\Big)$, invoking \cref{lem:L_p norm}. We are now in position to prove \eqref{eq:suffice1}.

\begin{itemize}
    \item To show $\EE(\upsilon_p- \upsilon)^2 =\OPX(1)$, we observe that by applying part (b) of \cref{lem:lincombound} with $\gamma_j = q_j^2, U = \upsilon$, we have
        \begin{align}
            \EE \Big[\sumjn q_j^2 \psi_j''(c_j) - \upsilon \Big]^2  &\lesssim \sumjn q_j^4 (\|\Sigman\| + \mathfrak{E}_1) + \Big[\sumjn q_j^2 \EE\psi_0''(c_j^{\star}) - \upsilon \Big]^2 = \OPX(1). \notag
        \end{align}
        Here, the first error term is $\OPX(1)$ by \eqref{eq:newcall} and the fact that $\sumin q_i^4\le \lVert \bq\rVert_{\infty}^2=o(1)$. For the second error term, with $\{\phi_a(\cdot)\}_{a\in \{1,2\}}$ as in \cref{def:allnot}, note that 
        $$\Big[\sumjn q_j^2 \EE\psi_0''(c_j^{\star}) - \upsilon \Big]^2 \lesssim \Big[\sumjn q_j^2(\EE\psi_0''(c_j^{\star})-\phi_2(d_0\beta_j^{\star}))\Big]^2+\big(\upsilon-\sumjn q_j^2 \phi_2(d_0\beta_j^{\star})\big)^2.$$
        The first term is $\OPX(1)$ by applying \cref{lem:lincomb expectation random design} (see \eqref{eq:lincomb expectation}) with $a=2$ and the second term is $\OPX(1)$ from the definition of $\upsilon$ in \eqref{eq:regression i.i.d convergence assumption upsilon} (as $\phi_2(d_0\beta_j^{\star})=\EE\psi_0''(d_0\beta_j^{\star}+W)$ with $W\sim \mcn(0,\sigma^{-2})$.

    \item The bound for $R_{1p}$ follows by defining $\tilde{\bq}$ as $\tilde{q}_i := q_i (\psi_i''(c_i)-\upsilon_p)$ and noting that $|\tilde{q}_i|\le |q_i|$, which gives
    \begin{align*}
        R_{1p} = \|\A_p \tilde{\bq}\|^2 \le \|\A_p\|^2 \|\tilde{\bq}\|^2 \le \|\A_p\|^2 = \BOPX\left( \frac{p}{n}\right).
    \end{align*}
    
    \item The bound for $R_{2p}$ follows similarly, as
    \begin{align*}
        R_{2p} = \|\A_p \Psi'(\bc)\|^2 \le \|\A_p\|^2 \|\Psi'(\bc)\|^2 = \BOPX\left(\frac{p^2}{n}\right).
    \end{align*} 
    Here, recall the notation $\Psi'(\bc)$ from \cref{thm:CLT for u}. In both the above displays, we have also used $\lVert \A_p\rVert=\BOPX(\sqrt{p/n})$ which follows from \cref{lem:L_p norm}. 

    \item We note that 
    \begin{align}\label{eq:epbd}
    \|\boldsymbol{\epsilon}\|=\|\A_p{\bf q}\|\le \|\A_p\|=\BOPX(\sqrt{p/n})
    \end{align}
    by \cref{lem:L_p norm}. The bounds for $R_{1p}R_{2p}$, $\an R_{2p}$, and $R_{2p}\|\boldsymbol{\epsilon}\|^2$ in \eqref{eq:suffice1} now follow from the individual bounds for $R_{1p}$, $R_{2p}$, $\an$, and $\|\boldsymbol{\epsilon}\|$ obtained above.

    \item To bound $R_{3p}$, we use part (a) of \cref{lem:R_{1p} to R_{4p} bound}, along with \eqref{eq:basicbounds1} and \eqref{eq:newcall} to get
    \begin{align*}
        \EE R_{3p} &\lesssim \sumin \Big(\bigg[\sumjn \A_p(i,j) \EE\psi_0'(c^{\star}_j)\bigg]\Big)^4 + \OPX(1).
        \end{align*}
        Using \eqref{eq:tocite} followed by a second order Taylor's expansion we get
        \begin{align}\label{eq:twotermder}
            \EE\psi_0'(c^{\star}_j)=\phi_1(d_0\beta_j^\star-\mathfrak{F}_j)=\phi_1(d_0\beta_j^\star)-\mathfrak{F}_j\phi_1'(d_0\beta_j^\star)+O(\mathfrak{F}_j^2).
        \end{align}
        Using this, the first term in the bound for $\EE R_{3p}$ can be bounded, upto constants, by
        \begin{align}\label{eq:r_3 i.i.d expression}
        \notag& \sumin \Big(\sumjn \A_p(i,j) \phi_1(d_0 \beta_j^\star) \Big)^4 + \sumin \Big(\sumjn \A_p(i,j) \mathfrak{F}_j \phi_1'(d_0 \beta_j^\star) \Big)^4 + \\ & \quad O\Big(\sumin \Big(\sumjn |\A_p(i,j)| \mathfrak{F}_j^2 \Big)^4 \Big).
    \end{align}
    It thus suffices to show that all three terms in the RHS of \eqref{eq:r_3 i.i.d expression} are $\OPX(1)$.
    For the first term in \eqref{eq:r_3 i.i.d expression}, applying part (c) of \cref{lem:random design properties} gives the bound $O_{P,\bx}(\frac{p^3}{n^2})$. The second term is $O_{P,\bx}(\frac{p^6}{n^4})$ by applying part (d) of the same lemma. Here, we use the uniform bounds for $\phi_1(\cdot), \phi'_1(\cdot)$ and $\phi''_1(\cdot)$. The third term can be bounded by using Cauchy-Schwartz to write
    $$\sumin \Big(\sumjn |\A_p(i,j)| \mathfrak{F}_j^2 \Big)^4 \le \sumin \Big(\sumjn \A_p(i,j)^2 \Big)^2 \Big(\sumjn \mathfrak{F}_j^4 \Big)^2 \le p\alpha_p^2 \Big(\sumjn \mathfrak{F}_j^4 \Big)^2.$$
    By combining the above four displays, we get:
    $$\EE R_{3p}=\BOPX\left[\frac{p^3}{n^2}+\frac{p^6}{n^4}+p\alpha_p^2 \left(\sumjn \mathfrak{F}_j^4\right)^2\right]+\OPX(1)=\OPX(1),$$
    where the last equality follows by noting that $p \alpha_p^2 = \OPX(1)$ (see \eqref{eq:basicbounds1}) and $\sumjn \mathfrak{F}_j^4 = \BOPX(p^3/n^2)=\OPX(1)$ (see part (c) of \cref{lem:random design properties}).

    \item To show $R_{4p} \xp 0$, we apply part (b) of \cref{lem:R_{1p} to R_{4p} bound} with $U=\upsilon$ and bound each summand separately. First, by the earlier bounds for $\upsilon_p - \upsilon, R_{1p}, R_{3p}$, we have $\EE(\upsilon_p - \upsilon)^2 = \OPX(1)$ and $\EE R_{1p} + \EE \sqrt{R_{3p}} = \OPX(1).$ By assumption on $\bbst$, we have $\|\bbst\|^2 = O(p)$. 
    It remains to show
    \begin{align*}
        \sumij \A_p(i,j) q_i (\EE \psi_0''(c_i^\star) - \upsilon) \EE \psi_0'(c^{\star}_j) &= \OPX(1), \quad \sumin q_i^2 \Big[\sumjn \A_p(i,j) \EE \psi_0'(c^{\star}_j) \Big]^2 = \OPX(1), \\
        \sumjn \Big[\sumin \A_p(i,j)q_i(\EE \psi_0''(c^{\star}_i) -\upsilon)\Big]^2 &= \OPX(1), 
        \quad \Big[\sumij \A_p(i,j) q_i \EE \psi_0'(c^{\star}_j)\Big]^2 = O_{P,\bx}(1). 
    \end{align*}
    We only prove the first bound in the above display, and the others follow by similar calculations. To show the first bound in the above display, we use the identities in \eqref{eq:tocite} and the following Taylor expansions   
    \begin{align}\label{eq:taylor phi_1}
        \EE \psi_0'(c^{\star}_j) = \phi_1(d_0\beta_j^\star-\mathfrak{F}_j) = \phi_1(d_0\beta_j^\star) - \mathfrak{F}_j \phi'_1(\xi_j), 
    \end{align}
    \begin{align}\label{eq:taylor phi_2}
        \EE \psi_0''(c^{\star}_i) = \phi_2(d_0\beta_i^\star-\mathfrak{F}_i) = \phi_2(d_0\beta_i^\star) - \mathfrak{F}_i \phi'_2(\eta_i)
    \end{align}
    to write
\begin{align*}
     &\sumij \A_p(i,j) q_i (\phi_2(d_0\beta_i^\star-\mathfrak{F}_i) - \upsilon) \phi_1(d_0\beta_j^\star-\mathfrak{F}_j) \\
     = & \sumij \A_p(i,j) q_i(\phi_2(d_0\beta_i^\star)-\upsilon) \phi_1(d_0\beta_j^\star) \\
    \quad - & \sumij \A_p(i,j) q_i \Big(\mathfrak{F}_i \phi'_2(\eta_i) \phi_1(d_0 \beta_j^\star) + {\phi}_2(d_0 \beta_i^\star) \mathfrak{F}_j \phi'_1(\xi_j) - \mathfrak{F}_i \phi'_2(\eta_i) \mathfrak{F}_j \phi'_1(\xi_j) \Big). 
\end{align*}

The RHS is $\OPX(1)$ by (i) applying part (a) of \cref{lem:random design properties} to bound
$$\EE_{\bx} \Big[\sumij \A_p(i,j) q_i(\phi_2(d_0\beta_i^\star)-\upsilon) \phi_1(d_0\beta_j^\star)\Big]^2 = \BOPX\left(\frac{p}{n}\right),$$ and (ii) using $\|\A_p\| = O_{P,\bx}\Big(\sqrt{\frac{p}{n}}\Big)$ (from \cref{lem:L_p norm}) and $\EE_{\bx} \mathfrak{F}_i^2\lesssim \frac{p}{n}$ (from \cref{lem:random design properties} part (b)), to bound the quadratic form
$$\Big[\sumij \A_p(i,j) q_i \mathfrak{F}_i \phi'_2(\eta_i) \phi_1(d_0 \beta_j^\star) \Big]^2 \lesssim p \|\A_p\|^2 \sumin q_i^2 \mathfrak{F}_i^2 = O_{P,\bx} \Big(\frac{p^3}{n^2} \Big) = \OPX(1).$$
This complets the bound of the first term, as promised. The other quadratic forms in the long display above can be bounded by $O_{P,\bx} \Big(\frac{p^3}{n^2} \Big)$ via similar calculations.

\item Finally, we prove that $\boldsymbol{\epsilon}^\top \Psi'(\bc) = \sumin \epsilon_i \psi_i'(c_i) = \OPX(1)$, where $\boldsymbol{\epsilon} = \A_p \bq$. To begin, note that $\|\boldsymbol{\epsilon}\|=\BOPX\Big(\sqrt{\frac{p}{n}}\Big)$ as proved in \eqref{eq:epbd}, a fact we use several times below.
By consecutively applying part (a) of \cref{lem:lincombound}, \eqref{eq:lincomb expectation} from \cref{lem:lincomb expectation random design} (with $\boldsymbol{\gamma} = \boldsymbol{\epsilon}, ~a = 1$), and part (a) of \cref{lem:random design properties}, we have
\begin{align*}
    \EE \Big[\sumin \epsilon_i \psi_i'(c_i)\Big]^2 &= \Big(\sumin \epsilon_i \EE \psi_0'(c^{\star}_i) \Big)^2 + \BOPX \Big(\frac{p}{n}\Big) \\
    &= \Big(\sumin \epsilon_i \phi_1(d_0 \beta_i^\star) \Big)^2 + \BOPX \Big(\frac{p}{n}\Big) = \BOPX \Big(\frac{p}{n}\Big)=\OPX(1).
\end{align*}
\end{itemize}

(b) %
We prove \eqref{eq:iid clt 2}. %
Under the given assumptions, the error term in \cref{lem:remove centering two} is $\OPX(1)$, using the bounds
$$\|{\bf q}\|_\infty=o(1),\quad  \mathfrak{E}_1=\BOPX(1) ~~ (c.f.~\eqref{eq:newcall}),\quad    p\alpha_p^2=\OPX(1) ~~(c.f.~\eqref{eq:basicbounds1}), \quad \sumin (\beta_i^\star)^4 = O(p).$$
In particular, we recall $d_i = \sumkn X_{k,i}^2$ and use the bound $\EEX(d_i-d_0)^2\lesssim n^{-1}$ (see \eqref{eq:newcall1}),  to bound
$$\EEX\sumin (d_i-d_0)^2 (\beta_i^\star)^2 \lesssim \frac{1}{n}\sumin (\beta_i^\star)^2 = \BOPX\left(\frac{p}{n}\right).$$
By using the bounds in \eqref{eq:suffice1}, the error term in \cref{lem:remove centering} also vanishes and we have
\begin{align*}
    &\;\;\;\;d_{BL}\left((1-\lambda_p\upsilon_p)\delta(\by,\bx)-\EE[\sumin q_i \psi_i'(c_i)|\bx,\bbst], ~\mcn(0, \varsigma_p^2) \mid \bx, \bbst \right)\nonumber = \OPX(1).
\end{align*}
Now, the convergence in distribution \eqref{eq:iid clt 2} is immediate if the following holds:
\begin{equation*}
    \EE \big[\sumin q_i \psi_i'(c_i)\big] - \sumin q_i \phi_1(d_0\beta_i^\star) = \OPX(1), \quad \varsigma^2_p - \varsigma^2 = \OPX(1).
\end{equation*}

Here, recall from \cref{def:allnot} that $\phi_1(d_0\beta_i^{\star})=\EE\psi_0'(d_0\beta_i^{\star}+W)$ with $W\sim \mcn(0,\sigma^{-2})$.
The statement for the mean directly follows by combining the expectation bounds in (i) \cref{lem:lincombound} (with $\gamma_i = q_i$), and (ii) \eqref{eq:lincomb expectation} from \cref{lem:lincomb expectation random design} (with $a=1$).

For the second statement, recalling the definition of $\varsigma_p^2$ from \eqref{eq:vthetan}, it suffices to show
\begin{equation*}
\sumin q_i^2 \Var(\psi_0'(c^{\star}_i)) = \varsigma^2 + \OPX(1),~ \sumij \A_p(i,j) q_i q_j \EE\psi_0''(c^{\star}_i) \EE\psi_0''(c^{\star}_j) =\OPX(1).
\end{equation*}
The first conclusion follows by using \eqref{eq:lincomb moment bound} from \cref{lem:lincomb expectation random design} with $\Phi(x) := \Var\Big[\psi_0'(x+W) \Big]$ and $\gamma_j = q_j^2$, which gives
$$\sumin q_i^2 \Var(\psi_0'(c^{\star}_i)) = \sumin q_i^2 \Phi(d_0 \beta_i^\star) + \OPX\bigg(\sqrt{\sumin q_i^4}\bigg) = \varsigma^2 + \OPX(1).$$
Also, the second conclusion is immediate by viewing the LHS as a quadratic form and bounding it by $\|\A_p\|_2 = \BOPX(\sqrt{p/n})$ (where we use \cref{lem:L_p norm}). This completes the proof.
\end{proof}

Next, we prove \cref{prop:iid mean-field clt}, which directly follow from the CLT for independent random variables combined with previous bounds.
\begin{proof}[Proof of \cref{prop:iid mean-field clt}]
    (a) To begin, note that $\sumin q_i(u_i - \psi_i'(c_i)) = o_P(1)$ by part (a) of \cref{thm:CLT for u}, as the error bounds of \cref{thm:CLT for u} converge to $0$, as shown in the proof of \cref{cor:regression i.i.d design}. It thus suffices to verify the first conclusion of \eqref{eq:iid clt 1}. Recall that each $Q_i^{\text{prod}}$ has mean $u_i$ and variance $\psi_i''(s_i + c_i)$. Hence, the standard Lindberg-Feller CLT for independent random variables imply that
    $$\frac{\bq^\top (\bbeta - \bu)}{\sqrt{\sumin q_i^2 \psi_i''(s_i + c_i)}} \mid \by, \bx \to \mcn(0,1),$$
    under $\bbeta \sim \pQ$.
    
    The second last equation in \citep[Step 2 in the proof of Theorem 2.4,][]{lee2025rfim} implies that
    $$\sumin q_i^2 \psi_i''(s_i + c_i) = \sumin q_i^2 \psi_i''(c_i) + O(\|\bq\|_\infty + \sqrt{R_{3p}}) = \upsilon_p + o_P(1).$$
    Here, the simplification of the error term $\sqrt{R_{3p}}$ follows from the computations in \cref{cor:regression i.i.d design}. Also, recall from \cref{cor:regression i.i.d design} that $\upsilon_p = \upsilon_p(\by,\bx) \mid \bx, \bbst \xp \upsilon$, which imply $\upsilon_p \mid \bx \xp \upsilon.$
    Now, the desired conclusion follows by Slutsky's theorem.

    (b) This directly follows by noting that $\delta(\by, \bx) = \bq^\top \bu + o_P(1)$, which was proved in \eqref{eq:optmf} in \cref{prop:regret}.
\end{proof}

Next, we move on to the proof of \cref{cor:sparse}. As $\bbst$ is now a random vector (independent of $\bx$), we additionally let $\EEB$ denote the expectation over $\bbeta^\star$. 
\begin{proof}[Proof of \cref{cor:sparse}]
   Define $\Phi(x):=\Var\psi_0'(x+W)$. Recall the $\phi_1(\cdot)$ notation from \cref{def:allnot} and that $d_0=\sigma^{-2}$. By~\cref{cor:regression i.i.d design} part (b), we get
   $$\delta(\by,\bx)-\sumin q_i\phi_1(d_0\beta_i^{\star})\stackrel{d}{\to}N(0, \varsigma^2),$$
   provided the limit
   $$\varsigma^2=\lim_{p \to \infty}\sum_{i=1}^p q_i^2 \Phi(d_0\beta_i^*)$$
   introduced in \eqref{eq:regression i.i.d convergence assumption sigma} exists. Thus, to understand the limiting distribution of $\delta(\by,\bx)$, it suffices to study the mean and variance of the limiting Gaussian distribution above.
Focusing on the variance, note that
   $$\EE_{\bbst}\sumin q_i^2\Phi(d_0\beta_i^\star)=\EE_{r_p,\otb}[(1-r_p)\Phi(0)+r_p \Phi(d_0\otb)]=\Phi(0)+o(1).$$ 
 For the mean, we claim
   \begin{align}\label{eq:claim_mean}
   \sumin q_i\EE_{\beta_i^{\star}}\phi_1(d_0\beta_i^{\star}) = \begin{cases} 0 & \mbox{if}\, \EE\phi_1(d_0 \otb)=0 \\ \zeta \EE\phi_1(d_0\beta_i^{\star}) + o(1) & \mbox{if}\, \zeta\in\R,\,\, \,\EE\phi_1(d_0\beta_1^{\star})\ne 0 \\ \infty & \mbox{if}\, \zeta=\infty.\end{cases}
   \end{align}
   Finally, we claim that both the mean and the variance concentrates around the expectations:
   \begin{align}\label{eq:claim_conc}
   \mbox{Var} \left(\sumin q_i\phi_1(d_0\beta_i^\star)\right)=o(1), \quad  \mbox{Var} \left(\sumin q_i^2\Phi(d_0\beta_i^\star) \right)=o(1).
   \end{align}
   Combining the above two claims \eqref{eq:claim_mean} and \eqref{eq:claim_conc}, the desired conclusions of the Corollary follow. Focusing on verifying claim \eqref{eq:claim_mean},
  note that $\phi_1(0)=0$ and hence the conclusion follows from the following: 
    $$\EE_{\bbst} \sumin q_i \phi_1(d_0\beta_i^\star) = \EE \phi_1(d_0\otb)\frac{1}{1+p^u}\sumin q_i = p^{-u}q^{(p)}_{tot} \EE \phi_1(d_0\otb)(1+o(1)).$$
    Proceeding to verify \eqref{eq:claim_conc}, we only prove the first display as the second display can be proved similarly. Note that
       \begin{align*}
        &\;\;\;\Var\Big(\sumin q_i \phi_1(d_0\beta_i^\star)\Big)\\ &= \Var_{r_p}\EE_{\bbst}\Big[\sumin q_i \phi_1(d_0\beta_i^\star)\,\mid \, r_p\Big] \, + \, \EE_{r_p} \Var_{\bbst}\Big[\sumin q_i \phi_1(d_0\beta_i^\star)\, \mid \, r_p\Big] \\ &=\Var_{r_p}(r_p(\EE_{\otb} \phi_1(d_0\otb)-\phi_1(0)))+\EE_{r_p}\Big[r_p\EE_{\otb}(\phi_1(d_0\otb)-\phi_1(0))^2\Big] \\ & \quad + \EE_{r_p}\Big[r_p^2(\EE _{\otb}\phi_1(\otb)-\phi_1(0))^2\Big]\\ &=O(p^{-2u})+O(p^{-u}) = o(1).
    \end{align*}
  
    This completes the proof.
\end{proof}

\subsection{CLTs under true Bayesian model}\label{sec:pftrueBayes}

Throughout this section, the expectation $\EE$ computes the expectation under the randomness of $\bc \mid \bx, \bbeta^\star$, and $\EEB$ takes account the randomness of $\bbeta^\star$ while still conditioning on $\bx$. We write $\OPB(1)$ and $\BOPB(1)$ to capture the randomness of $\bbst$. Throughout section \ref{sec:pftrueBayes}, all implied constants depend only on 
$d_0$, (an upper bound of) $\mathfrak{E}_1, \mathfrak{E}_2, \sigma$, and the fourth moment of the probability measure $\mu^\star$. Recall the definition of $\upsilon$ from \cref{cor:regression Bayesian truth} part (a). 

We begin with a couple of preparatory lemmas that control moments of appropriate functions under the law of $\bbst$, which are respective analogs of Lemmas \ref{lem:random design properties} and \ref{lem:lincomb expectation random design} that were used in the previous subsection. Throughout this Section, we will use the symmetry of $\mu^\star$ around $0$ and the fact that $\EEB\phi_1(d_0\beta_1^\star)=0$ and $\EEB\phi_2(d_0\beta_1^\star)=\upsilon$ (recall $\phi_a(\cdot)$  from \cref{def:allnot}).

\begin{lemma}\label{lem:beta concentration}
Let $\mu^\star$ be a symmetric measure with finite fourth moment, and suppose $\beta_1^{\star},\ldots ,\beta_p^{\star}\overset{i.i.d}{\sim}\mu^{\star}$. Suppose that $\|\A_p\| =O(1)$. Recall the definition of $\mathfrak{F}_i$s from \eqref{eq:c_approx}. 
Then, for any deterministic vector $\boldsymbol{\gamma} \in \mathbb{R}^p$ and bounded function $\Phi:\R \to [-1,1]$, the following holds:
\begin{enumerate}[(a)]
    \item If $\EE_{\bbst} \Phi(\beta_1^{\star})=0$, $\EEB \left[\sumjn \gamma_j \Phi(\beta^\star_j)\right]^2 \le \|\bgamma\|^2$. In particular, 
    $\max_{i=1}^p \EEB \mathfrak{F}_i^2 \le \an$. 
    \item $\EEB \Big[\sumjn \gamma_j \mathfrak{F}_j {\Phi}(\beta^\star_j)\Big]^2 \lesssim \|\bgamma\|^2$. In particular, if $\EE_{\bbst} \Phi(\beta_1^{\star}) = 0$, the above bound improves to $\an \|\bgamma \|^2$.
    
    \item If $\EEB \Phi(\beta_1^{\star})=0$, then $ \EEB \Big[\sumjn \gamma_j \Phi(\beta_j^\star) \Big]^4 \lesssim \|\bgamma\|^4$. Consequently, $\max_{i=1}^p \EEB \mathfrak{F}_i^4 \lesssim  \alpha_p^2$ and $\EEB \Big[\sumjn |\gamma_j| \mathfrak{F}_j^2 \Big]^2 \lesssim p\alpha_p^2 \|\bgamma\|^2.$ 
    \item $\EEB \Big[\sumjn \gamma_j \mathfrak{F}_j \Phi(\beta_j^\star) \Big]^4 \lesssim (1+\alpha_p^2)\|\bgamma\|^4$.
\end{enumerate}
\end{lemma}

\begin{lemma}\label{lem:lincomb expectation Bayes}
Let $\mu^\star$ be a symmetric measure with finite fourth moment, and suppose $\beta_1^{\star},\ldots ,\beta_p^{\star}\overset{i.i.d}{\sim}\mu^{\star}$. Suppose $\|\A_p\|=O(1)$ and Assumptions \ref{assn:mf},\ref{assn:homogeneous diagonals} hold.
Then, for any deterministic vector $\boldsymbol{\gamma} = (\gamma_1,\ldots ,\gamma_p) \in \mathbb{R}^p$ (potentially depending on $\bx$) and $\mathcal{C}^2_b$  function $\Phi:\R \to [-1,1]$ with $\EE_{\bbst} \Phi(d_0\beta_1^\star) = 0$, the following holds:
    \begin{align}\label{eq:lincomb moment bound general}
        \EEB \Big[\sumin \gamma_j \Phi(d_0\beta_j^\star - \mathfrak{F}_j) \Big]^2\lesssim \|\bgamma\|^2.
    \end{align}
    In particular, we have
    \begin{align}\label{eq:lincomb expectation Bayes}
         \EEB \Big[\sumjn\gamma_j \EE [\psi_0'(c^{\star}_j)] \Big]^2 \lesssim \|\bgamma\|^2, \quad \EEB \Big[\sumjn \gamma_j \big(\EE [\psi_0''(c^{\star}_j)]-\upsilon \big)\Big]^2 \lesssim \|\bgamma\|^2,
    \end{align}
    and
     \begin{align}\label{eq:remeq}
        \EEB \EE\Big[\Big[\sumjn \gamma_j \psi_j'(c_j) \Big]^2\Big] \lesssim \|\bgamma\|^2, \quad 
        \EEB \EE\Big[\Big[\sumjn \gamma_j \big(\psi_j''(c_j) - \upsilon\big) \Big]^2\Big]  \lesssim \|\bgamma\|^2.
    \end{align}
\end{lemma}

Now, we are ready to prove Theorem \ref{cor:regression Bayesian truth}.
\begin{proof}[Proof of \cref{cor:regression Bayesian truth}]
\emph{(a).} Positivity of $\upsilon>0$ follows from positivity of $\psi''(\cdot)$. For verifying parts (a)(i) and (a)(ii), following the proof of \cref{cor:regression i.i.d design}, we need to show that the RHS in Lemmas \ref{thm:clt} and \ref{thm:CLT for u} respectively, are $o_P(1)$. Since the analogue of \eqref{eq:basicbounds1} %
is part of the hypothesis (see Assumptions \ref{assn:ht} and \ref{assn:mf}), it suffices to show \eqref{eq:suffice1} by replacing $\OPX, \BOPX$-bounds to $\OPB, \BOPB$, where the randomness arises from the true coefficients $\bbeta^\star$. Noting that part (ii) of the statement requires a stronger eigenpair assumption \eqref{eq:stronger eigenpair}, we verify the first line of \eqref{eq:suffice1} under the assumptions in part (i) and the second line additionally assuming \eqref{eq:stronger eigenpair}.
We will repeatedly use the bounds
\begin{align}\label{eq:citelater}
\mathfrak{E}_1 = \sumin (d_i - d_0)^2 \lesssim 1, \quad \|\A_p\| < 1, \quad \|\Sigman\| \le d_0 + \sqrt{\mathfrak{E}_1} + 1\lesssim 1,
\end{align}
which follow from Assumptions \ref{assn:ht} and \ref{assn:homogeneous diagonals}.

    \begin{itemize}[itemsep = 1em]
        \item We first show 
        \begin{align}\label{eq:newcall5}
        \EEB \EE\left[(\upsilon_p - \upsilon)^2\right] \lesssim  \sumin q_i^4 = o(1).
        \end{align}
        This follows by recalling that $\upsilon_p = \sumin q_i^2 \psi_i''(c_i)$ (see \cref{def:upsilon_p}) and applying Equation \eqref{eq:remeq} of  \cref{lem:lincomb expectation Bayes} with $\gamma_j = q_j^2$.

        \item To control $R_{1p}$, we upper bound its expectation. Define $$S_{1i} := \sumjn \A_p(i,j) q_j(\psi_j''(c_j) -\upsilon_p), \text{ and }\tilde{S}_{1i} := \sumjn \A_p(i,j) q_j(\psi_j''(c_j) -\upsilon),$$ so that $R_{1p} = \sumin S_{1i}^2$. %
        By \eqref{eq:remeq} of \cref{lem:lincomb expectation Bayes} with $\gamma_j = \A_p(i,j) q_j$, we have $$\EEB \EE [\tilde{S}_{1i}^2] \lesssim \sumjn \A_p(i,j)^2 q_j^2.$$
        Hence, using the fact that $S_{1i}-\tilde{S}_{1i}=(\upsilon-\upsilon_p)\sumjn \A_p(i,j)q_j$ and the above display, we have
        \begin{align*}
            \EEB\EE [S_{1i}^2 ] &\le 2 \EEB\EE[\tilde{S}_{1i}^2] + 2\Big[\sumjn \A_p(i,j) q_j \Big]^2 \EEB\EE [(\upsilon -\upsilon_p)^2]
            \\
            &\lesssim \sumjn \A_p(i,j)^2 q_j^2 + [(\A_p \bq)_i]^2\EE [(\upsilon -\upsilon_p)^2].
        \end{align*}
        Now, by summing the above over all $i$, we get
        \begin{align*}
            \EEB\EE [R_{1p}] &= \sumin \EEB\EE [S_{1i}^2] 
            \lesssim  \Big( \sumij \A_p(i,j)^2 q_j^2 \Big) + \|\A_p\bq\|^2 \EEB \EE[(\upsilon - \upsilon_p)^2].
        \end{align*}
        By using the bound $\|\A_p \bq\|^2=\|\boldsymbol{\epsilon}+\lambda_p \bq\|^2\le 2\|\boldsymbol{\epsilon}\|^2+2\lambda_p^2$ and \eqref{eq:newcall5}, we observe that 
    \begin{align*}
        \lVert \A_p \bq \rVert^2 \EEB\EE[(\upsilon - \upsilon_p)^2] & \lesssim\lVert\boldsymbol{\epsilon}\rVert^2\sumin q_i^4 + \sumin (\lambda_p q_i)^2 q_i^2 \nonumber \\ &  \lesssim \lVert \boldsymbol{\epsilon}\rVert^2 + \sumin \epsilon_i^2 q_i^2 + \sumin \left(\sumjn \A_p(i,j)q_j\right)^2 q_i^2 \nonumber \\ & \lesssim \lVert \boldsymbol{\epsilon}\rVert^2+\an,
    \end{align*}
    where the last step uses Cauchy-Schwarz inequality to get
    $$\Big(\sum_{j=1}^p\A_p(i,j)q_j\Big)^2\le \sum_{j=1}^p \A_p(i,j)^2 \sum_{j=1}^p q_j^2\le \alpha_p.$$
        As $\sumij \A_p(i,j)^2q_j^2\le \an$, the above bounds yield 
        \begin{align}\label{eq:newcall101}
            \EEB\EE[R_{1p}] \lesssim \an+\lVert \boldsymbol{\epsilon}\rVert^2.
        \end{align}

        \item Next, we show that $\EEB\EE [R_{2p}] \lesssim p \an$. To this effect, we recall that $$R_{2p} = \sumin t_i^2,\quad \text{ for }\quad t_i=\sumjn \A_p(i,j) \psi_j'(c_j).$$ By \eqref{eq:remeq} of \cref{lem:lincomb expectation Bayes} with $\gamma_j = \A_p(i,j)$, we have $$\EEB \EE [t_i^2] \lesssim \sum_{j=1}^p\A_p(i,j)^2\le \an.$$
        The proof is complete by summing over $i$ to get
        \begin{align}\label{eq:newcall102}
        \EEB\EE [R_{2p}] = \sumin \EEB\EE [t_i^2] \lesssim \sumij \A_p(i,j)^2 \le p\an.
        \end{align}
        
        \item Now, we combine the high probability bounds
        $$R_{1p} = O_p(\an + \|\boldsymbol{\epsilon}\|^2), \quad R_{2p} = O_p(p \an)$$ to check the related conclusions in \eqref{eq:suffice1}. Under the eigenpair assumption $\|\boldsymbol{\epsilon}\| = o(1)$ in \eqref{eq:approximate eigenpair}, the bounds for $R_{1p}, \alpha_p R_{2p}$ hold. Next, under the stronger eigenpair assumption in \eqref{eq:stronger eigenpair}, the bounds $R_{1p}R_{2p}, R_{2p}\|\boldsymbol{\epsilon}\|^2 = o_{P,\bbst}(1)$ follow.
        
        \item To bound $R_{3p}=\sumin t_i^4=\OPB(1)$, we apply part (a) of \cref{lem:R_{1p} to R_{4p} bound} and \eqref{eq:citelater} to get 
        \begin{align}\label{eq:combcall}
        \EE R_{3p}\lesssim p\alpha_p^2+\sumin \Big[\sumjn \A_p(i,j) \EE [\psi_0'(c^{\star}_j)]\Big]^4.
        \end{align}
        Proceeding to control the second term above, 
        by doing a two term Taylor expansion of $\EE [\psi_0'(c^{\star}_j)] = \phi_1(d_0 \beta_j^\star - \mathfrak{F}_j)$ around $d_0 \beta_j^\star$ (see \eqref{eq:tocite} and \eqref{eq:twotermder}), it suffices to bound uniformly over $1\le i\le p$: 
        {\small
        $$\Big[\sumjn \A_p(i,j) \phi_1(d_0 \beta_j^\star)\Big]^4,~ \Big[\sumjn \A_p(i,j) \mathfrak{F}_j \phi'_1(d_0 \beta_j^\star)\Big]^4, ~ \Big[\sumjn |\A_p(i,j)| \mathfrak{F}_j^2 \Big]^4.$$}
        The first and second terms are $\BOPB(\alpha_p^2)$ by applying part (c) and (d) of \cref{lem:beta concentration} with $\gamma_j := \A_p(i,j)$, respectively. 
        For the third term, we use the Cauchy-Schwartz inequality to get 
        $$\Big[\sumjn |\A_p(i,j)| \mathfrak{F}_j^2 \Big]^4 \le \Big(\sumjn \A_p(i,j)^2 \sumjn \mathfrak{F}_j^4 \Big)^2 \le \alpha_p^2 \left(\sumjn \mathfrak{F}_j^4\right)^2.$$
        As $\EEB[\sumjn \mathfrak{F}_j^4] \lesssim p \alpha_p^2$ by \cref{lem:beta concentration}, part (c), invoking \eqref{eq:combcall}, we get: 
        \begin{align}\label{eq:newcall103}
           \EE [R_{3p}] = \BOPB(p\alpha_p^2 + p^3 \alpha_p^6)=\OPB(1). 
        \end{align}        
        
        \item To show $R_{4p}=\OPB(1)$ via a first moment bound, we apply \eqref{eq:citelater}, $\lVert \bbeta^\star\rVert^2=\BOPB(p)$, and part (b) of \cref{lem:R_{1p} to R_{4p} bound} with $U=\upsilon$, %
        to get
        \begin{small}
        \begin{align*}
          &\;\;\;\;\big[\EE  R_{4p}\big]^2 \\ &\lesssim \Big[\sumij \A_p(i,j) q_i (\EE \psi_0''(c_i^\star) - \upsilon) \EE \psi_0'(c_j^\star)\Big]^2 + \BOPB(p\alpha_p^2) +\Bigg(\sumin q_i^2 \Big[\sumjn \A_p(i,j) \EE \psi_0'(c^{\star}_j) \Big]^2 \\ &+ \sumjn \Big[\sumin \A_p(i,j)q_i(\EE \psi_0''(c^{\star}_i) -\upsilon)\Big]^2 \Bigg)+ \OPB(1)\Bigg(\Big[\sumij \A_p(i,j) q_i \EE \psi_0'(c^{\star}_j) \Big]^2+1\Bigg)\\
          &=\bigg[\sumin \tilde{\Gamma}_i(\EE\psi_0''(c_i^\star)-\upsilon)\bigg]^2+\lVert \tilde{\boldsymbol{\Gamma}}\rVert^2+\lVert \boldsymbol{\Gamma}\rVert^2+\OPB(1)\bigg(1+\EEB\bigg[\sumin \tilde{\Gamma}_i\bigg]^2\bigg),
        \end{align*}
        \end{small}
        where 
        \begin{align}\label{eq:gamma tilde}
        \tilde{\Gamma}_i := \sumjn \A_p(i,j) q_i \EE[\psi_0'(c^{\star}_j)] \quad \mbox{and}\quad \Gamma_j := \sumin \A_p(i,j) q_i (\EE [\psi_0''(c^\star_i) ] - \upsilon).
        \end{align}
        Here, the first inequality used previous bounds for $\upsilon_p - \upsilon, R_{1p}, R_{3p}$ to simplify the RHS of \cref{lem:R_{1p} to R_{4p} bound} (see \eqref{eq:newcall5}, \eqref{eq:newcall101},  \eqref{eq:newcall103}, part (a) of \cref{lem:beta concentration}).
        It then suffices to show the following:
        \begin{small}
        \begin{align}\label{eq:R_4 suffice}
        \EEB\left(\Big[\sumin \tilde{\Gamma}_i(\EE [\psi_0''(c^\star_i)] - \upsilon)\Big]^2+\|\boldsymbol{\Gamma}\|^2+\|\tilde{\boldsymbol{\Gamma}}\|^2\right) = o(1),\,\,\,
        \EEB \Big[\sumin \tilde{\Gamma}_i\Big]^2 = O(1).
        \end{align}
        \end{small}
       The last conclusion of \eqref{eq:R_4 suffice} follows on  using the definition of $\tilde{\boldsymbol{\Gamma}}$ which yields
        \begin{align*}
           \sum_{i=1}^p\tilde{\Gamma}_i=\sum_{j=1}^p \EE[\psi_0'(c_j^\star)] \sum_{i=1}^p \A_p(i,j)q_i=\sum_{j=1}^p\EE[\psi_0'(c_j^\star)] (\A_p\bq)_j,
        \end{align*}
       and invoking the first bound in \eqref{eq:lincomb expectation Bayes} with $\gamma_j = (\A_p \bq)_j$, along with the bound  $\|\A_p \bq\| \le 1$.

        Now, we move onto the first conclusion of \eqref{eq:R_4 suffice}, and argue the bounds $\EEB \|\boldsymbol{\Gamma}\|^2, \EEB \|\tilde{\boldsymbol{\Gamma}}\|^2 = O(\an)$. But this follows on using the first and second bounds in \eqref{eq:lincomb expectation Bayes} respectively, to get
        \begin{align}
        \begin{aligned}
        \EEB \tilde{\Gamma}_i^2=&\EEB \Big[\sum_{j=1}^p\A_p(i,j)q_i \EE[\psi_0'(c_j^\star)]\Big]^2 \lesssim \sumjn \A_p(i,j)^2 q_i^2,\label{eq:Gammtil}\\
         \EEB {\Gamma}_j^2 = &\EEB \Big[\sumin\A_p(i,j)q_i(\EE[\psi_0''(c_i^\star)]-\upsilon)\Big]^2\lesssim \sumin \A_p(i,j)^2 q_i^2.
         \end{aligned}
        \end{align}
        and summing over $i,j$.
        Finally, to bound the first term in \eqref{eq:R_4 suffice}, we use the Taylor expansion \eqref{eq:taylor phi_2} to write
        \begin{align}\label{eq:beta quadratic form}
            &\;\;\;\;\;\EEB\bigg[\sumin \tilde{\Gamma}_i(\EE [\psi_0''(c^\star_i)] - \upsilon)\bigg]^2\notag  \\ & = \EEB \Big[\sumin \tilde{\Gamma}_i \big(\phi_2(d_0 \beta_i^\star) - \upsilon- \mathfrak{F}_i \phi_2'(\eta_i) \big)\Big]^2 \notag \\
            &\le 2 \EEB \Big[\sumin \tilde{\Gamma}_i \big(\phi_2(d_0 \beta_i^\star) - \upsilon\big)\Big]^2 + 2 \EEB \Big[\sumin \tilde{\Gamma}_i \mathfrak{F}_i \phi_2'(\eta_i) \Big]^2 \notag \\
            &\lesssim \EEB \Big[\sumin \tilde{\Gamma}_i \big(\phi_2(d_0 \beta_i^\star) - \upsilon\big)\Big]^2 + \EEB \|\tilde{\boldsymbol{\Gamma}}\|^2 \EEB \sumin \mathfrak{F}_i^2,
        \end{align}
        where the last step uses Cauchy-Schwartz inequality.
In \eqref{eq:beta quadratic form} above, the second term is $O(p\alpha_p^2)$ by combining individual expectation bounds  $\EEB\|\tilde{\boldsymbol{\Gamma}}\|^2\lesssim \alpha_p$ (obtained above in \eqref{eq:Gammtil}), and $\max_{1\le i\le p}\EEB[\mathfrak{F}_i^2]\lesssim \alpha_p$ (see part (a) of \cref{lem:beta concentration}).

        \noindent To bound the first term in \eqref{eq:beta quadratic form}, define $$\bar{\Gamma}_j := \sumin \A_p(i,j) q_i(\phi_2(d_0 \beta_i^\star) - \upsilon)$$ where the second equality uses \eqref{eq:tocite}.  Note that 
        \begin{align*}
            \sumin \tilde{\Gamma}_i \big(\phi_2(d_0 \beta_i^\star) - \upsilon\big) =& \sumij \A_p(i,j)q_i\EE\psi_0'(c_j^{\star})(\phi_2(d_0\beta_i^{\star})-\upsilon))\\=&\sumjn \bar{\Gamma}_j \EE \psi_0'(c_j^\star)
            =\sumjn \bar{\Gamma}_j \phi_1(d_0 \beta_j^\star - \mathfrak{F}_j),
            \end{align*}
        where the last equality uses \eqref{eq:tocite}. 
        Further, using arguments similar to the derivation of \eqref{eq:Gammtil}, we get  $\EE\lVert \bar{\boldsymbol{\Gamma}}\rVert^2\lesssim \sumin \A_p(i,j)^2 q_i^2\le \an$. Finally, 
        by using the Taylor expansion \eqref{eq:taylor phi_1}, repeating the arguments in the derivation of \eqref{eq:beta quadratic form} we get
        \begin{align*}
            &\;\;\;\;\;\EEB\Big[\sumjn  \bar{\Gamma}_j \phi_1(d_0 \beta_j^\star-\mathfrak{F}_j)\Big]^2 \\ &= \EEB\Big[\sumjn \bar{\Gamma}_j (\phi_1(d_0 \beta_j^\star)-\mathfrak{F}_j\phi_1'(\xi_j))\Big]^2 \\
            & \lesssim \EEB\Big[\sumij \A_p(i,j) q_i \big(\phi_2(d_0 \beta_i^\star) - \upsilon\big) \phi_1(d_0 \beta_j^\star) \Big]^2 + \EEB \lVert \bar{\boldsymbol{\Gamma}}\rVert^2\EEB\sumjn \mathfrak{F}_j^2 \\ &=\EEB\Big[\sumij \A_p(i,j) q_i  \big(\phi_2(d_0 \beta_i^\star) - \upsilon\big) \phi_1(d_0 \beta_j^\star) \Big]^2 + O(p\alpha_p^2),
        \end{align*}
        where the last bound uses $\EEB\|\bar{\boldsymbol{\Gamma}}\|^2\lesssim \alpha_p$ (obtained above), and $\max_{1\le i\le p}\EEB[\mathfrak{F}_i^2]\lesssim \alpha_p$ (see part (a) of \cref{lem:beta concentration}).
       Since $p\alpha_p^2\to 0$ by \cref{assn:mf}, it only remains to bound the first term in the above display. By using $\EE \phi_1(d_0 \beta_j^\star) = \EE \phi_2(d_0 \beta_j^\star) - \upsilon = 0$, a direct second moment computation along with the symmetry assumption for $\beta_i^\star$ yields 
        $$\EEB \Big[\sumij \A_p(i,j) q_i \phi_1(d_0 \beta_j^\star) \big(\phi_2(d_0 \beta_i^\star) - \upsilon\big) \Big]^2 \lesssim \sumij \A_p(i,j)^2 q_i^2 \le \an.$$
        This controls the first term in \eqref{eq:beta quadratic form}, and hence establishes \eqref{eq:R_4 suffice}.

        \item The bound for ${\bm \epsilon}^\top \Psi'({\bf c})$ follows by applying \eqref{eq:remeq} of \cref{lem:lincomb expectation Bayes} with $\gamma_i = \epsilon_i$, which gives
        $$\EEB \EE\left[\left({\boldsymbol \epsilon}^\top \Psi'({\bf c})\right)^2\right] \lesssim \|{\bm \epsilon}\|^2 = o(1).$$
        \end{itemize}

    \noindent\emph{(b)} 
    We first claim that it suffices to show that 
    \begin{align}\label{eq:sufpf}
    \left((1-\lambda_p\upsilon_p)\delta(\by,\bx) - \sumin q_i \big(\phi_1(d_0\beta_i^\star) - \lambda \upsilon \beta_i^\star \big)\right) \mid \bx, \bbst \xd \mcn(0,\tilde{\varsigma}^2),
    \end{align}
    where $\tilde{\varsigma}^2:=(1-\lambda \upsilon)^2\varsigma^2$, and $\phi_1(\cdot)$ is defined as in \eqref{eq:tocite}. 
    To see this, recall $\lambda_p\to \lambda \in (-1, 1)$ (by assumption), $\upsilon_p \mid \bx, \bbst \xp \upsilon \in (0,1)$ (using \eqref{eq:newcall5}), and that $1-\lambda \upsilon > 1-\upsilon > 0$. By Slutsky's theorem, \eqref{eq:sufpf} implies
    $$\delta(\by,\bx) - \frac{\sumin q_i \big(\phi_1(d_0\beta_i^\star) - \lambda \upsilon \beta_i^\star \big)}{1-\lambda_p \upsilon_p} \mid \bx, \bbst \xd \mcn(0, \varsigma^2).$$
    Now, as the symmetry of $\beta_i^\star \sim \mu^\star$ and $\mu$ implies $\EEB \phi_1(d_0\beta_i^\star) = \EEB \beta_i^\star = 0$, we have $\sumin q_i \big(\phi_1(d_0\beta_i^\star) - \lambda \upsilon \beta_i^\star \big) = \BOPB(1)$ by a second moment bound. Hence, we have
    $$\frac{\sumin q_i \big(\phi_1(d_0\beta_i^\star) - \lambda \upsilon \beta_i^\star \big)}{1-\lambda_p \upsilon_p} - \frac{\sumin q_i \big(\phi_1(d_0\beta_i^\star) - \lambda \upsilon \beta_i^\star \big)}{1-\lambda \upsilon} \mid \bx, \bbst \xp 0,$$
    and the desired claim in \eqref{eq:regression CLT annealed} follows by adding the above two displays.
    
    We postpone the argument for $\varsigma^2 > 0$ to the end of the proof.
    
    To begin, use~\cref{lem:remove centering} to get that
    \begin{align}\label{eq:clt42}\frac{(1-\lambda_p\upsilon_p)\delta(\by,\bx)-\sumin q_i\EE[\psi_i'(c_i)]}{\sqrt{\mbox{Var}\Big(\sumin q_i\psi_i'(c_i)|\bx,\bbst\Big)}} \mid \bx, \bbst  \xd \mcn(0,1).
    \end{align}
    Note that the error bounds in~\cref{lem:remove centering} goes to $0$, using individual bounds in part (a) of the current theorem (\cref{cor:regression Bayesian truth}). 
Also, with $\varsigma_p^2$ defined as in \eqref{eq:vthetan}, using the assumptions $\|\bq\|_\infty = o(1), p\alpha_p^2=o(1), \mathfrak{E}_1=o(1)$, and the bounds 
$$\|\bbst\|^2 = \BOPB(p), \quad \sumin (\beta_i^\star)^4 = \BOPB(p), \quad \sumin (d_i-d_0)^2 (\beta_i^\star)^2 = \BOPB(\mathfrak{E}_1),$$
\cref{lem:remove centering two} gives 
    $$\mbox{Var}\Big(\sumin q_i\psi_i'(c_i)|\bx,\bbst\Big)-\varsigma_p^2=\OPB(1).$$
Let $M_p\equiv M_p(\bbst) := \sumin q_i \EE [\psi_i'(c_i)]$ be the centering in the LHS of \eqref{eq:clt42}. Combining the above two displays with \cref{lem:clt marginalize}, the conclusion in \eqref{eq:sufpf} would follow if we separately show the following:
    \begin{align}\label{eq:conditional CLT terms}
        M_p = \sumin q_i \Big[\phi_1(d_0 \beta_i^\star) -\upsilon \lambda \beta_i^\star\Big] + \OPB(1), \quad \varsigma_p^2 = \tilde{\varsigma}^2 + \OPB(1).
    \end{align}
    
    \noindent \emph{Proof of the first claim in \eqref{eq:conditional CLT terms}.} Recall that we assume $\mathfrak{E}_1 = o(1)$ for part (c). Therefore, by using \eqref{eq:lincombound} in \cref{lem:lincombound}, we get: 
    $$\bigg|M_p-\sumin q_i\EE\psi_0'(c_i^\star)\bigg|=\bigg|\sumin q_i(\EE\psi_i'(c_i^\star)-\EE \psi_0'(c_i^\star))\bigg|\lesssim \sqrt{\mathfrak{E}_1}=o(1).$$
    Therefore by the Taylor series approximation in \eqref{eq:taylor phi_1}, we can write $M_p$ as 
    $$M_p=\sumin q_i\EE \psi_0'(c_i^\star)+o(1)=\sumin q_i \phi_1(d_0 \beta_i^\star) - \sumin q_i \mathfrak{F}_i \phi'_1(d_0 \beta_i^\star) + \BOPB\Big( \sumin |q_i| \mathfrak{F}_i^2\Big)+ o(1).$$
    By \cref{lem:beta concentration}, part (c), we have
    $$\sumin |q_i|\mathfrak{F}_i^2 \le \sqrt{\sumin \mathfrak{F}_i^4}=\BOPB(\sqrt{p}\an)=\OPB(1).$$
    Let $W_0\sim \mcn(0,d_0)$ as in the statement of the theorem and observe that by \cref{def:allnot}, $\phi_1(d_0\beta_1^\star)=\EE[\psi_0'(d_0\beta_1^\star+W_0)|\beta_1^\star]$, which on differentiating gives  $\EE_{\beta_1^\star}\phi_1'(d_0\beta_1^\star)=\EE\psi_0''(d_0\beta_1^\star+W_0)=\upsilon$. Therefore, an application of \cref{lem:beta concentration} part (b) with $\gamma_i=q_i$ and  $\Phi(\beta_i^\star)=\phi_1'(d_0\beta_i^\star)-\upsilon$ yields 
    $$\sumin q_i\mathfrak{F}_i(\phi_1'(d_0\beta_i^\star)-\upsilon)=\BOPB(\sqrt{\an})=\OPB(1).$$
    Next, by combining the three displays above, we get: 
    $$M_p=\sumin q_i\phi_1(d_0\beta_i^\star)-\upsilon\sumin q_i\mathfrak{F}_i+\OPB(1).$$
    As $\mathfrak{F}_i=\sumjn \A_p(i,j)\beta_j^\star$ (see \cref{def:c_approx}), a second moment argument gives: 
    $$\sumin q_i\mathfrak{F}_i-\lambda\sumjn q_j\beta_j^\star=\sumjn (\A_p\bq -\lambda\bq)_j \beta_j^\star = \BOPB(\lVert \A_p\bq-\lambda \bq\rVert)=\OPB(1).$$
    Combining the two displays above, we finally get: 
    $$M_p=\sumin q_i \Big[\phi_1(d_0 \beta_i^\star) -\upsilon \lambda \beta_i^\star\Big] + \OPB(1).$$

    \vspace{3mm}

    \noindent \emph{Proof of the second claim in \eqref{eq:conditional CLT terms}.}
    Recall from \eqref{eq:vthetan} that $$\varsigma^2_p = \sumin q_i^2 \Var(\psi_0'(c^{\star}_i)) - \sum_{i \neq j} \A_p(i,j) q_i q_j \EE \psi_0''(c^{\star}_i) \EE \psi_0''(c^{\star}_j)$$
    and from \eqref{eq:sufpf} that $$\tilde{\varsigma}^2 = (1-\lambda \upsilon)^2 \varsigma^2 = \EE \Big[\Var \Big(\psi_0'(d_0 \beta_1^\star + W_0 \mid \beta_1^\star \Big)\Big] - \lambda \upsilon^2.$$
    Thus, it suffices to show the individual limits:
    \begin{equation}\label{eq:suffices varsigma^2}
    \begin{aligned}
        \sumin q_i^2 \Var(\psi_0'(c^{\star}_i)) \mid \bbst &\xp \EE \Big[\Var \Big(\psi_0'(d_0 \beta_1^\star + W_0) \mid \beta_1^\star \Big)\Big],  \\
        \sum_{i \neq j} \A_p(i,j) q_i q_j \EE \psi_0''(c^{\star}_i) \EE \psi_0''(c^{\star}_j) \mid \bbst &\xp \lambda \upsilon^2. 
    \end{aligned}
    \end{equation}
    
    To show first line in \eqref{eq:suffices varsigma^2} we will use \eqref{eq:lincomb moment bound general}. To wit, define $$\widetilde{\Phi}(x) := \Var\Big[\psi_0'(x + W_0) \Big]
    \Rightarrow  \EE_{\beta_1^\star}\widetilde{\Phi}(d_0\beta_1^\star)= \EE_{\beta_1^\star} \mbox{Var}\Big[\psi_0'(d_0\beta_1^\star+W_0)|\beta_1^\star\Big] .$$ 
    Define $$\Phi(x):=\widetilde{\Phi}(x)-\EE_{\beta_1^\star}\widetilde{\Phi}(d_0\beta_1^\star), \implies \EE_{\beta_1^\star}\Phi(d_0\beta_1^\star)=0.$$
    By the definition of $c_i^\star$ from \eqref{eq:c_approx}, we have $$\mbox{Var}(\psi_0'(c_i^\star))=\mbox{Var}(\psi_0'(d_0\beta_i^\star-\mathfrak{F}_i+W_0) \mid \bbst)=\widetilde{\Phi}(d_0\beta_i^\star-\mathfrak{F}_i).$$ As a result, by \eqref{eq:lincomb moment bound general}, we get:
    
    \begin{align}\label{eq:first term varsigma Bayes}
        &\;\;\;\;\sumin q_i^2 \big(\Var(\psi_0'(c^{\star}_i)) - \EE_{\beta_1^\star} \widetilde{\Phi}(d_0 \beta_1^\star)\big) \nonumber \\ &= \sumin q_i^2 \Phi(d_0\beta_i^\star-\mathfrak{F}_i) = \BOPB(\|\bq\|_\infty)=\OPB(1),
    \end{align}
   
    which verifies the first limit in \eqref{eq:suffices varsigma^2}.
    
    To prove the second limit in \eqref{eq:suffices varsigma^2}, as in the proof of part (a), we set $$\Gamma_i = \sumjn \A_p(i,j) q_j (\EE [\psi_0''(c^\star_j) ] - \upsilon).$$
    Then, by adding and subtracting $\upsilon$ from each term $\EE \psi_0''(c^{\star}_i), \EE \psi_0''(c^{\star}_j)$ we can write 
    \begin{align}
    &\;\;\;\;\sum_{i \neq j} \A_p(i,j) q_i q_j \EE \psi_0''(c^{\star}_i) \EE \psi_0''(c^{\star}_j) \nonumber \\ &=\sum_{i\neq j} \A_p(i,j)q_iq_j(\EE\psi_0''(c_i^\star)-\upsilon)(\EE\psi_0''(c_i^\star)-\upsilon) + 2\upsilon \sum_{i\neq j} \A_p(i,j)q_iq_j(\EE\psi_0''(c_j^\star)-\upsilon) + \upsilon^2 \bq^\top \A_p \bq\nonumber \\ &= \sumin q_i \Gamma_i (\EE \psi_0''(c^{\star}_i) - \upsilon) + 2 \upsilon \sumin q_i \Gamma_i + \upsilon^2 \bq^\top \A_p \bq. \label{eq:toshow101}
    \end{align}
    To analyze the last term in the RHS of \eqref{eq:toshow101}, using $\A_p{\bf q}=\lambda_p {\bf q}+{\bm \epsilon}$ we get
    $$\bq^\top \A_p \bq = \lambda_p + \bq^{\top}\boldsymbol{\epsilon} = \lambda + o(1).$$
    
    \noindent Hence, it suffices to show that the first two terms in the RHS of \eqref{eq:toshow101} are $\OPB(1)$. For the middle term in the RHS of \eqref{eq:toshow101}, interchanging the sum and again using $\A_p{\bf q}=\lambda_p {\bf q}+{\bm \epsilon}$ we get  %
    \begin{align*}
        \sumin q_i\Gamma_i= \sum_{i\neq j} \A_p(i,j)q_iq_j(\EE\psi_0''(c_j^\star)-\upsilon)=\sumjn (\lambda_p q_j+\epsilon_j)q_j(\EE \psi_0''(c_j^\star)-\upsilon).
    \end{align*}
    Consequently, using \eqref{eq:lincomb expectation Bayes} gives
    \begin{align*}
        \EEB \Big[\sumin q_i \Gamma_i\Big]^2 &= \EEB \Big[\sumjn (\lambda_p q_j + \epsilon_j) q_j (\EE \psi_0''(c_j^\star) - \upsilon) \Big]^2 \\
        &\lesssim \sumjn (\lambda_p q_j+\epsilon_j)^2 q_j^2 \le \|\bq\|_\infty^2 (1 + \|\boldsymbol{\epsilon}\|^2)= o(1).
    \end{align*}
    Finally, we move on to the first term in the RHS of \eqref{eq:toshow101}. By similar computations as in \eqref{eq:beta quadratic form}, we get:
    \begin{align}\label{eq:rev1}
    &\;\;\;\;\EEB \Big[\sumin q_i  \Gamma_i \big(\EE \psi_0''(c^{\star}_i) - \upsilon\big) \Big]^2 \nonumber \\ &\lesssim \EEB \Big[\sumin q_i \Gamma_i \big(\phi_2(d_0 \beta_i^\star) - \upsilon\big)\Big]^2 + \EEB \left(\sumin q_i^2 \Gamma_i^2\right) \sumin \EEB \mathfrak{F}_i^2.
    \end{align}
    We first bound the second term in \eqref{eq:rev1}. Note that $\maxin \EEB \mathfrak{F}_i^2 = O(\an)$ (by \cref{lem:beta concentration}, part (a)). Further, the second bound in \eqref{eq:lincomb expectation Bayes} with $\gamma_j=\A_p(i,j)q_j$ yields that
    \begin{align}\label{eq:prevbound}
         \EEB [\Gamma_i^2] \lesssim \sumjn \A_p(i,j)^2  q_j^2\le \sumjn \A_p(i,j)^2 \sumjn q_j^2 \le \an.
    \end{align}
    Therefore, we have 
    $$\EEB \left(\sumin q_i^2 \Gamma_i^2\right) \sumin \EEB \mathfrak{F}_i^2 \lesssim \an \left(\sumin q_i^2\right) p\an = p\alpha_p^2 = o(1).$$

    \noindent We now move on to the first term of \eqref{eq:rev1}. Recall the definition of $\bar{\Gamma}_j$ from the proof of part (a) given by
    $$\bar{\Gamma}_j = \sumin \A_p(i,j) q_i(\phi_2(d_0 \beta_i^\star) - \upsilon).$$
    The first term of \eqref{eq:rev1} then simplifies to 
    $$\EEB \Big[\sumin q_i \Gamma_i \big(\phi_2(d_0 \beta_i^\star) - \upsilon\big)\Big]^2 = \EEB\Big[\sumjn q_j \bar{\Gamma}_j (\EE \psi_0''(c_j^\star)-\upsilon)\Big]^2.$$
    Once again repeating the same computation as in \eqref{eq:beta quadratic form}, we get: 
    \begin{align*}
        &\;\;\;\;\EEB\Big[\sumjn q_j \bar{\Gamma}_j (\EE \psi_0''(c_j^\star)-\upsilon)\Big]^2 \\ &\le \EEB\Big[\sumjn q_j \bar{\Gamma}_j (\phi_2(d_0\beta_j^\star)-\upsilon)\Big]^2 + \left(\EEB\sumjn q_j^2 \bar{\Gamma}_j^2\right)\EEB \sumjn \mathfrak{F}_j^2 \\ &=\EEB\Big[\sum_{i\neq j} \A_p(i,j)q_i q_j(\phi_2(d_0\beta_i^\star)-\upsilon)(\phi_2(d_0\beta_j^\star)-\upsilon)\Big]^2 + \left(\EEB\sumjn q_j^2 \bar{\Gamma}_j^2\right)\EEB \sumjn \mathfrak{F}_j^2.
    \end{align*}
    Once again, we note that $\maxjn \EEB \mathfrak{F}_j^2 = O(\an)$ (by \cref{lem:beta concentration}, part (a)). Also by a second moment computation using the independence of $\beta_i^\star$s, we have: 
    $$\EEB \bar{\Gamma}_j^2 \lesssim \sumin \A_p(i,j)^2q_i^2 \le \big(\sumin \A_p(i,j)^2\big) \sumin q_i^2 \le \an.$$
    Therefore, 
    $$\left(\EEB\sumjn q_j^2 \bar{\Gamma}_j^2\right)\EEB \sumjn \mathfrak{F}_j^2 \le \an \bigg(\sumin q_i^2\bigg)p\an = p\alpha_p^2=o(1).$$
    Finally, another second moment computation using the independence of $\beta_i^\star$s yields
    \begin{align*}
    \EEB\Big[\sum_{i\neq j} \A_p(i,j)q_i q_j(\phi_2(d_0\beta_i^\star)-\upsilon)(\phi_2(d_0\beta_j^\star)-\upsilon)\Big]^2 &\lesssim \sum_{i\neq j} q_i^2 q_j^2 \A_p(i,j)^2 \\ &\le \sum_{i} q_i^2 \sum_{j\neq i} \A_p(i,j)^2\le \an = o(1).
    \end{align*}

    \vspace{1mm}
    
    \noindent \emph{Positivity of $\varsigma^2$}. Let $\upsilon = \EE \psi_0''(d_0\beta_1^{\star}+W_0)$ as in \cref{cor:regression Bayesian truth}. Using the decomposition of $\Sigman$ in \eqref{eq:anbayes}, we can write
    $$0 \le \frac{\bq^\top \bx^\top \bx \bq}{\sigma^2} = \bq^\top \diag(\mathbf{d}) \bq - \bq^\top \A_p \bq.$$
    Under the assumption $\mathfrak{E}_1=\sum_{i=1}^p (d_i-d_0)^2=o(1)$ (see \cref{cor:regression Bayesian truth}, part (b)), we can simplify each term as
    $$\bq^\top \diag(\mathbf{d}) \bq = \sumin q_i^2 d_i = \sumin q_i^2 d_0 + \sumin q_i^2 (d_i - d_0) = d_0 + O(\|\bq\|_\infty \sqrt{\mathfrak{E}_1}) = d_0 +o(1)$$
    (the error term is bounded via Cauchy-Schwartz alongside $\sumin q_i^4 \le \|\bq\|_\infty^2$), and
    $$\bq^\top \A_p \bq = \bq^\top \lambda \bq + o(1) = \lambda + o(1).$$
    This computation implies $\lambda \le d_0$.
    
    Using $d_0$ as an upper bound for $\lambda$, it suffices to show that 
    \begin{align}\label{eq:varineqshow}
       \EE \Big[\Var\big(\psi_0'(d_0 \beta_1^\star+W_0) \big) \Big] > d_0 \big(\EE \psi_0''(d_0\beta_1^{\star}+W_0)\big)^2.
    \end{align}
    Here, recall that $W_0 \sim \mcn(0,d_0) \perp \beta_1^\star$.
    By using Stein's identity (while conditioning on $\beta_1^\star$), we get:
    $$d_0\EE [\psi_0''(d_0\beta_1^\star+W_0)\mid \beta_1^\star] = \EE [W_0 \psi_0'(d_0\beta_1^\star+W_0)\mid \beta_1^\star] = \Cov(W_0,\psi_0'(d_0\beta_1^\star+W_0)\mid \beta_1^\star).$$
    By taking another expectation and squaring the above display, using the Cauchy-Schwartz inequality twice (for the outer expectation and the inner covariance, respectively), we have 
    \begin{align*}
        \big(\EE \psi_0''(d_0\beta_1^{\star}+W_0)\big)^2 & = \left(\EE\left[\frac{\Cov(W_0,\psi_0'(d_0\beta_1^{\star}+W_0)\mid \beta_1^\star)}{d_0} \right]\right)^2 \\ 
        & \le \frac{\EE\left[\Cov(W_0,\psi_0'(d_0\beta_1^{\star}+W_0)\mid \beta_1^\star)^2\right]}{d_0^2}
        \\
        & \le \frac{\EE \Big[\Var\big(\psi_0'(d_0 \beta_1^\star+W_0) \mid \beta_1^\star\big)\Big]}{d_0} .
    \end{align*}
    The last inequality uses the fact that $\Var(W_0)=d_0$. This shows \eqref{eq:varineqshow} holds with a $\ge$ sign. For equality to hold in the above Cauchy-Schwartz, $\psi_0'$ must be linear which implies $\mu_0$ must be Gaussian, which is ruled out by our compact support assumption on $\mu$. This completes the proof. 
\end{proof}

\begin{proof}[Proof of \cref{prop:Bayes true model mean-field clt}]
    The proof is identical to \cref{prop:iid mean-field clt}, where we use the error bounds computed under \cref{cor:regression Bayesian truth}.
\end{proof}

\subsection{Application to credible intervals}\label{sec:pfBvmanalog}
To prove \cref{cor:coverage probability Bayes}, we additionally use the following technical lemma, which helps marginalize out the random parameters $\bbst$.

\begin{lemma}\label{lem:clt marginalize}
    Let $G_p = G_p(\by,\bx,\bbst)$, $H_p = H_p(\bbst)$ be random variables, and $\varsigma^2, \vartheta^2 > 0$ be deterministic values. If $G_p \mid\bx,\bbst \xd \mcn(0, \varsigma^2)$ and $H_p \xd \mcn(0,\vartheta^2)$, we have
        $G_p + H_p \mid \bx \xd \mcn(0, \varsigma^2+\vartheta^2).$
\end{lemma}

\begin{proof}[Proof of \cref{cor:coverage probability Bayes}]

    (a) Using \eqref{eq:optmf} and the decomposition \eqref{eq:bias}, we can write the LHS of \eqref{eq:coverage clt modified} as
    $$\sumin q_i (u_{i} - \beta_i^\star) = \delta(\by, \bx)+o(1) - \sumin q_i \beta_i^\star = \text{LHS of \eqref{eq:regression CLT annealed}} + \text{bias}_p(\bbst)+o(1).$$
    where the notation $\text{bias}_p(\bbst)$ is defined in \eqref{eq:bias}.
    We separately control the two terms in the above expression. For notational convenience, define 
    \begin{align*}
        G_p(\by, \bx, \bbst) &:= \text{LHS of \eqref{eq:regression CLT annealed}} = \delta(\by, \bx) - \frac{1}{1-\lambda \upsilon} \sumin q_i\Big(\phi_1(d_0\beta_i^\star)-\lambda \upsilon \beta_i^\star\Big)%
    \end{align*}
    and  
    $$\vartheta^2 := \frac{\Var \left[\phi_{1}(d_0 \beta_1^\star) - \beta_1^\star \right]}{(1-\lambda \upsilon)^2},$$
    
    \noindent where $\phi_1(\cdot)$ is defined as in \cref{def:allnot}. 

    \noindent First, use part (b) of \cref{cor:regression Bayesian truth} to conclude
    $$G_p(\by, \bx, \bbst) \mid \bx, \bbst \xd \mcn(0, \varsigma^2).$$
   Proceeding to analyze $\text{bias}_{p}(\bbst)$, use \eqref{eq:bias} and the definition of $\phi_1(\cdot)$ from  \cref{def:allnot}, to get:
   $$\text{bias}_p(\bbst)=\frac{1}{1-\lambda\upsilon}\sumin q_i(\phi_1(d_0\beta_i^\star)-\beta_i^\star).$$
   We then use the symmetry of $\mu$ and $\mu^\star$ to conclude that $$\EEB \phi_{1}(d_0 \beta_1^\star) = \EEB \EE \psi_{0}'(\mcn(d_0 \beta_1^\star, d_0)) = 0 = \EEB \beta_1^\star.$$ Consequently, by the Lindberg-Feller CLT for independent random variables, we have
    $$\sumin q_i \Big(\phi_{1} (d_0 \beta_i^\star) - \beta_i^\star\Big) \xd \mcn(0, (1-\lambda \upsilon)^2 \vartheta^2).$$
    Now, by applying \cref{lem:clt marginalize} to combine the two limiting distributions, we have
    $$\sumin q_i\big(u_{i} - \beta_i^\star\big) = (G_p + \text{bias}_p) \mid \bx \xd \mcn\big(0, \vartheta^2 + \varsigma^2\big).$$

    Finally, it remains to prove that the above limiting variance matches the value $\tau^2$ in \eqref{eq:coverage clt modified}, i.e. $\vartheta^2+\varsigma^2 = \tau^2$. 
    For notational convenience, define random variables $B, C$ (given $\mu^\star, d_0$) by setting 
    \begin{align}\label{eq:B and C}
        B \sim \mu^\star, \quad C \mid B \sim \mcn(d_0 B, d_0).
    \end{align}
    Then, we can write $\phi_1(d_0 B) = \EE[ \psi_0'(C) \mid B]$ (see \cref{def:allnot}) so we have $\EE \phi_1(d_0 B) = \EE[ \psi_0'(C)] = 0$. Using this notation, we can write
    \begin{align*}
        (1-\lambda \upsilon)^2 \vartheta^2 &= \Var\Big[\phi_1(d_0 B) -B \Big] = \Var\Big[\EE[\psi_0'(C)  - B \mid B] \Big],\\
        (1-\lambda \upsilon)^2 \varsigma^2 &= \EE\Big[\Var[\psi_0'(C) \mid B]\Big] - \lambda \upsilon^2 = \EE\Big[\Var[\psi_0'(C) - B \mid B]\Big] - \lambda \upsilon^2.
    \end{align*}
    Hence, by the law of total variance,
    \begin{align}\label{eq:tau decomposition}
        \vartheta^2 + \varsigma^2 = \frac{\Var[\psi_0'(C) - B] - \lambda \upsilon^2}{(1-\lambda \upsilon)^2 }  = \tau^2.
    \end{align}
    
    (b) Note from \cref{def:upsilon_p} that $\limsup_p \upsilon_p\le 1$ and hence $\liminf_p (1-\lambda_p \upsilon_{p}) \ge  1-\lambda > 0$. By combining the CLT in part (a) of the current theorem and the limit $\upsilon_{p} \mid \bx \xp \upsilon > 0$  (also proved in part (a) of \cref{cor:regression Bayesian truth}),  Slutsky's theorem gives
    $$\frac{\sumin q_i(u_i - \beta_i^\star)}{\sqrt{\upsilon_{p}/(1-\lambda_p \upsilon_{p})}} \mid \bx \xd \mcn\Big(0, \frac{\tau^2}{\upsilon/(1-\lambda \upsilon)}\Big).$$
    Now, noting that 
    $$\text{AC}(\mathcal{I}(\by, \bx, \mu); \mu^\star) = \PP\left(\frac{\big|\sumin q_i(u_i - \beta_i^\star)\big|}{\sqrt{\upsilon_{p}/(1-\lambda_p \upsilon_{p})}} \le c_{\alpha/2} \mid \bx \right),$$
    the definition of convergence in distribution immediately implies \eqref{eq:coverage}.

    \vspace{2mm}
    It remains to show that the RHS of \eqref{eq:coverage} equals $1-\alpha$ when $\mu = \mu^\star$, for which it suffices to show that $\tau^2 =\upsilon/(1-\lambda \upsilon).$
    In the remainder of the proof, we assume $\mu = \mu^\star$ and write $\mu^\star_0$ as the quadratic tilt of $\mu^\star$ (see \cref{def:psi_0}) with density $\frac{d\mu^\star_0}{d\mu^\star}(b) \propto e^{-d_0 b^2/2}$. Also let $\psi_0(\theta) := \log \int e^{\theta b} d{\mu^\star_0}(b)$ denote the cumulant generating function of $\mu^\star_0$. We also assume that the constants $\upsilon, \tau^2$ are defined under $\mu = \mu^\star$.
    
    Using the notations $(B,C)$ from \eqref{eq:B and C} along with the formula  
    $$\tau^2 = \frac{\Var\Big(\beta_1^\star - \psi_{0}'(d_0 \beta_1^\star+W_0)\Big) - \lambda \upsilon^2}{(1-\lambda \upsilon)^2}=  \frac{\Var[\psi_{0}'(C) - B]-\lambda \upsilon^2}{(1-\lambda \upsilon)^2}$$
    from the theorem statement, it suffices to show that $\Var[\psi_{0}'(C) - B] = \upsilon$.
    We first use \eqref{eq:B and C} to make an observation that the conditional distribution of $B$ given $C$ is absolutely continuous with respect to $\mu^\star$ with derivative
    $$\frac{d\PP(B \mid C)}{d\mu^\star}(b) \propto \exp\Big[- \frac{d_0 b^2}{2} + Cb \Big].$$
    Hence, we can view $\PP(B \mid C)$ as $\mu^\star_{0,(C)}$ (this is the exponential tilt of $\mu^\star_0$; again see \cref{def:psi_0}):
    $$\frac{d\mu^\star_{0,(C)}}{d\mu^\star}(b) = \exp (C b - \psi_{0}(C)), \quad\text{where} \quad \psi_{0}(c)= \log \int e^{c b} d\mu^\star_0(b).$$
   Thus by the exponential tilt properties in \cref{def:expfam},
     \begin{align*}
         \psi_{0}'(C) = \EE[B\mid C], \quad \psi_{0}''(C) = \Var[B \mid C].
     \end{align*}
     Finally, using the law of total variance to expand $\Var[\psi_{0}'(C) - B]$ conditional on $C$, we can write
    \begin{align*}
        \Var[\psi_{0}'(C) - B] 
        &= \EE \Big[ \Var\big[B \mid C\big] \Big] + \Var \Big[ \EE\big[\psi_{0}'(C) - B \mid C\big] \Big] \\
        &= \EE \psi_{0}''(C) + 0 = \upsilon,
    \end{align*}
    so the proof is complete.
\end{proof}

\section{Proof of main lemmas from \cref{sec:pfbayeslinhigh}}\label{sec:pfpfblh}
Recall all the notation from \cref{sec:bayeslinhigh} and \cref{sec:pfbayeslinhigh}.

\subsection{Proof of common lemmas: \ref{lem:remove centering}--\ref{lem:remove centering two}}

First, we prove \cref{lem:remove centering}. To wit, we require the following result from \cite[Theorem 2.2]{chatterjee2009fluctuations}.

\begin{prop}[See \cite{chatterjee2009fluctuations}]\label{prop:Chatterjee clt}
	Let $\mathbf{Z} = (Z_1, \ldots, Z_p)$ be a Normal random vector with distribution $\mcn(\boldsymbol{\mu}, \boldsymbol\Sigma)$. Let $F:\mathbb{R}^p \to \mathbb{R}$ be a twice continuously differentiable  function. Suppose that the random variable $F(\mathbf{Z})$ has mean $0$, variance $\tau^2$, a finite fourth moment, and let $W\sim \mcn(0,\tau^2)$. Then,
	$$d_{BL}(F(\mathbf{Z}), W) \le \frac{2 \sqrt{5} \|\boldsymbol\Sigma \| \big(\EE \lVert \nabla F(\mathbf{Z}) \rVert^4 \EE \lVert \nabla^2 F(\mathbf{Z}) \rVert^4\big)^{1/4}}{\tau^2}.$$
\end{prop}

\begin{proof}[Proof of \cref{lem:remove centering}]

To begin, using the definition of the bounded Lipschitz metric, we get: 
\begin{align*}
    &\;\;\;\;d_{BL}\left((1-\lambda_p\upsilon_p)\delta(\mathbf{y},\mathbf{X})-\sum_{i=1}^p q_i \EE [\psi_i'(c_i)],\widetilde{W}_p\right)\\ &\le d_{BL}\left((1-\lambda_p\upsilon_p)\sumin q_i (\psi_i'(c_i) - \EE [\psi_i'(c_i)]),\widetilde{W}_p\right) + (1-\lambda_p \upsilon_p)\EE\bigg|\delta(\by,\bx)-\sumin q_i u_i\bigg| \\ &\qquad \qquad + \EE\bigg|\sumin q_i((1-\lambda_p\upsilon_p)u_i-\psi_i'(c_i))\bigg|.
\end{align*}
The bounds on the second and third terms in the above display follow from \eqref{eq:optmf} and \eqref{eq:CLT for u} respectively. Proceeding to bound the first term, it suffices to show that 
$$d_{BL}\left(\sumin q_i(\psi_i'(c_i)-\EE[\psi_i'(c_i)]),\widetilde{W}_p\right)\lesssim \lVert \bq\rVert_{\infty}.$$
Towards this direction, define $F(\bc):=\sumin q_i(\psi_i'(c_i)-\EE[\psi_i'(c_i)])$. Then, $\lVert \nabla F(\bc) \rVert^4 \lesssim \big(\sumin q_i^2 \big)^2 = 1$, and $$\lVert \nabla^2 F(\bc) \rVert^4 = \maxin q_i^4 \psi_i'''(c_i)^4 \lesssim \lVert \bq \rVert_{\infty}^4.$$
    An application of \cref{prop:Chatterjee clt} completes the proof.
    \end{proof}

Next, we prove lemmas \ref{lem:lincombound}--\ref{lem:remove centering two} using the following two technical lemmas. 
To approximate the expectations in \cref{lem:lincombound}, \cref{lem:exponential tilt expectation} quantifies the difference of two Gaussian expectations, and \cref{lem:quadratic tilt moment} provides a uniform bound for the difference in the derivatives of $\psi_i$ and $\psi_0$. We defer the proof of these lemmas to Appendix \ref{sec:pfauxlem}.

\begin{lemma}\label{lem:exponential tilt expectation}
    Let $g$ be any $C^1$ function with $\lVert g' \rVert_{\infty} < \infty$. Suppose $C \sim \mcn(m, d), C_0 \sim \mcn(m_0, d_0)$ for some constants $m, m_0$ and $d, d_0 > 0$. Then, we have
        $$|\EE g(C) - \EE g(C_0)| \lesssim \lVert g' \rVert_{\infty} (|m - m_0| + |d - d_0|).$$
        Here, the constant in $\lesssim$ only depends on $d_0$.
    \end{lemma}

 \begin{lemma}\label{lem:quadratic tilt moment}
    Fix $p\ge 1$. For any $\theta$ and $i = 1, \ldots, p$, under \cref{assn:homogeneous diagonals}, %
    we have
    \begin{align*}
        |\psi'_i(\theta) - \psi'_0(\theta)| &= |\EE_{\mu_{i,\theta}} (V) - \EE_{\mu_{0,\theta}} (V)| \lesssim |d_i-d_0|, \\
        |\psi''_i(\theta) - \psi''_0(\theta)| &= |\Var_{\mu_{i,\theta}} (V) - \Var_{\mu_{0,\theta}} (V)| \lesssim |d_i-d_0|.
    \end{align*}
        Here, the constant in $\lesssim$ only depends on (an upper bound of) $\mathfrak{E}_1=\sumin (d_i-d_0)^2$ which is bounded under \cref{assn:homogeneous diagonals}.
    \end{lemma}

For the proofs of lemmas \ref{lem:lincombound}--\ref{lem:remove centering two}, all hidden constants depend on $d_0, \mathfrak{E}_2,$ and an upper bound of $\mathfrak{E}_1$. Further, all following expectations (and variances, covariances, probabilities, etc.) are conditional on $\bx, \bbeta^\star$.

\begin{proof}[Proof of \cref{lem:lincombound}]
    It suffices to separately show the variance bound and expectation bound in \eqref{eq:lincombound}, for any $a=1,2$. We start by bounding the variance.
    For simplicity of notation, define $F(\bc):=\sumjn \gamma_j \psi_j^{(a)}(c_j)$ where $\bc=(c_1,\ldots ,c_p)$. We note that $\bc\sim \mcn(\Sigman\bbeta^\star, \Sigman)$ and 
    $$\lVert \nabla F(\bc))\rVert^2 = \sumjn \gamma_j^2 \big(\psi_j^{(a+1)}(c_j)\big)^2\lesssim \lVert \bgamma\rVert^2.$$
   Using the Gaussian Poincar\'{e} inequality  (e.g. Corollary 2.27 in \cite{van2014probability}), this gives 
    $\Var F(\bc)\le \|\Sigman\| \lVert \bgamma\rVert^2 \lesssim  \lVert \bgamma\rVert^2$ (see \eqref{eq:citelater}).
    
    Now, we move on to proving the expectation bound. Note that 
    $\sumjn |\gamma_j| |d_j - d_0| \le \|\boldsymbol{\gamma}\| \sqrt{\mathfrak{E}_1},$
    and recall that $c_j \sim \mcn(d_j \beta_j^\star - \mathfrak{F}_j, d_j)$ and $c^{\star}_j \sim \mcn(d_0 \beta_j^\star - \mathfrak{F}_j, d_0)$ (see \eqref{eq:truec} and \eqref{eq:c_approx}).
    By consecutively applying \cref{lem:exponential tilt expectation} (with $g = \psi_j^{(a)}$) and \cref{lem:quadratic tilt moment}, we get 
    \begin{align*}
        &\EE\Big[\sumjn \gamma_j \psi_j^{(a)}(c_j)\Big] - \EE\Big[\sumjn \gamma_j\psi_0^{(a)}(c^{\star}_j)\Big] \\
        = &\EE\Big[\sumjn \gamma_j \psi_j^{(a)}(c^{\star}_j)\Big] - \EE\Big[\sumjn \gamma_j\psi_0^{(a)}(c^{\star}_j)\Big] + O(\|\boldsymbol{\gamma}\| \sqrt{\mathfrak{E}_1}) \\
        =& \EE\sumjn \gamma_j \Big[\psi_j^{(a)}(c^{\star}_j)-\psi_0^{(a)}(c^{\star}_j)\Big] + O(\|\boldsymbol{\gamma}\| \sqrt{\mathfrak{E}_1}) = O(\|\boldsymbol{\gamma}\| \sqrt{\mathfrak{E}_1}).
    \end{align*}
    This completes the proof.
\end{proof}

\begin{proof}[Proof of \cref{lem:R_{1p} to R_{4p} bound}]

\noindent \emph{(a)} Setting $S_{3i}(\bc):=\big(\sumjn \A_p(i,j) \psi_j' (c_j)\big)^2 = t_i(\bc)^2$, we have $R_{3p} = \sumin S_{3i}(\bc)^2.$ A direct computation gives
    $$\lVert \nabla S_{3i}(\bc)\rVert^2 = 4 S_{3i}(\bc)\sumjn \A_p(i,j)^2 \psi_j''(c_j)^2 \le 4 S_{3i} (\bc) \an.$$
    Hence, by applying the Gaussian Poincar\'{e} inequality to bound the variance, we have
    \begin{align}
        \EE S_{3i}^2(\bc) &= \Var [S_{3i}(\bc)] + [\EE S_{3i}(\bc)]^2 \lesssim \an \|\Sigman\|\EE[S_{3i}(\bc)] + [\EE S_{3i}(\bc)]^2 \notag \\
        &\lesssim \alpha_p^2 (\mathfrak{E}_1^2+\|\Sigman\|^2) + \Big[\sumjn \A_p(i,j) \EE \psi_0'(c^{\star}_j)\Big]^4. \label{eq:e_3i final}
    \end{align}
    Here, the inequality in \eqref{eq:e_3i final} follows by using part (a) of \cref{lem:lincombound} with $\gamma_j = \A_p(i,j)$ to get
    $$\big|\EE S_{3i}(\bc)\big| \lesssim \an (\mathfrak{E}_1+\|\Sigman\|) +  \Big[\sumjn \A_p(i,j) \EE \psi_0'(c^{\star}_j)\Big]^2.$$
    Now, the final conclusion follows by summing \eqref{eq:e_3i final} over $i$.

\vspace{3mm}
\noindent \emph{(b)} Define functions $S_1(\bc), S_2(\bc)$ by setting
$$S_1(\bc) := \sum_{i,j} \A_p(i,j) q_i (\psi_i''(c_i) - U) \psi'_j(c_j), \quad S_2(\bc) := \sum_{i,j} \A_p(i,j) q_i \psi_j'(c_j).$$
Then, we can write $R_{4p} = |S_1(\bc) + (U - \upsilon_p) S_2(\bc)|$. Since 
$$\EE R_{4p} \le \EE|S_1(\bc)| + \EE\big|(U - \upsilon_p) S_2(\bc)\big| \le \sqrt{\EE S_1(\bc)^2} + \sqrt{\EE(\upsilon_p - U)^2 \EE S_2(\bc)^2},$$
it suffices to show
\begin{small}
\begin{align}\label{eq:s_1}
    \EE &S_1(\bc)^2 \lesssim \Big[\sumij \A_p(i,j) q_i (\EE \psi_0''(c_i^\star) - U) \EE \psi_0'(c_j^\star)\Big]^2 
    + \|\Sigman\| \EE \Big[{R}_{1p} + \sqrt{R_{3p}} +  (\upsilon_p-U)^2\Big] \\ 
    &+ \alpha_p^2 \Big((1+\mathfrak{E}_1) \|\bbst\|^2 + p(1+\mathfrak{E}_1) \Big) \notag \\
    &+ \mathfrak{E}_1\Bigg(\sumin q_i^2 \Big[\sumjn \A_p(i,j) \EE \psi_0'(c^{\star}_j) \Big]^2 + \sumjn \Big[\sumin \A_p(i,j)q_i(\EE \psi_0''(c^{\star}_i) -U)\Big]^2 + \an \Bigg), \notag
\end{align}
\end{small}

\begin{align}\label{eq:s_2}
    \EE S_2(\bc)^2 &\le \Big[ \sumij \A_p(i,j)q_i \EE \psi_0'(c^{\star}_j) \Big]^2 + \mathfrak{E}_1 + \|\Sigman\|.
\end{align}

\noindent \emph{Proof of \eqref{eq:s_1}.} We write $\EE S_1(\bc)^2 = [\EE S_1(\bc)]^2 + \Var[S_1(\bc)]$ and separately bound the mean and variance.

We begin with bounding the variance term. Note that
    \begin{align*}
        \frac{\partial S_1}{\partial c_k} &= \sumin \A_p (i,k) q_i (\psi_i''(c_i) - U) \psi_k''(c_k) + \sumjn \A_p (k,j) q_k \psi_k'''(c_k) \psi_j'(c_j).
    \end{align*}
    Recalling the definition of $R_{1p}$ and $R_{3p}$ (see \eqref{eq:random field}), we can write
    \begin{align*}
         \left\lVert \frac{\partial S_1}{\partial \bc} \right\rVert^2 &= \sumkn \Big(\sumin \A_p (i,k) q_i (\psi_i''(c_i) - U) \psi_k''(c_k) + \sumjn \A_p (k,j) q_k \psi_k'''(c_k) \psi_j'(c_j) \Big)^2 \\
         &\lesssim \sumkn \Big(\sumin \A_p (i,k) q_i (\psi_i''(c_i) - U) \Big)^2 + \sumkn \Big( \sumjn \A_p (k,j) \psi_j'(c_j) \Big)^2 q_k^2 \\
         &\le R_{1p} + \|\bq\|_\infty \sqrt{R_{3p}} +  (\upsilon_p-U)^2.
    \end{align*}
    Now, the variance bound follows by applying the Gaussian Poincar\'{e} inequality:
    \begin{align}\label{eq:var R4}
        \Var[S_1(\bc)] &\lesssim \|\Sigman\| \EE\Big[R_{1p} + \|\bq\|_\infty \sqrt{R_{3p}} + (\upsilon_p-U)^2 \Big].
    \end{align} %

    \vspace{0.1in}
    
    \noindent Next, we upper bound $\EE S_1(\bc)$. To this extent, for all $i \neq j$, we define 
    \begin{align}\label{eq:onedell}
    \ct_j^i := c_j + \frac{\A_p(i,j)}{d_i} c_i.
    \end{align}
    By a Taylor expansion, we have
    \begin{align}\label{eq:psi_j' taylor}
        \psi_j'(c_j) = \psi_j'(\ct_j^i) + O\Big(|\A_p(i,j)| |c_i|\Big).
    \end{align}
    Since the Gaussianity of $\bc$ (see \eqref{eq:truec} with $\Sigman = \diag(\mathbf{d}) - \A_p$) implies $$\Cov(c_i, \ct_j^i) = \Cov(c_i, c_j) + \frac{\A_p(i,j)}{d_i} \Var(c_i)=-\A_p(i,j)+\A_p(i,j)= 0,$$ %
    $c_i$ and $\ct_j^i$ are independent. Using the above two displays, we factorize the expectation:
    \begin{align*}
        &\EE (\psi_i''(c_i) - U) \psi_j'(c_j) = (\EE \psi_i''(c_i) - U) \EE \psi_j'(\ct_j^i)+ O(|\A_p(i,j)| \EE |c_i|). 
    \end{align*}

   \noindent Using this, we get
    \begin{align}%
        \notag~~~\EE S_1(\bc) &=
        \sumij \A_p(i,j) q_i (\EE \psi_i''(c_i) - U) \EE \psi_j'(\ct_j^i)+ O\left(\sumij \A_p(i,j)^2 |q_i|\EE|c_i|\right) \notag \\
        &= \sumij \A_p(i,j) q_i (\EE \psi_i''(c_i) - U) \EE \psi_j'(c_j)+ O\left(\sumij \A_p(i,j)^2 |q_i| \EE|c_i|\right).\label{eq:expectation of R in R_{4p}}
    \end{align}

    The second term in the RHS of \eqref{eq:expectation of R in R_{4p}} can be bounded as
    \begin{align*}
        \sumij \A_p(i,j)^2 |q_i| \EE|c_i| \le &\an \sumin |q_i| \EE|c_i| \le \an \sqrt{\EE\sumin c_i^2} \\ 
        \le &\alpha_p\sqrt{(1+\mathfrak{E}_1)\|{\bm \beta}^\star\|^2+p(1+\sqrt{\mathfrak{E}_1})},
        \end{align*}
where in the last inequality we use the facts that $c_i\sim N(d_i\beta_i^\star-\mathfrak{F}_i,d_i)$ and $d_i\le d_0+\sqrt{\mathfrak{E}_1}$ to conclude that for $\ell=2,4$ we have
        \begin{align}\label{eq:expectation c_i}
       \notag \sum_{i=1}^p \EE c_i^\ell \lesssim &\sum_{i=1}^p \Big[d_i^\ell (\beta_i^\star)^\ell+\mathfrak{F}_i^\ell+d_i^{\ell/2}\Big] \le (d_0^\ell + \mathfrak{E}_1^{\ell/2})\|{\bm \beta^\star}\|_\ell^\ell+\sum_{i=1}^\ell \mathfrak{F}_i^\ell+p (d_0^{\ell/2}+\mathfrak{E}_1^{\ell/4})\\
        \lesssim &(1+\mathfrak{E}_1^{\ell/2})\lVert \bbeta^\star\rVert_{\ell}^{\ell}+p(1+\mathfrak{E}_1^{\ell/4}),
\end{align}
where in the last inequality above we use the fact that $\|\A_p\|_4\le 1$.

To analyze the first term in the RHS of \eqref{eq:expectation of R in R_{4p}},
define $\Gamma^\star_j := \sumin \A_p(i,j) q_i (\EE \psi_i''(c_i) - U)$ and  $\tilde{\Gamma}^\star_i := \sumjn \A_p(i,j) q_i \EE[\psi_0'(c^{\star}_j)]$ as in\eqref{eq:gamma tilde}. Note that $\Gamma^\star_j$s are analogs of $\Gamma_j$s defined in \eqref{eq:gamma tilde}, with $\psi_i''$ instead of $\psi_0''$ and $U$ instead of $\upsilon$. Also $\tilde{\Gamma}^\star_i$ is an analog of $\tilde{\Gamma}_i$ from \eqref{eq:gamma tilde} with $q_i$ and $q_j$ flipped. By applying \eqref{eq:lincombound} twice (for each index $j$ and $i$), we can write 
        \begin{align*}
            \sumij \A_p(i,j) q_i (\EE \psi_i''(c_i) - U) \EE \psi_j'(c_j) &= \sumjn \Gamma^\star_j \EE \psi_j'(c_j) = \sumjn \Gamma^\star_j \EE \psi_0'(c^{\star}_j)  + O\Big(\|\boldsymbol{\Gamma}^\star\| \sqrt{\mathfrak{E}_1} \Big) \\
            &= \sumin \tilde{\Gamma}^\star_i (\EE \psi_i''(c_i) - U) + O\Big(\|\boldsymbol{\Gamma}^\star\| \sqrt{\mathfrak{E}_1} \Big) \\
            &= \sumin \tilde{\Gamma}^\star_i (\EE \psi_0''(c^{\star}_i) - U) + O\Big((\|\boldsymbol{\Gamma}^\star\| + \|\tilde{\boldsymbol{\Gamma}}^\star \|) \sqrt{\mathfrak{E}_1} \Big).
        \end{align*}
        The error term $\|\boldsymbol{\Gamma}^\star\|$ can be further simplified by again using \eqref{eq:lincombound} to write
        $$\Gamma^\star_j = \sumin \A_p(i,j)q_i(\EE \psi_0''(c^{\star}_i) -U) + O\Bigg(\sqrt{\sumin \A_p(i,j)^2 q_i^2}\Bigg),$$
        which gives
        $\|\boldsymbol{\Gamma}^\star\|^2 \le 2 \sumjn \Big[\sumin \A_p(i,j)q_i(\EE \psi_0''(c^{\star}_i) -U)\Big]^2 + 2 \an.$
        Consequently, by spelling out the $\tilde{\Gamma}_i$s,
        we can write
        \begin{align*}
            [\EE &S_1(\bc)]^2 = \Big[\sumij \A_p(i,j) q_i (\EE \psi_0''(c^\star_i) - U) \EE \psi_0'(c^\star_j) \Big]^2 + O(p\alpha_p^2) \\
            +& O\Big(\big(\sumin q_i^2 \Big[\sumjn \A_p(i,j) \EE \psi_0'(c^{\star}_j) \Big]^2 + \sumjn \Big[\sumin \A_p(i,j)q_i(\EE \psi_0''(c^{\star}_i) -U)\Big]^2 + \an \big) \mathfrak{E}_1\Big).
        \end{align*}
    The final bound in \eqref{eq:s_1} is immediate by combining the bounds for the variance in \eqref{eq:var R4} and bias squared in the above line.

    \noindent \emph{Proof of \eqref{eq:s_2}.} To bound $\EE S_2(\bc)^2$, define $\boldsymbol{\gamma} := \A_p \bq$ and write $S_2(\bc) = \sumjn \gamma_j \psi_j'(c_j)$. By part (a) of \cref{lem:lincombound} and using $\|\boldsymbol{\gamma}\| \lesssim 1$, we get
    $$\EE S_2(\bc)^2 - \Big[ \sumjn \gamma_j \EE \psi_0'(c^{\star}_j) \Big]^2 \lesssim \mathfrak{E}_1 + \|\Sigman\|.$$
    This completes the proof.
\end{proof}

\begin{proof}[Proof of \cref{lem:remove centering two}]
With $\ct_j^i=c_j+\frac{\A_p(i,j)}{d_i}c_i$ as in \eqref{eq:onedell}, recall that $c_i$ and $\ct_j^i$ are independent. By using a second order Taylor expansion of $\psi_j'(\cdot)$ around $\ct_j^i$, we get:
  \begin{align}\label{eq:variance of sum}
      &\;\;\;\;\Var\bigg(\sumin q_i \psi_i'(c_i)\bigg) \nonumber \\ &= \sumin q_i^2 \Var(\psi_i'(c_i)) + \sum_{i \neq j} q_i q_j \Cov\bigg(\psi_i'(c_i), \psi_j'\bigg(\ct_j^i - \frac{\A_p(i,j) c_i}{d_i}\bigg)\bigg) \notag \\ 
      &=  \sumin q_i^2 \Var(\psi_i'(c_i)) - \sum_{i \neq j} q_i q_j \Cov\bigg(\psi_i'(c_i), \frac{\A_p(i,j) c_i}{d_i} \psi_j''(\ct_j^i)\bigg) \\
      & \quad + O \bigg(\sumij |q_i| |q_j| |\A_p(i,j)|^2 \EE c_i^2 \bigg). \notag
\end{align}
We now simplify the three terms in \eqref{eq:variance of sum} one at a time. For the first term, we note that $\lVert \psi_i'-\psi_0'\rVert_{\infty}\lesssim |d_i-d_0|$ and $\lVert (\psi_i')^2 - (\psi_0')^2\rVert_{\infty}\lesssim |d_i-d_0|$ by \cref{lem:quadratic tilt moment}. As a result, an application of \cref{lem:exponential tilt expectation} yields: 
\begin{align*}
    &\;\;\; \big|\Var(\psi_i'(c_i)) - \Var(\psi_0'(c^{\star}_i))\big|\lesssim |d_i-d_0|.
\end{align*}
Consequently, we have 
\begin{align}\label{eq:t1}
&\;\;\;\;\Big|\sumin q_i^2 \big[\Var(\psi_i'(c_i)) - \Var (\psi_0'(c^{\star}_i)) \big] \Big| \lesssim \sumin q_i^2 |d_i-d_0| \le \sqrt{\sumin q_i^4}\sqrt{\mathfrak{E}_1} \le \|\bq\|_\infty \sqrt{\mathfrak{E}_1}.
\end{align}

For the second term of \eqref{eq:variance of sum}, by expanding the covariance, using the independence of $c_i$ and $\tilde{c}_j^i$, and using Stein's lemma (see \cite{stein1981estimation} and recall that $\Var(c_i) = d_i$), we get %

\begin{align*}
    \Cov\left(\psi_i'(c_i), \frac{\A_p(i,j) c_i}{d_i} \psi_j''(\ct_j^i)\right) &= \frac{\A_p(i,j)}{d_i} \EE \Big[\big(\psi_i'(c_i) - \EE \psi_i'(c_i)\big) c_i \psi_j''(\ct_j^i) \Big] \\
    &= \frac{\A_p(i,j)}{d_i} \Cov(\psi_i'(c_i), c_i) \EE \Big[ \psi_j''(\ct_j^i) \Big] \\
    &= \A_p(i,j) \EE \psi_i''(c_i) \EE \psi_j''(\ct_j^i).
\end{align*}

Hence, by the triangle inequality, we can write
\begin{align}
&\;\;\;\;\bigg|\sum_{i\neq j} q_i q_j \Cov\left(\psi_i'(c_i), \frac{\A_p(i,j) c_i}{d_i} \psi_j''(\ct_j^i)\right)-  \sum_{i\neq j} \A_p(i,j)q_iq_j\EE \psi_0''(c^{\star}_i)\EE\psi_0''(c^{\star}_j)\bigg| \notag \\ 
&\lesssim \bigg|\sumij \A_p(i,j) q_i q_j \EE \psi_i''(c_i) \Big(\EE \psi_j''(\ct_j^i) - \EE \psi_j''(c_j) \Big) \bigg| \label{eq:aux1} \\
&+ \bigg|\sumij \A_p(i,j) q_i q_j \EE \psi_i''(c_i) \Big(\EE \psi_j''(c_j) - \EE \psi_0''(c_j^\star) \Big) \bigg| \label{eq:aux2} \\
&+ \bigg|\sumij \A_p(i,j) q_i q_j \EE \psi_0''(c_j^\star) \Big(\EE \psi_i''(c_i) - \EE \psi_0''(c_i^\star) \Big) \bigg|. \label{eq:aux3}
\end{align}
We separately bound each summand. To control \eqref{eq:aux1}, we first use \cref{lem:exponential tilt expectation} and the naive bound $d_i \le d_0 + \sqrt{\mathfrak{E}_1}$ and $|\A_p(i,j)|\le 1$, to write 
\begin{align*}
|\EE \psi_j''(c_j) - \EE \psi_j''(\ct_j^i)|&\lesssim |\A_p(i,j)||d_i \beta_i^\star - \mathfrak{F}_i| + \A_p(i,j)^2 \lesssim |\A_p(i,j)|\Big((d_0+\sqrt{\mathfrak{E}_1}) |\beta_i^\star| + |\mathfrak{F}_i| + 1\Big),
\end{align*}
since using \eqref{eq:onedell} along with the independence of $c_i$ and $\ct_j^i$ we have $$ |\EE c_j -  \EE \ct_j^i |= \frac{\A_p(i,j)}{d_i}(d_i\beta_i^\star - \mathfrak{F}_i),\quad \Var c_j - \Var \ct_j^i = \frac{\A_p(i,j)^2}{d_i}.$$
Hence, 
\begin{align}
    \eqref{eq:aux1} &\le \sum_{i\neq j} |q_i| |q_j| |\A_p(i,j)|^2 \big((d_0+\sqrt{\mathfrak{E}_1}) |\beta_i^\star| + |\mathfrak{F}_i| + 1\big) \nonumber \\
    &\lesssim \an \sqrt{\sumin q_i^2((d_0+\sqrt{\mathfrak{E}_1})^2 |\beta_i^\star|^2 + |\mathfrak{F}_i|^2+1)} \nonumber \\
    &\lesssim \an   (1+\sqrt{\mathfrak{E}_1}) \big(\|\bq\|_\infty \|\bbst\| + 1\big)\label{eq:t21}
\end{align}
The second inequality uses the following bound for the entry-wise squared matrix $\mathbf{B}_p(i,j) = \A_p(i,j)^2$:
    \begin{align}\label{eq:b_n matrix}
        \lVert \mathbf{B}_p\rVert \le \an.
    \end{align}
The last inequality follows from the fact that  $\sumin \mathfrak{F}_i^2\le \Vert \bbeta^\star\rVert^2$ (as $\|\A_p\|\le 1$).

Next, we control \eqref{eq:aux2}. 
We use Lemmas \ref{lem:exponential tilt expectation}, \ref{lem:quadratic tilt moment} to bound
$$|\EE \psi_j''(c_j) - \EE \psi_0''(c^{\star}_j)| \lesssim (1+|\beta_j^\star|) |d_j - d_0|$$ since $\EE c_j - \EE c_j^\star = (d_j - d_0)\beta_j^\star, ~ \Var c_j - \Var c_j^\star = d_j - d_0$. Now, viewing \eqref{eq:aux2} as a quadratic form, we have
\begin{align}
    \eqref{eq:aux2} &\lesssim \|\A_p\| \sqrt{\sumin q_i^2 [\EE \psi_i''(c_i)]^2} \sqrt{\sumjn q_j^2 (d_j-d_0)^2 (1+|\beta_j^\star|^2) } \nonumber \\
    &\lesssim \|\bq\|_\infty \sqrt{\mathfrak{E}_1 + \sumjn (d_j-d_0)^2 (\beta_j^\star)^2}.\label{eq:t22}
\end{align}

\eqref{eq:aux3} can be bounded via the exact same argument as for \eqref{eq:aux2}, which leads to an identical bound. 

\vspace{2mm}
Finally, to control the last term (error term) in \eqref{eq:variance of sum}, we again use \eqref{eq:b_n matrix} followed by \eqref{eq:expectation c_i}:
\begin{align}
    \sumij |q_i||q_j| &\A_p(i,j)^2 \EE c_i^2 \le \an \sqrt{\sumin q_i^2 (\EE c_i^2)^2} \lesssim \an \sqrt{\sumin q_i^2 \Big( (d_0^2 +\mathfrak{E}_1)^2 \big((\beta_i^\star)^4 + 1\big) + \mathfrak{F}_i^4\Big)} \nonumber \\
    &\lesssim \an (1 + \mathfrak{E}_1) \Big(\sqrt{\sumin  q_i^2 (\beta_i^\star)^4} + 1\Big) + \an \sqrt{\sumin q_i^2 \mathfrak{F}_i^4} \nonumber\\
    &\lesssim \an  (1 + \mathfrak{E}_1) \Big(\|\bq\|_\infty \sqrt{\sumin  (\beta_i^\star)^4} + 1\Big) + \alpha_p^2 \|\bbst\|^2. \label{eq:t3}
\end{align}
Here, the last inequality follows from a Cauchy-Schwartz to bound $$ \sqrt{\sumin q_i^2 \mathfrak{F}_i^4} \le \maxin \mathfrak{F}_i^2 = \maxin \Big[\sumjn \A_p(i,j) \beta_j^\star\Big]^2 \le \an \|\bbst\|^2.$$ Now, the proof is complete by combining the bounds from \eqref{eq:t1}, \eqref{eq:t21}, \eqref{eq:t22}, \eqref{eq:t3}, and applying them to \eqref{eq:variance of sum}, we get:
\begin{align*}
    &\;\;\;\;\bigg|\mbox{Var}\left(\sumin q_i\psi_i'(c_i)\right)-\sumin q_i^2\mbox{Var}(\psi_0'(c_i))+\sum_{i\neq j}\A_p(i,j)q_iq_j\EE\psi_0''(c_i^\star)\EE\psi_0''(c_j^\star)\bigg| \\ &\lesssim \lVert \bq\rVert_{\infty}\bigg(\sqrt{\mathfrak{E}_1}+\an(1+\sqrt{\mathfrak{E}_1}) \lVert \bbst\rVert + \sqrt{\sumin (d_i-d_0)^2(\beta_i^\star)^2}+\an(1+\mathfrak{E}_1)\sqrt{\sumin (\beta_i^\star)^4}\bigg) \\ & \qquad +\an\big(1+\sqrt{\mathfrak{E}_1}+\mathfrak{E}_1+\an\lVert\bbst\rVert^2\big).
\end{align*}
.%

\end{proof}

\subsection{Proof of Lemmas used for Theorem \ref{cor:regression i.i.d design}}
We prove Lemmas \ref{lem:random design properties}, \ref{lem:lincomb expectation random design} in this section. Recall that all $\lesssim$ and $O(\cdot)$ notations hide constants that depend on $\sigma$ and the sub-Gaussian norm of the $X_{k,i}$s.

\begin{proof}[Proof of \cref{lem:random design properties}]
Note that $Z_{k,i} = \sqrt{n} X_{k,i}$ are iid sub-Gaussian random variables with mean 0 and variance 1. Recall that $\A_p(i,j) = - \frac{1}{n \sigma^2}\sumkn Z_{k,i} Z_{k,j}$. For simplicity, here we work under the assumption that $\sigma^2 = 1$. 
\begin{enumerate}[(a)]
    \item By direct expansion, we can write the LHS as
    \begin{align*}
        \EEX [\bv^\top \A_p \bw]^2 &= \frac{1}{n^2} \sum_{1\le i\ne j\le p} \sum_{ 1\le i' \neq j'\le p} \sum_{ 1\le k, k'\le n} \EEX \big[Z_{k,i} Z_{k,j} Z_{k',i'} Z_{k',j'}\big] v_i w_j v_{i'} w_{j'} \\
        &= \frac{1}{n} \sum_{1\le i\ne j\le n} \big[v_i^2 w_j^2 + v_i v_j w_i w_j \big] \le \frac{2\|\bv\|^2 \|\bw\|^2}{n}.
    \end{align*}
    The equality in the second line follows as the inner expectation is zero unless $k = k'$ and $\{i,j\} = \{i',j'\}$.

    \item Similarly, we expand the LHS as
    \begin{align*}
        \EEX \Big[\sum_{j\ne i} \A_p(i,j) v_j \Big]^2 &= \frac{1}{n^2} \sum_{1\le j \neq i, j' \neq i\le p} \sum_{1\le k, k' \le n} \EEX \big[Z_{k,i} Z_{k,j} Z_{k',i} Z_{k',j'}\big] v_j v_{j'} \lesssim  \frac{1}{n} \sum_{j \neq i} v_j^2.
    \end{align*}
    The inequality follows as the inner expectation is zero unless $k = k'$ and $j=j'$. The inner expectation is always bounded by a universal constant, due to the sub-Gaussian assumption.
    Now, the claim for $\mathfrak{F}_i=(\mathbf{A}{\bm \beta}^\star)_i$ follows on recalling that $|\beta_j^\star| \le 1$ for all $1\le j\le p$.

    \item We expand the LHS as
    \begin{align*}
        &\EEX \Big[\sum_{j \neq i} \A_p(i,j) v_j \Big]^4 \\
        =& \frac{1}{n^4} \sum_{1\le j_1, \ldots, j_4 \neq i\le p} \sum_{1\le k_1, \ldots, k_4 \le n} \EEX \big[Z_{k_1,i} Z_{k_1,j_1} \ldots Z_{k_4,i} Z_{k_4,j_4} \big] v_{j_1} v_{j_2} v_{j_3} v_{j_4} \\
        \lesssim& \frac{1}{n^2} \sum_{j_1, j_3 \neq i} v_{j_1}^2 v_{j_3}^2 \le 
        \frac{\|\bv\|^4}{n^2}.
    \end{align*}
    The second line follows as the inner expectation is zero unless $k_1 = k_2, k_3 = k_4$ and $j_1 = j_2, j_3 = j_4$, upto a permutation of the subscripts. The claim for $\mathfrak{F}_i$ follows as $|\beta_j^\star| \lesssim 1$ for all $j$.

    \item 
    By an application of the Cauchy-Schwartz inequality, we have
    \begin{align*}
        \Big[\sum_{j \neq i} \A_p(i,j) \mathfrak{F}_j v_j \Big]^4 &\le 
       \Big[\sum_{j \neq i} \A_p(i,j) \mathfrak{F}_j v_j \Big]^2\left( \sumjn \mathfrak{F}_j^2\right) \left(\max_{1\le i\le p} \sumjn \A_p(i,j)^2\right).
    \end{align*}
    Recall from \cref{lem:alpha_n} that $\max_{1\le i\le p} \sumjn \A_p(i,j)^2=\BOPX(p/n)$ and from \cref{lem:random design properties}, part (b) that $\sumjn \mathfrak{F}_j^2=\BOPX(p^2/n)$. As a result, the above display implies 
     \begin{align*}
        \Big[\sum_{j \neq i} \A_p(i,j) \mathfrak{F}_j v_j \Big]^4 &\le 
       \Big[\sum_{j \neq i} \A_p(i,j) \mathfrak{F}_j v_j \Big]^2 \BOPX(p^3/n^2).
    \end{align*}
    It therefore suffices to show that 
    \begin{align}\label{eq:newcall4}
    \EE_{\bx} \sumin\Big[\sumjn \A_p(i,j) \mathfrak{F}_j v_j\Big]^2  \lesssim \frac{p^3}{n^2}.
    \end{align}
    By expansion of the LHS of \eqref{eq:newcall}, we get
    $$\frac{1}{p^4}\sumin \sum_{\substack{j_1 \neq i, ~\ell_1 \neq j_1, \\ j_2 \neq i, ~\ell_2 \neq j_2}} \sum_{k_1,k_2,s_1,s_2} \EE_{\bx} \Big[Z_{k_1,i} Z_{k_1,j_1} Z_{s_1,j_1} Z_{s_1,\ell_1} Z_{k_2,i} Z_{k_2,j_2} Z_{s_2,j_2} \Big] \beta_{\ell_1}^\star \beta_{\ell_2}^\star v_{j_1} v_{j_2}.$$
    Here, the indices take values $1\le j_1,j_2, \ell_1, \ell_2 \le p$ and $1\le k_1,s_1,k_2,s_2 \le n$. Using $Z_{k,i}$s have mean 0 and variance 1, the inner expectation is zero unless the indices $(k_1,s_1,k_2,s_2)$ consist of one or two distinct values, e.g. $k_1=k_2, ~s_1=s_2$. For simplicity, consider indices such that $k_1=k_2\neq s_1=s_2$. Then, the inner expectation is nonzero only if $j_1 = j_2, ~ \ell_1 = \ell_2$. Summing over these indices and bounding $\lVert \bbst\rVert_{\infty}\le 1$ and $\lVert\mathbf{v}\rVert_{\infty}\le 1$ gives
    $$\frac{1}{n^2}\sumin \sum_{j_1 \neq \ell_1, k_1 \ne s_1} \EE \Big[Z_{k,i}^2 Z_{k,j_1}^2 Z_{s,j}^2 Z_{s,\ell}^2 \Big] (\beta^{\star}_{\ell_1})^2 v_{j_1}^2 \le \frac{p^3}{n^2}.$$
    Other choices of $(k_1,k_2,s_1,s_2)$ also give the same or smaller bound. We omit the details for brevity. This completes the proof of \eqref{eq:newcall4}. 
\end{enumerate}
\end{proof}

\begin{proof}[Proof of \cref{lem:lincomb expectation random design}]
    We first show \eqref{eq:lincomb moment bound}. By a Taylor expansion, we can write
    \begin{align}\label{eq:phi taylor}
        \sumjn \gamma_j \Phi (d_0 \beta_j^\star - \mathfrak{F}_j) &= \sumjn \gamma_j \Phi (d_0 \beta_j^\star) - \sumjn \gamma_j \mathfrak{F}_j {\Phi}' (d_0 \beta_j^\star) + O(\sumjn |\gamma_j| \mathfrak{F}_j^2).
    \end{align}
    The second moment of the second term is $O(\|\bgamma\|^2 p /n)$ by part (a) of \cref{lem:random design properties}, since $$\sumjn \gamma_j \mathfrak{F}_j {\Phi}' (d_0 \beta_j^\star) = \sumij \A_p(i,j) \big[\gamma_j \Phi' (d_0 \beta_j^\star)\big] \beta_i^\star.$$ Also, the second moment of the third term is $O(\|\bgamma\|^2 p^3 /n^2)$ by Cauchy-Schwartz and part (c) of the same lemma. This verifies \eqref{eq:lincomb moment bound}. Finally, both conclusions in \eqref{eq:lincomb expectation} follow by taking $\Phi = \phi_a$ for each $a = 1,2$.
\end{proof}

\subsection{Proof of Lemmas used for Theorem \ref{cor:regression Bayesian truth}}
Now, we prove Lemmas \ref{lem:beta concentration}, \ref{lem:lincomb expectation Bayes}. In the following proofs, the hidden constants only depend on the fourth moment of $\mu^\star$.
\begin{proof}[Proof of \cref{lem:beta concentration}]
We will repeatedly use the symmetry of the law $\mu^{\star}$ around $0$, and the assumption that its fourth (and lower) moments are bounded.
\begin{enumerate}[(a)]
    \item This is immediate by computing the second moment. The conclusion about $\mathfrak{F}_i=\sumjn \A_p(i,j)\beta^{\star}_j$ follows on taking $\gamma_j=\A_p(i,j)$ and $\Phi(x)=x$.

    \item By direct expansion, we can write the LHS as
    \begin{align}
        &\;\;\;\;\EEB \Big[\sumjn \gamma_j \mathfrak{F}_j \Phi(\beta^\star_j)\Big]^2 \notag \\ &= \sum_{1\le i,j \le p} \sum_{1 \le i',j' \le p} \A_p(i,j) \A_p(i',j') \gamma_j \gamma_{j'} \EEB \Big[\beta_i^\star \beta_{i'}^\star \Phi(\beta_j^\star) \Phi(\beta_{j'}^\star) \Big] \notag \\
        & = \EEB[(\beta_1^{\star})^2]\sum_{1\le i, j, j' \le p} \A_p(i,j) \A_p(i,j') \gamma_j \gamma_{j'} \EEB[\Phi(\beta_j^{\star})\Phi(\beta_{j'}^{\star})] \notag \\ &\quad + (\EEB[\beta_1^{\star}\Phi(\beta_1^{\star})])^2\sum_{1\le i,j \le p} \A_p(i,j)^2 \gamma_i \gamma_j. \label{eq:quadratic form expansion}
    \end{align}
    \eqref{eq:quadratic form expansion} follows as the inner expectation is zero unless (i) $i=i'$ or (ii) $i=j'\neq j=i'$. The first term above can be further split into two terms corresponding to the cases $j=j'$ and $j\ne j'$. With that in mind the first term in \eqref{eq:quadratic form expansion} reduces to
    \begin{small}
    \begin{align*}
    &\;\;\;\EEB[(\beta_1^{\star})^2](\EE[\Phi^2(\beta_1^{\star})])\sumij \A_p(i,j)^2\gamma_j^2+\EEB[(\beta_1^{\star})^2](\EEB\Phi(\beta_1^{\star}))^2 \sum_{1\le i,j,j'\le p, j\ne j'} \A_p(i,j)\A_p(i,j')\gamma_j\gamma_j' \\ &\le \EEB[(\beta_1^{\star})^2](\EEB[\Phi^2(\beta_1^{\star})])\an \lVert \bgamma\rVert^2+\EEB[(\beta_1^{\star})^2](\EEB\Phi(\beta_1^{\star}))^2\bigg|\sum_{1\le j\ne j'\le p} (\A_p^2)(j,j')\gamma_j\gamma_j'\bigg| \\
    &\lesssim \alpha_p\|\bgamma\|^2+(\EEB\Phi(\beta_1^{\star}))^2\bgamma^\top \A_p^2 \bgamma\lesssim \lVert \bgamma\rVert^2. 
    \end{align*}
    \end{small}
    Here we have used $\|\A_p\|=O(1)$ to conclude that $\max(\alpha_p,\|\A_p^2\|)\lesssim 1$.
    The second term in \eqref{eq:quadratic form expansion} is bounded above by $\an \|\bgamma\|^2$ up to constants due to \cref{eq:b_n matrix}.

    The statement under $\EEB \Phi(\beta_1^{\star}) = 0$ is immediate from the display above, since the second term in the bound on the last line is zero.

    \item By direct expansion, 
    \begin{align*}
        \EEB \Big[\sumjn \gamma_j \Phi(\beta_j^\star) \Big]^4 &= \sum_{j \neq j'} \gamma_j^2 \gamma_{j'}^2 \EE[\Phi(\beta_j^\star)^2] \EE[\Phi(\beta_{j'}^\star)^2] + \sumjn \gamma_j^4 \EE[\Phi(\beta_j^\star)^4] \le \Big(\sumjn \gamma_j^2\Big)^2.
    \end{align*}
    In particular, by taking $\gamma_j := \A_p(i,j)$ and $\Phi(\cdot)$ to be the identity function alongside the bounded fourth moment assumption of $\mu^\star$, we have $\EEB \mathfrak{F}_i^4 \lesssim \alpha_p^2$ for all $i$. The last claim of part (c) follows by a Cauchy-Schwartz:
    $$\EEB \Big[\sumjn \gamma_j \mathfrak{F}_j^2 \Big]^2 \le \|\bgamma\|^2 \sumjn \EEB \mathfrak{F}_j^4 \lesssim p\alpha_p^2 \|\bgamma\|^2.$$

    \item     
   For part (d) we use the generalized Efron-Stein inequality from \cite[Theorem 2]{Boucheron2005}. To this effect, sample $\tilde{\beta}_1^{\star},\ldots ,\tilde{\beta}_p^{\star} \overset{i.i.d}{\sim}\mu^{\star}$, independent of $\bbst$. Setting $S_p:=\sum_{j,k=1}^p \gamma_j\A_p(j,k)\beta_k^{\star}\Phi(\beta_j^{\star})$ and
    \begin{align*}
        S_p^{(i)}:=\sum_{j\neq i,k\neq i}\gamma_j\A_p(j,k)\beta_k^{\star}\Phi(\beta_j^{\star}) + \tilde{\beta}_i^{\star}\sumjn \gamma_j\A_p(j,i)\Phi(\beta_j^{\star})+\gamma_i\Phi(\tilde{\beta}_i^{\star})\sumkn \A_p(i,k)\beta_k^{\star},
    \end{align*}
    we note that 
    \begin{align*}
        |S_p-S_p^{(i)}|&\le \big|(\tilde{\beta}_i^{\star}-\beta_i^{\star})\sumjn \gamma_j\A_p(j,i)\Phi(\beta_j^{\star})\big|+\big|\gamma_i(\Phi(\tilde{\beta}_i^{\star})-\Phi(\beta_i^{\star}))\sumkn \A_p(i,k)\beta_k^{\star}\big| \\ &\le \big|(\tilde{\beta}_i^{\star}-\beta_i^{\star}) \sumjn \gamma_j\A_p(j,i)\Phi(\beta_j^{\star})\big|+2\big|\gamma_i\mathfrak{F}_i\big|.
    \end{align*}
    Therefore, 
    \begin{align*}
        \Big(\sumin (S_p-S_p^{(i)})^2 \Big)^2\lesssim \left(\sumin (\tilde{\beta}_i^{\star}-\beta_i^{\star})^2 \left(\sumjn \gamma_j\A_p(j,i) \Phi(\beta_j^{\star}) \right)^2\right)^2+\left(\sumin \gamma_i^2\mathfrak{F}_i^2\right)^2.
    \end{align*}
    First we note that
    \begin{align}\label{eq:obs1}
    \EEB\left(\sumin \gamma_i^2 \mathfrak{F}_i^2\right)^2 = \EEB\sumij \gamma_i^2\gamma_j^2 \mathfrak{F}_i^2\mathfrak{F}_j^2\lesssim \alpha_p^2\lVert \bgamma\rVert^4, 
    \end{align}
    where the last line follows from the Cauchy-Schwartz inequality to note that  $\EEB[\mathfrak{F}_i^2\mathfrak{F}_j^2]\le \max_{i=1}^p \EEB \mathfrak{F}_i^4\lesssim \alpha_p^2$ by part (c). 

    Next define $\mathfrak{g}_{i}:=\sumjn \gamma_j\A_p(i,j)\Phi(\beta_j^{\star})$ and $\mathfrak{g}_i^{(k)}:=\sum_{j=1,j\neq k}^p \gamma_j\A_p(i,j)\Phi(\beta_j^{\star})$. As $\mathfrak{g}_i=\mathfrak{g}_i^{(k)}+\gamma_k\A_p(i,k)\Phi(\beta_k^{\star})$, we get 
    
    \begin{align*}
    \mathfrak{g}_i^2\lesssim (\mathfrak{g}_i^{(k)})^2+\gamma_k^2\A_p(k,i)^2\,,\, (\mathfrak{g}_i^{(k)})^2 \lesssim \mathfrak{g}_i^2+\gamma_k^2 \A_p(k,i)^2.
    \end{align*}
    Using this gives
    \begin{align}\label{eq:gi1}
    \mathfrak{g}_i^2+(\mathfrak{g}_i^{(k)})^2 \lesssim \mathfrak{g}_i^2+\gamma_k^2\A_p(i,k)^2\le \lVert \bgamma\rVert^2\an,
    \end{align}
    where the last step uses the Cauchy-Schwartz inequality. Moreover, for $i\neq k$, we have: 
    \begin{align}\label{eq:gi2}
        \mathfrak{g}_i^2 \mathfrak{g}_k^2 &\le 2 \mathfrak{g}_i^2(\mathfrak{g}_k^{(i)})^2+2\mathfrak{g}_i^2 \gamma_i^2\A_p(i,k)^2\nonumber  \\ &\le 4(\mathfrak{g}_k^{(i)})^2\big((\mathfrak{g}_i^{(k)})^2+\gamma_k^2 \A_p(i,k)^2\big)+2\mathfrak{g}_i^2\gamma_i^2\A_p(i,k)^2 \nonumber \\ &\le  4(\mathfrak{g}_k^{(i)}\mathfrak{g}_i^{(k)})^2 + 4\lVert \boldsymbol{\gamma}\rVert^2\an\A_p(i,k)^2(\gamma_i^2+\gamma_k^2).
    \end{align}
    Here the last inequality uses \eqref{eq:gi1}. Similarly,
    \begin{align}\label{eq:gi3}
        (\mathfrak{g}_i^{(k)} \mathfrak{g}_k^{(i)})^2 \le 4\mathfrak{g}_i^2 \mathfrak{g}_k^2 + 4\lVert \boldsymbol{\gamma}\rVert^2 \an \A_p(i,k)^2 (\gamma_i^2+\gamma_k^2).
    \end{align}
    Finally observe that
    \begin{align}\label{eq:gi4}
        \sumin \mathfrak{g}_i^2=\sum_{j_1,j_2=1}^p \gamma_{j_1}\gamma_{j_2}\Phi(\beta_{j_1}^{\star})\Phi(\beta_{j_2}^{\star})\sumin \A_p(i,j_1)\A_p(i,j_2) \lesssim \lVert \bgamma\rVert^2,
    \end{align}
   since $\|\A_p^2\|\le \|\A_p\|^2 \le 1$. We will now use the above observations to complete the proof. Note that:
    \begin{align}\label{eq:obs4}
        &\;\;\;\;\EEB\left(\sumin (\tilde{\beta}_i^{\star}-\beta_i^{\star})^2\mathfrak{g}_i^2\right)^2 \nonumber \\ &=\EEB\sumin (\tilde{\beta}_i^{\star}-\beta_i^{\star})^4 \mathfrak{g}_i^4+\EEB\sum_{i\ne k} (\tilde{\beta}_i^{\star}-\beta_i^{\star})^2 (\tilde{\beta}_k^{\star}-\beta_k^{\star})^2 \mathfrak{g}_i^2\mathfrak{g}_k^2 \nonumber \\ &\lesssim \EEB\sumin (\tilde{\beta}_i^{\star}-\beta_i^{\star})^4\mathfrak{g}_i^4+\EEB\sum_{i\ne k} (\tilde{\beta}_i^{\star}-\beta_i^{\star})^2 (\tilde{\beta}_k^{\star}-\beta_k^{\star})^2(\mathfrak{g}_i^{(k)})^2(\mathfrak{g}_k^{(i)})^2 \nonumber\\ &\quad +\lVert \boldsymbol{\gamma}\rVert^2 \an \EEB\sumik \A_p(i,k)^2 (\tilde{\beta}_i^{\star}-\beta_i^{\star})^2 (\tilde{\beta}_k^{\star}-\beta_k^{\star})^2(\gamma_i^2+\gamma_k^2).
    \end{align}
    The last inequality here follows from \eqref{eq:gi2}. Next we observe that $\tilde{\beta}_i^{\star}-\beta_i^{\star}$ and $\mathfrak{g}_i$ are independent. Also $(\tilde{\beta}_i^{\star}-\beta_i^{\star}) (\tilde{\beta}_k^{\star}-\beta_k^{\star})$ and $\mathfrak{g}_i^{(k)}\mathfrak{g}_k^{(i)}$ are independent, for $i\neq k$. Therefore, as $\EEB(\tilde{\beta}_i^{\star}-\beta_i^{\star})^4<\infty$ and $\EEB[(\tilde{\beta}_i^{\star}-\beta_i^{\star})^2 (\tilde{\beta}_k^{\star}-\beta_k^{\star})^2]<\infty$ uniformly in $i,k$, \eqref{eq:obs4} yields
    \begin{align*}
        &\;\;\;\;\EEB\left(\sumin (\tilde{\beta}_i^{\star}-\beta_i^{\star})^2\mathfrak{g}_i^2\right)^2 \\ &\lesssim \EEB\sumin \mathfrak{g}_i^4 + \EEB\sum_{i\neq k} (\mathfrak{g}_i^{(k)})^2 (\mathfrak{g}_k^{(i)})^2+\lVert \boldsymbol{\gamma}\rVert^2\an\sum_{i,k=1}^p \A_p(i,k)^2(\gamma_i^2+\gamma_k^2) \\ &\le \EEB\sumin \mathfrak{g}_i^4 + \EEB\sum_{i\neq k} (\mathfrak{g}_i^{(k)})^2 (\mathfrak{g}_k^{(i)})^2+2\lVert \boldsymbol{\gamma}\rVert^4\alpha_p^2.
    \end{align*}
    Using \eqref{eq:gi3}, we get:
    \begin{align*}
        \EEB\left(\sumin (\tilde{\beta}_i^{\star}-\beta_i^{\star})^2\mathfrak{g}_i^2\right)^2 & \lesssim \EEB\sumin \mathfrak{g}_i^4+\EEB\sum_{i\ne k}\mathfrak{g}_i^2\mathfrak{g}_k^2+\lVert\boldsymbol{\gamma}\rVert^2\an \sum_{i,k=1}^p \A_p(i,k)^2(\gamma_i^2+\gamma_k^2)\\ &\lesssim \EEB\sumin \mathfrak{g}_i^4+\EEB\sum_{i\ne k}\mathfrak{g}_i^2\mathfrak{g}_k^2+\lVert \boldsymbol{\gamma}\rVert^4\alpha_p^2 \\ &=\EEB\left(\sumin \mathfrak{g}_i^2\right)^2+\lVert\boldsymbol{\gamma}\rVert^4\alpha_p^2 \le (1+\alpha_p^2)\lVert\boldsymbol{\gamma}\rVert^4.
    \end{align*}
    
    The last inequality here follows from \eqref{eq:gi4}. Next, by applying the generalized Efron-Stein inequality (see \cite[Theorem 2]{Boucheron2005}), coupled with \eqref{eq:obs1} and \eqref{eq:obs4}, we get: 
    \begin{align*}
        \EEB S_p^4\lesssim \EE\left(\sumin (S_p-S_p^{(i)})^2\right)^2\lesssim (1+\alpha_p^2)\lVert \bgamma\rVert^4.
    \end{align*}
 This completes the proof.   
\end{enumerate}
\end{proof}

\begin{proof}[Proof of \cref{lem:lincomb expectation Bayes}]
    We first show \eqref{eq:lincomb moment bound general}, using the same Taylor expansion \eqref{eq:phi taylor} from the proof of \cref{lem:lincomb expectation random design}, which gives 
    \begin{align*}
        \sumjn \gamma_j \Phi (d_0 \beta_j^\star - \mathfrak{F}_j) &= \sumjn \gamma_j \Phi (d_0 \beta_j^\star) - \sumjn \gamma_j \mathfrak{F}_j {\Phi}' (d_0 \beta_j^\star) + O(\sumjn |\gamma_j| \mathfrak{F}_j^2).
    \end{align*}
    We bound each summand above. 
    The second moments of the first and second term are both $O(\|\bgamma\|^2)$ by parts (a) and (b) of \cref{lem:beta concentration}. Also, the second moment of the third term is $p \alpha_p^2 \|\bgamma\|^2$ by part (c) of \cref{lem:beta concentration}. Now, \eqref{eq:lincomb moment bound general} follows as we assume the SMF condition $\sqrt{p}\an = o(1)$.

    Now, both conclusions in \eqref{eq:lincomb expectation Bayes} follow by taking $\Phi = \phi_1$ and $\phi_2 - \upsilon$ respectively, where $\phi_1(\cdot)=\EE\psi_0'(\mcn(\cdot,d_0^{\star}))$ and $\phi_2(\cdot)=\EE\psi_0''(\mcn(\cdot,d_0^{\star}))$ are defined as in \cref{def:allnot}.
    The bounds in \eqref{eq:remeq} directly follow by summing up the bounds in \cref{lem:lincombound} (parts (a), (b) with $U = \upsilon$) and \eqref{eq:lincomb expectation Bayes}. Note that the RHS of parts (a) and (b) in \cref{lem:lincombound} is $O(\|\bgamma\|^2)$ since we have $\mathfrak{E}_1 = O(1)$ by \cref{assn:homogeneous diagonals}, and $\|\Sigman\| = O(1)$ from \eqref{eq:citelater}.
\end{proof}

\section{Proof of \cref{prop:regret} and auxiliary lemmas}\label{sec:pfauxlem}
In this section, we prove \cref{prop:regret} and auxiliary lemmas that were used in previous proofs, namely Lemmas \ref{lem:clt marginalize}, \ref{lem:exponential tilt expectation}, \ref{lem:quadratic tilt moment}. First, we prove \cref{prop:regret} using the following moment bound.
\begin{lemma}[Lemma 3.2(b) in \cite{lee2025rfim}]\label{lem:m-n contraction}
    Suppose \eqref{assn:ht} holds. For each $1\le i \le p$, define the local averages under the posterior as
    $m_i := \sumjn \A_p(i,j) \beta_j$. Also recall the definition of $s_i$ from \eqref{eq:fpeq}. Then, we have $$\EE_{\nu_{\by,\bx}} \left[\sumin (m_i - s_i)^4\right] \lesssim p\alpha_p^2,$$
    where the implied constant depends only on $\rho$ (in \cref{assn:ht}). %
\end{lemma}

\begin{proof}[Proof of \cref{prop:regret}]
    We only prove \eqref{eq:optmf}, as the regret bound directly follows by the bias-variance decomposition, i.e.,
    $$\EE_{\nu_{\by,\bx}}(\bq^{\top}\bbeta-\bq^{\top}\bu)^2 = \EE_{\nu_{\by,\bx}}(\bq^{\top}\bbeta-\delta(\by,\bx))^2+(\delta(\by,\bx)-\bq^{\top}\bu)^2.$$
    We first introduce additional notations from \cite{lee2025rfim}.
    For non-negative integers $\ell \ge 0$ and $1\le j\le n$, recursively define $\bq^{(\ell)}$ by $\bq^{(0)} = \bq$ and $\bq^{(\ell)} := \C_p \bq^{(\ell-1)},$
    where $\C_p$ is a $p \times p$ matrix defined as $ \C_p(i,j) := \A_p(i,j) \psi_j''( s_j + c_j)$.
    Define 
    $$S_{(\ell)}:= \EE_{\nu_{\by,\bx}}\big[(\bq^{(\ell)})^{\top} \bbeta\big] - (\bq^{(\ell)})^{\top}\bu = \sumin q^{(\ell)}_i(\EE_{\nu_{\by,\bx}} \psi_i'(m_i+c_i) - u_i),$$
    for $\ell\ge 0$. By a standard Taylor expansion and the fact that $u_i=\psi_i'(s_i+c_i)$ (see \eqref{eq:fpeq}), we get:
    \begin{align*}
    S_{(\ell)} & = \EE_{\nu_{\by,\bx}}\left[\sumin q_i^{(\ell)}(m_i-s_i)\psi_i''(s_i+c_i)\right] + \frac{1}{2}\EE_{\nu_{\by,\bx}}\left[\sumin q_i^{(\ell)}(m_i-s_i)^2\psi_i'''(\xi_i+c_i)\right] \\ &=\EE_{\nu_{\by,\bx}}\left[\sumij q_i^{(\ell)}\A_p(i,j)(\beta_j-u_j)\psi_i''(s_i+c_i)\right] + \frac{1}{2}\EE_{\nu_{\by,\bx}}\left[\sumin q_i^{(\ell)}(m_i-s_i)^2\psi_i'''(\xi_i+c_i)\right] \\ &=\EE_{\nu_{\by,\bx}}\left[\sumin q_i^{(\ell+1)}(\beta_i-u_i)\right] + \frac{1}{2}\EE_{\nu_{\by,\bx}}\left[\sumin q_i^{(\ell)}(m_i-s_i)^2\psi_i'''(\xi_i+c_i)\right] \\ &=S_{(\ell+1)} + \frac{1}{2}\EE_{\nu_{\by,\bx}}\left[\sumin q_i^{(\ell)}(m_i-s_i)^2\psi_i'''(\xi_i+c_i)\right], 
    \end{align*}
    for some $\{\xi_i\}_{1\le i\le p}$ such that $\xi_i \in (m_i, s_i)$. 

    Note that under the high-temperature \cref{assn:ht}, we have $\|\C_p\| \le \|\A_p\| \le \rho < 1$. This implies that $$|S_{(\ell)}| \le \sqrt{p} \|\bq^{(\ell)}\| \le \sqrt{p} \|\mathbf{C}_p\|^\ell \|{\mathbf q}\| \le \sqrt{p}\|\A_p\|^\ell \le \sqrt{p}\rho^\ell  \stackrel{\ell\to\infty}{\to} 0.$$
    By a telescoping argument, we get
    \begin{align*}
        \bigg|\sumin q_i(\EE_{\nu_{\by,\bx}}[\beta_i]-u_i)\bigg|& \le \sum_{\ell=0}^{\infty} |S_{(\ell)}-S_{(\ell+1)}| \\ &\lesssim \sum_{\ell=0}^{\infty} \EE_{\nu_{\by,\bx}}\left[\sumin q_i^{(\ell)}(m_i-s_i)^2\right]\\ &\lesssim  \EE_{\nu_{\by,\bx}}\lVert \mathbf{m}-\mathbf{s}\rVert_4^2\sum_{\ell=0}^{\infty}\rho^{\ell}\lesssim \sqrt{p}\an. 
    \end{align*}
    Here, we have used the fact that $|\psi_i'''|$ is bounded, and \cref{lem:m-n contraction} alongside $\lVert \bq^{(\ell)}\rVert\le \rho^{\ell}$. This completes the proof of \eqref{eq:optmf}.

    \vspace{0.1in}

\end{proof}

Now, we prove the remaining lemmas.

    \begin{proof}[Proof of \cref{lem:clt marginalize}]
    The proof is straightforward by using characteristic functions. For any $t$, $G_p \mid \bx, \bbst \xd \mcn(0, \varsigma^2)$ implies
    $$\tilde{G}_p = \tilde{G}_p(\bx, \bbst) := \EE\Big(e^{it G_p} \mid \bx, \bbst \Big) \xp e^{- t^2 \varsigma^2/2},$$
    so $\EE[|\tilde{G}_p - e^{- t^2 \varsigma^2/2}| \mid \bx] \to 0.$
    Using the tower property and the above $L^1$ convergence, we have
    \begin{align*}
        \EE\Big[e^{it (G_p+H_p) } \mid \bx \Big] &= \EE \Big[\EE \Big(e^{it G_p } \mid \bx, \bbst \Big) e^{it H_p} \mid \bx \Big] = \EE \Big[\tilde{G}_p e^{it H_p} \mid \bx \Big]\\
        &= e^{-t^2\varsigma^2/2} \EE [e^{it H_p} \mid \bx] + \EE\Big[(\tilde{G}_p- e^{- t^2 \varsigma^2/2}) e^{it H_p} \mid \bx \Big]
        \to 
        e^{-t^2(\vartheta^2+\varsigma^2)/2}.
    \end{align*}
    Hence, the claim holds.
\end{proof}

\begin{proof}[Proof of \cref{lem:exponential tilt expectation}]
It suffices to prove the claim under two cases: (i) $m = m_0$, (ii) $d = d_0$. Hence, we present the proof by considering each case separately.
\begin{enumerate}[(i)]
    \item Suppose $m = m_0$. Without the loss of generality, we may let $m = m_0 = 0$.
        Let $p_0(c)$ be the density of $C_0$, and note that $C$ has the same distribution as $\sqrt{\frac{d}{d_0}} C_0$. Then,
        \begin{align*}
            |\EE g(C) - \EE g(C_0)| &= \bigg|\EE g\left(\sqrt{\frac{d}{d_0}} C_0\right) - \EE g(C_0)\bigg| \\
            &\le \int_{-\infty}^{\infty} \lVert g' \rVert_{\infty} \left|\sqrt{\frac{d}{d_0}} - 1 \right| |c| p_0(c) dc \\
            &\le \lVert g' \rVert_{\infty} \EE|C_0| \frac{|d - d_0|}{\sqrt{d_0} (\sqrt{d} + \sqrt{d_0})} \lesssim\lVert g' \rVert_{\infty} |d - d_0|.
        \end{align*}

    \item Now, suppose $d = d_0$.
    Noting that $C$ has the same distribution as $C_0 + m - m_0$, we have
    \begin{align*}
        \left|\EE g(C) - \EE g(C_0) \right| &\le \int_{-\infty}^\infty |g(c + m - m_0) - g(c)| p_0(c) dc \\
        &\le \lVert g' \rVert_{\infty} |m - m_0| \int_{-\infty}^\infty p_0(c) dc = \lVert g' \rVert_{\infty} |m - m_0|.
    \end{align*}
\end{enumerate}
    \end{proof}

\begin{proof}[Proof of \cref{lem:quadratic tilt moment}]
    For $0\le i\le p$, let $\delta_i:= \left|e^{\frac{d_0 - d_i}{2}} - 1 \right|,$ and let $p_{i,\theta}(z) := e^{\theta z - \frac{d_i z^2}{2}}$ and $M_i(\theta) := \int_{[-1,1]}  p_{i,\theta}(z) \mu(dz) $. %
    
    We first claim that for any measurable function $f:[-1,1]\to [-1,1]$, 
    \begin{equation}\label{eq:integral difference}
        \left|\int_{[-1,1]} f(z) p_{i,\theta}(z) \mu(dz) - \int_{[-1,1]} f(z) p_{0,\theta}(z) \mu(dz)\right| \le M_0(\theta) \delta_i.
    \end{equation}
    The proof of the claim is deferred to the end. First, let us assume the claim and complete the proof of the original lemma. 

    Towards this direction, note that by taking $f\equiv 1$, we have 
    \begin{equation}\label{eq:m_i - m_0 bound}
        |M_i(\theta) - M_0(\theta)| \le M_0(\theta) \delta_i.
    \end{equation}

   Further, for any bounded function $f$, we have
    \begin{align*}
        &|\EE_{\mu_{i,\theta}} f(Z) - \EE_{\mu_{0,\theta}} f(Z)| \\
        =& \left|\frac{\int_{[-1,1]} f(z) p_{i,\theta}(z) \mu(dz)}{M_i(\theta)} - \frac{\int_{[-1,1]} f(z) p_{0,\theta}(z) \mu(dz)}{M_0(\theta)} \right| \\
        \le& \frac{|M_0(\theta) - M_i(\theta)|}{M_i(\theta) M_0(\theta)} \int_{[-1,1]} |f(z)| p_{i,\theta}(z) \mu(dz)+ \frac{1}{M_0(\theta)}\bigg|\int_{[-1,1]} f(z) (p_{i,\theta}(z) - p_{0,\theta}(z)) \mu(dz)\bigg|  \\
        \le & \frac{\delta_i}{M_i(\theta)}\int_{[-1,1]}  p_{i,\theta}(z) \mu(dz) + \delta_i = 2\delta_i.
    \end{align*}
    The inequality in the last line follows by applying \eqref{eq:m_i - m_0 bound} and \eqref{eq:integral difference} respectively. 
    
    \noindent The bound on $|\psi_i'-\psi_0'|$ follows now by taking $f(z)=z$. For the bound on $|\psi_i''-\psi_0''|$, we choose $f(z) = z^2$ and note that
    \begin{align*}
        |\psi_i''(\theta)-\psi_0''(\theta)| =&|\Var_{\mu_{i,\theta}}(Z) - \Var_{\mu_{0,\theta}}(Z)| \\
        \le& \big|\EE_{\mu_{i,\theta}} [Z^2] - \EE_{\mu_{0,\theta}} [Z^2]\big| + \big|\EE_{\mu_{i,\theta}}[Z] - \EE_{\mu_{0,\theta}}[Z]\big|~ \big|\EE_{\mu_{i,\theta}}[Z] + \EE_{\mu_{0,\theta}}[Z]\big| \le 6\delta_i.
    \end{align*}
    Now, the proof is complete by noting that
    $\delta_i = \left|e^{\frac{d_0 - d_i}{2}} - 1 \right| \lesssim |d_i - d_0|$ which holds under \cref{assn:homogeneous diagonals}.
    
    \vspace{0.1in}
    
    \emph{Proof of \eqref{eq:integral difference}.} This follows from a standard change of measure argument:
    \begin{align*}
        &\left|\int_{[-1,1]} f(z) p_{i,\theta}(z) \mu(dz) - \int_{[-1,1]} f(z) p_{0,\theta}(z) \mu(dz)\right| \\
        =& \left|\int_{[-1,1]} f(z) p_{0,\theta}(z) \left( e^{\frac{(d_0-d_i) z^2}{2}} -1 \right) \mu(dz) \right| \\
        \le& \lVert f \rVert_{\infty} \left(\int_{[-1,1]}p_{0,\theta}(z) \mu(dz)\right) \sup_{z \in [-1,1]} \left| e^{\frac{(d_0-d_i) z^2}{2}} -1 \right| \\
        = & \lVert f\rVert_{\infty} M_0(\theta) \delta_i.
    \end{align*}
    The last line follows from the definition of $M_0(\theta)$, and observing that the supremum is attained at $z = \pm 1$ regardless of the sign of $d_0 - d_i$.
    \end{proof}

\end{appendices}

\end{document}